\documentclass[reqno,a4paper,11pt]{amsart}
\usepackage[left=2cm,right=2cm,top=2cm,bottom=2cm,bindingoffset=0cm]{geometry}
\usepackage{amsmath}
\usepackage{amssymb}
\usepackage{amsthm}
\usepackage{amsbsy}
\usepackage{enumitem}
\usepackage{hyperref}
\usepackage[foot]{amsaddr}

\numberwithin{equation}{section}

\newcommand{\R}{\mathbb{R}}
\newcommand{\N}{\mathbb{N}}

\renewcommand{\div}{\operatorname{div}}
\newcommand{\tr}{\operatorname{tr}}

\newcommand{\wto}{\rightharpoonup}
\newcommand{\mc}{\mathcal}
\renewcommand{\to}{\rightarrow}
\renewcommand{\d}{\,\mathrm{d}}
      
\newcommand{\norm}[1]{\left\lVert#1\right\rVert}  
\newcommand{\scal}[2]{\langle #1,#2\rangle}
\DeclareMathOperator*{\supp}{supp}

\newtheorem{teo}{Theorem}[section]
\newtheorem{defin}[teo]{Definition}
\newtheorem{rem}[teo]{Remark}
\newtheorem{prop}[teo]{Proposition}
\newtheorem{lemma}[teo]{Lemma}
\newtheorem{cor}[teo]{Corollary}

\begin{document}
	
	\author[G. Lazzaroni]{Giuliano Lazzaroni}\author[R. Molinarolo]{Riccardo Molinarolo}\author[F. Riva]{Filippo Riva$^*$}\author[F. Solombrino]{Francesco Solombrino}
	\address[G.\ Lazzaroni]{Dipartimento di Matematica e Informatica ``Ulisse Dini'',
		Universit\`a degli Studi di Firenze, Viale Morgagni 67/a, 50134 Firenze, Italy}
	\email{giuliano.lazzaroni@unifi.it}
	\address[R.\ Molinarolo and F.\ Solombrino]{Dipartimento di Matematica e Applicazioni ``Renato Caccioppoli'',
		Universit\`a degli Studi di Napoli Federico II, Via Cintia, Monte S.\ Angelo,
		80126 Napoli, Italy}
	\email{riccardo.molinarolo@unina.it}
	\email{francesco.solombrino@unina.it}
	\address[F.\ Riva]{Dipartimento di Matematica ``Felice Casorati'', Universit\`a degli Studi di Pavia, Via Ferrata, 5, 27100, Pavia, Italy}
	\email{filippo.riva@unipv.it}
	\thanks{$^*$ Corresponding Author}
	
	\title[Wave equation on moving domains and dynamic debonding]{On the wave equation on moving domains: regularity, energy balance and application to dynamic debonding}
	
	\begin{abstract}
		We revisit some issues about existence and regularity for the wave equation in noncylindrical domains. Using a method of diffeomorphisms, we show how, through increasing regularity assumptions, the existence of weak solutions, their improved regularity and an energy balance can be derived. As an application, we give  a rigorous definition of dynamic energy release rate density for some problems of debonding, and we formulate a proper notion of solution for such problems. We discuss the consistence of such formulation with previous ones, given in literature for particular cases.
	\end{abstract}
	
	\maketitle
	
	
	\bigskip
	
	
	\tableofcontents
	
	\section*{Introduction}
In this paper we revisit some issues about existence and regularity for the wave equation in noncylindrical domains, i.e., in domains evolving in time. Our main motivation comes from an elastodynamic model for a thin film, initially glued onto a rigid substrate and progressively debonded by applying an external loading. As a result of this process, the debonded region is deformed according to the law of elastodynamics; moreover, its oscillations influence the evolution of the debonding front, i.e., the interface between the debonded region and the part of film still constrained onto the substrate. It is then natural to parametrize the debonded region by means of a time-dependent, growing domain, where the (transverse component of the) displacement satisfies the classical wave equation.
\par
On the other hand, the evolution of the domain is also an unknown of the model
and is governed by the physical principle of stability of the internal energy (kinetic and potential) of the body.
Rigorously writing the precise form of this energetic criterion requires some technical work,
in particular to characterize the energy release rate, which measures, loosely speaking, the amount of energy dissipated during an infinitesimal growth of the debonded region.
In fact, in literature the energy release rate and the propagation criterion were identified only in special cases, such as the one-dimensional setting \cite{DMLazNar, RivNar} and the case with radial solutions \cite{LazMolSol}, where some explicit formulas can be used.
\par
The main scope of this paper is to define the energy release rate for the dynamic debonding problem in a general setting, removing restrictive assumptions on the shape of the growing domains; this will lead to the flow rule governing the evolution of the domain. We show an integral formula for the dynamic energy release rate naturally arising from the energy balance for the wave equation and extending what previously found in special cases. To obtain this, we have to revisit the problem of the wave equation in time-dependent domains (with initial and boundary conditions) and provide a set of assumptions ensuring existence, uniqueness and regularity.
\par
In literature there are several results on the wave equation in noncylindrical domains, obtained with various methods
and under different assumptions on the evolution of the domains. In \cite{BernBonfLut} an abstract formulation is proposed, as well as a regularization procedure for the  operators involved in the problem. In \cite{Zol} the author performs a Galerkin method combined with a suitable penalization on the boundary. In \cite{Coop,BardCoop,DaPraZol,MaRuChu,Siko, ZhouSunLi} the authors employ changes of variables in order to recast the problem into  a fixed domain and then apply abstract results on hyperbolic equations via the semigroup theory. We also mention e.g.\  \cite{AlElSti,BernPozSav,CalvNovOrl,GianSav,HorRhe} and references therein for various approaches on different evolution equations (parabolic, Schr\"odinger, Navier-Stokes etc.) in noncylindrical domains.
\par 
In this paper we collect some results on existence, uniqueness and regularity for the wave equation and present them in a unitary perspective. Some results, already known in literature, are here provided with different proofs or under slightly different assumptions.
In Section \ref{sec:setting} we introduce a family $\Omega_t$ of Lipschitz domains in $\R^N$, depending on time $t\in[0,T]$. We define weak solutions of the wave equation in $\Omega_t$, complemented with natural initial conditions on $\Omega_0$ and a homogeneous Dirichlet condition on the boundary $\partial\Omega_t$. The results are then extended to nonhomogeneous Dirichlet conditions in Section \ref{sec:movingbdrycond}. We show an energy balance formula that holds true if the solution has a certain regularity in time and space
(Theorem \ref{teo energy balance}).
The most of the paper is then devoted to the rigorous proof of such regularity property.
\par
To this end, in Section~\ref{sec:fixeddomain} we follow the technique of changes of variables, assuming a certain time regularity of the family $\Omega_t$. 
Specifically, we require that there is a diffeomorphism $\Phi \colon[0,T] \times {\overline{\Omega}_0} \to \R^N$, such that $\Phi(t,\Omega_0)=\Omega_t$ for every $t$. This leads us to a hyperbolic problem in the fixed domain $\Omega_0$, with coefficients depending on time and space.
We compare two notions of solutions for such problem, called weak solutions and strong-weak solutions, respectively, and we prove they are equivalent (Proposition \ref{propequivweaksolv}).
Existence, uniqueness and regularity of solutions to the hyperbolic problem in the cylindrical domain $[0,T]\times\Omega_0$
are proved in Section~\ref{sec:hyperbolic}, by means of the Galerkin method (Theorem \ref{teoH2regularityv}).
The corresponding results in the noncylindrical domain follow under suitable assumptions on the diffeomorphisms
(Theorems \ref{teoequivalenceweaksol} and \ref{teo higher regularity u}).
\par
As a technical remark, in our results the regularity required on such diffeomorphisms is different if compared to the assumptions of other works in literature. For instance, in \cite{DaPraZol} and \cite{ZhouSunLi} the authors consider changes of variables of class $C^2([0,T];C^2({\overline{\Omega}_0}))$ and $C^3([0,T]\times{\overline{\Omega}_0})$, respectively; we instead require diffeomorphisms of class $C^{1,1}([0,T]\times{\overline{\Omega}_0})$ for existence of solutions (Sections~\ref{sec:fixeddomain} and \ref{sec:firstestimate}), while we need diffeomorphisms of class $C^{2,1}([0,T]\times{\overline{\Omega}_0})$ for uniqueness, regularity and energy balance (Sections~\ref{sec high reg} and \ref{sec:application}).
Moreover, for the regularity result we need to assume that $\Omega_0$ is convex or of class $C^2$. We stress in particular that, differently from previous works, we actually allow for a wider choice of the reference configuration $\Omega_0$, including some nonsmooth cases.
\par
In Section \ref{sec:application} we combine the results of Sections~\ref{sec:fixeddomain} and \ref{sec:hyperbolic}, thus providing a full statement of the energy balance (Theorem \ref{teo higher regularity u}). Moreover we show how our results may be applied to different settings:
the case of dimension one, extensively analysed in \cite{DMLazNar, DumMarCha,LBDM12,LazNar, LazNarQuas, LazNarInitiation, Riv, RivQuas, RivNar}; the case where each domain $\Omega_t$ is homothetic to $\Omega_0$;
the case where each domain is the sublevel set of a smooth function, which includes the case of radial solutions investigated in \cite{LazMolSol}.
\par
Our approach is deeply related with some works dealing with dynamic models for crack propagation in brittle fracture by means of the same ``method of diffeomorphisms'' \cite{Caponi, CaponiPhDThesis, CapLucTas, DMLuc}.
Indeed, also the formulation of dynamic fracture relies on the wave equation in a time-dependent domain, in this case a domain with a growing crack. However, since such a domain is not Lipschitz, our results do not apply to this case. On the other hand, in a dynamic fracture problem the domains only differ by a set of codimension one.
\par
The main similarity between the dynamic models for fracture and debonding is that
the wave equation is coupled with a flow rule governing the evolution of the domain. 
The latter arises from an energetic criterion \cite{Fre90} which may be stated as a maximum dissipation principle \cite{Lars}, 
or equivalently as a Griffith-type criterion involving the dynamic energy release rate.
In particular, it turns out that the flow rule implicitly depends on the solution of the wave equation.
\par
In Section \ref{sec:debonding} we show how the results of the previous sections allow us to rigorously define the energy release rate (and thus the propagation criterion) for the dynamic debonding model in a quite general setting.
More precisely, we introduce the \emph{density} of the dynamic energy release rate, which is obtained by a localization procedure, and a corresponding local version of Griffith's criterion, satisfied at each point of the debonding front.
\par
Our main achievement in this respect is a proper formulation of the coupled problem of dynamic debonding (wave equation together with local Griffith criterion) which includes the special cases analyzed in previous papers and may be applied without assuming a special form of the domains. We indeed show how the solutions found in the one-dimensional \cite{DMLazNar,RivNar} and radial \cite{LazMolSol} setting fulfil the formulation proposed here (Theorems~\ref{Thm coupled prob dim one} and \ref{Thm coupled prob radial}). The well posedness in the general framework still remains an open question, due to the high complexity arising from the coupling between the wave equation and Griffith criterion.

	\subsection*{Notation} Throughout the paper, the set of $M\times M$ matrices with real entries is denoted by $\R^{M\times M}$, and the subset of symmetric matrices is $\R^{M\times M}_{\rm sym}$. The identity matrix is denoted by $I$. For the transpose of a matrix $A$ we adopt the symbol $A^T$. The scalar product between two vectors $w,v\in\R^M$ is indicated by $w\cdot v$.\par  
	By $\dot f$ we mean the time derivative of a function $f=f(t,x)$. If $f$ is scalar-valued we write $\nabla f$ for its gradient with respect to spatial components, represented as a column vector. If $f$ is vector-valued we instead write $Df$ for its Jacobian matrix with respect to spatial components; as usual each row of $Df$ is the gradient of the corresponding component of $f$.\par
	Given an open set $E\subseteq \R^M$ with Lipschitz boundary $\partial E$, we denote by $\nu_E$ the outward unit normal to $E$. If $E\subseteq \R\times \R^N$, as often throughout the paper, we write $\nu_E^t$ and $\nu_E^x$ for the time- and space-component of the outward normal, respectively, namely $\nu_E=(\nu_E^t,\nu_E^x)\in \R\times \R^N$.\par 
	The integration with respect to the Lebesgue and to the $M$-dimensional Hausdorff measure is denoted respectively by $\d x$ (or $\d y$) and by $\d \mc H^M$. We adopt standard notations for Lebesgue and Sobolev spaces and for Bochner spaces. Given a Banach space $X$, we denote by $\langle w,v \rangle_X$ the duality product between $w\in X^*$ and $v\in X$. If $X=L^2(E)$, we identify it with its dual and, with a slight abuse of notation, we mean $\langle w,v \rangle_{L^2(E)}$ as the scalar product between $w$ and $v$. In the case of $X=H^1_0(E)$ we instead adopt the convention of rigged Hilbert spaces, namely if $w\in L^2(E)$ and $v\in H^1_0(E)$ one has $\langle w,v \rangle_{H^1_0(E)}=\langle w,v\rangle_{L^2(E)}$.
	
	\section{The wave equation on moving domains}\label{sec:setting}
	
	For $T>0$, we consider a family $\{\Omega_t\}_{t\in [0,T]}$ of domains in $\R^N$, with $N\in \N$, namely:
	\begin{subequations}\label{hypset}
		\begin{equation}
		\text{for every $t \in [0,T]$ the set $\Omega_{t}\subseteq \R^N$ is nonempty, open, bounded and Lipschitz;}\label{hypsetreg}
		\end{equation}
In some of the results, in view of the applications to debonding models (see Section~\ref{sec:debonding}), we shall also assume that the family $\{\Omega_t\}_{t\in [0,T]}$ is nondecreasing with respect to inclusion:
\begin{equation}\label{hypsetmonotone}
	\Omega_s\subseteq\Omega_t \quad\text{ for every }0\le s\le t\le T.
\end{equation}
	\end{subequations}	
	We denote the complement of $\Omega_t$ by
	\begin{equation*}
		\Omega_t^c:=\R^N\setminus \Omega_t.
	\end{equation*}
	Furthermore, we introduce the \lq\lq space-time\rq\rq domain $\mc O$ and its parabolic boundary $\Gamma$ by
	\begin{equation}\label{eq:O}
		\mc O:=\bigcup_{t\in(0,T)}\{t\}\times \Omega_t\quad\text{ and }\quad \Gamma:=\bigcup_{t\in(0,T)}\{t\}\times \partial\Omega_t.
	\end{equation}
	Let us consider the following formal problem for a function $u\colon \overline{\mc O} \to \R$:
	\begin{equation}\label{eq:u}
		\begin{cases}
			\ddot{u}(t,x)-\Delta u(t,x)=f(t,x), &(t,x)\in\mc O,\\
			u(t,x)=0, &(t,x)\in\Gamma,\\
			u(0,x)=u_0(x),& x \in \Omega_0,\\
			\dot{u}(0,x)=u_1(x), & x \in \Omega_0.
		\end{cases}
	\end{equation}
	The system above consists of a wave equation in the noncylindrical domain $\mc O$ with forcing term
	\begin{subequations}\label{eq:data}
		\begin{equation}\label{eq:f}
			f\in L^2(\mc O),
		\end{equation}
	complemented with initial conditions
		\begin{equation}
			u_0\in H^1_0(\Omega_0),\quad 
			u_1\in L^2(\Omega_0),
		\end{equation}
	and a homogeneous Dirichlet boundary condition on $\Gamma$.
	\end{subequations} 

	\begin{rem}\label{rem:nonhomogeneous}
All the results within the paper may also be adapted to more general hyperbolic equations of the form
		\begin{equation}\label{eq:uA}
			\ddot{u}(t,x)-\div(A(t,x)\nabla u(t,x))=f(t,x),\quad\quad(t,x)\in\mc O,
		\end{equation}
	which model for instance non-homogeneous materials. The minimal assumptions on the matrix $A(t,x)$ needed to perform all the arguments are the following regularity property
	\begin{equation}\label{eq:regA}
		A\in C^{1,1}(\overline{\mc O};\R^{N\times N}_{\rm sym}),
	\end{equation}
	and the uniform ellipticity condition
	\begin{equation}\label{eq:ellA}
		(A(t,x) w) \cdot w \geq c_A |w|^2, \quad \text{ for all } w \in \R^N,
	\end{equation}
which must hold for all $(t,x)\in \overline{\mc O}$ with a positive constant $c_A>0$.  \par
Since the proofs remain basically unchanged, in order to avoid too heavy notations, throughout the paper we prefer to focus our attention on problem \eqref{eq:u}, i.e. with $A(t,x)=I$. The main changes in the statements are instead highlighted, see Remarks~\ref{rem:A1}, \ref{rem:A2}, \ref{rem:A3}, \ref{rem:A4}, \ref{rem:A5}, \ref{rem:A6}.\par
For the sake of clarity we will also adopt the following convention:
\begin{equation*}
	|w|_{A(t,x)}:=\sqrt{(A(t,x) w) \cdot w},  \quad \text{ for all } w \in \R^N.
\end{equation*}
Notice that by \eqref{eq:regA} and \eqref{eq:ellA} the function $|\cdot|_{A(t,x)}$ defines a norm on $\R^N$ for every fixed $(t,x)\in \overline{\mc O}$. 
	\end{rem}

	Since we are working in time-dependent domains it is useful to introduce time-dependent Bochner spaces. Given a family of normed spaces $\{X_t\}_{t\in [0,T]}$, with a slight abuse of notation we say that a function $v$ belongs to $L^p(0,T;X_t)$, with $p\in [1,+\infty]$, if $v(t)\in X_t$ for a.e.\ $t\in (0,T)$ and the map $t\mapsto \|v(t)\|_{X_t}$ is in $L^p(0,T)$. Notice that $L^2(\mc O)=L^2(0,T;L^2(\Omega_t))$ by Fubini's theorem.
Using a similar convention for Sobolev spaces, whenever $\mc O$ is open one may write
\[
H^1(\mc O)=L^2(0,T;H^1(\Omega_t))\cap H^1(0,T;L^2(\Omega_t)) 
\]
and
\[
H^2(\mc O)=L^2(0,T;H^2(\Omega_t))\cap H^1(0,T;H^1(\Omega_t))\cap H^2(0,T;L^2(\Omega_t)) \,.
\]
However, we prefer to employ the notation of time-dependent Bochner spaces only when necessary, and only for spaces of the type $L^p(0,T;X_t)$. 
\par
	We can now give the definition of weak solution to problem \eqref{eq:u}. Notice that property $(i)$ of the definition below involves the time-dependent spaces $H^1(\Omega_t)$ and $L^2(\Omega_t)$, while $(ii)$ features usual Bochner spaces of continuous functions with fixed target.
In property $(ii)$ it is understood that we consider the restriction of $u$ to $[0,T]\times\Omega_0$
and of $\dot u$ to $[0,\delta]\times\Omega'$.
Here and henceforth, all solutions $u$ to equation \eqref{eq:u} will be extended to $0$ outside $\mc O$. 
	
	\begin{defin}\label{defweaksolu}
		We say that $u\colon\overline{\mc O} \to \R$ is a weak solution to problem \eqref{eq:u} with data \eqref{eq:data} if
		\begin{enumerate}
			\item[(i)] $u \in L^2(0,T; H^1_0(\Omega_t))$ and $\dot{u} \in L^2(0,T; L^2(\Omega_t))$;
			\item[(ii)] $u\in C^0([0,T]; L^2(\Omega_0))$ and  $\dot{u} \in C^0([0,\delta]; H^{-1}(\Omega'))$ for every $\delta>0$ and $\Omega'\subseteq \Omega_0$ open such that $[0,\delta]\times\Omega'\subseteq\mc O$; moreover, the initial conditions $u(0)=u_0$ and $\dot{u}(0)=u_1$ hold;
			\item[(iii)]  $u$ satisfies
			\begin{equation}\label{def eq u weak sol}
				\begin{aligned}
					-\int_{0}^{T} \langle \dot{u}(t), \dot{\eta}(t) \rangle_{L^2(\Omega_t)} \d t + \int_{0}^{T} \langle \nabla u(t), \nabla \eta(t) \rangle_{L^2(\Omega_t)}\d t 
					= \int_{0}^{T} \langle f(t), \eta(t) \rangle_{L^2(\Omega_t)}\d t,
				\end{aligned}
			\end{equation}
			for every $\eta \in L^2(0,T;H^1_0(\Omega_t))$ with $\dot\eta \in L^2(0,T;L^2(\Omega_t))$ and $\eta(T)=\eta(0)=0$.
		\end{enumerate}
	\end{defin}

\begin{rem}
The regularity assumptions of property $(ii)$, which allow one to state the initial conditions on position and velocity, are actually consequence of $(i)$  and $(iii)$ under additional hypotheses on $t\mapsto\Omega_t$.
\par
Assume for instance that \eqref{hypsetmonotone} holds, as in Section \ref{sec:exist-discrete}. Then
$(i)$ implies that $u$ and $\dot u$ belong to $ L^2(0,T; L^2(\Omega_0))$,
hence $u\in C^0([0,T]; L^2(\Omega_0))$;
moreover, the wave equation $(iii)$ implies that $\ddot u=\Delta u+f \in L^2(0,T; H^{-1}(\Omega_0))$,
hence $\dot u\in C^0([0,T]; H^{-1}(\Omega_0))$.
\par
Without assuming monotonicity, a similar argument holds provided there exist diffeomorphisms as in \eqref{assumptionsdiff}, satisfying \ref{H1}; such properties are assumed from Section \ref{sec:fixeddomain} onwards.
\end{rem}

In the paper we show how to obtain existence, uniqueness and an energy balance for solutions $u$ of problem \eqref{eq:u}, in particular to explicitly derive an expression for the energy variation due to the evolution of the domain (which will be interpreted as the energy spent during the debonding process). Such energy balance will be crucial in Section~\ref{sec:debonding}, where we present applications to dynamic debonding models. If one assumes existence of a regular solution as in \eqref{addreg}, deriving the energy balance \eqref{eq:enbalu} is a direct computation: we present it in Theorem~\ref{teo energy balance} for the Reader's convenience. A delicate issue is instead the rigorous proof of the regularity property \eqref{addreg}: this will be the aim of Sections~\ref{sec:fixeddomain} and \ref{sec:hyperbolic}. In Section~\ref{sec:application} we will then obtain a more precise statement of the energy balance, see Theorem~\ref{teo higher regularity u}.
	\begin{teo}\label{teo energy balance}
	Assume \eqref{hypsetreg}, assume that $\mc O$ is open with Lipschitz boundary and that 
	\begin{equation}\label{bdryO}
		\partial\mc O=\Gamma\cup (\{T\}\times {\overline{\Omega}_T})\cup(\{0\}\times {\overline{\Omega}_0}).
	\end{equation}
	Let $u$ be a weak solution to problem \eqref{eq:u} satisfying  the following regularity property:
	\begin{equation}\label{addreg}
		u\in L^2(0,T;H^2(\Omega_t)\cap H^1_0(\Omega_t)),\quad \dot{u}\in L^2(0,T;H^1(\Omega_t)),\quad \ddot u\in L^2(0,T;L^2(\Omega_t)).
	\end{equation}
	Then for every $t \in [0,T]$ the following energy balance holds true:
		\begin{equation}\label{eq:enbalu}
		\begin{aligned}
				&\quad\,\frac 12\|\dot u(t)\|^2_{L^2(\Omega_t)}+\frac 12\|\nabla u(t)\|^2_{L^2(\Omega_t)}-\int_{\Gamma_t}\frac{\nu_{\mc O}^t}{2}\left[1-\left(\frac{\nu_{\mc O}^t}{|\nu_{\mc O}^x|}\right)^2\right]|\nabla u|^2\d \mc H^N\\
				&= \frac 12\|u_1\|^2_{L^2(\Omega_0)}+\frac 12\|\nabla u_0|^2_{L^2(\Omega_0)}+ \int_{0}^{t}\langle f(s), \dot{u}(s)\rangle_{L^2(\Omega_s)} \d s,
		\end{aligned}
		\end{equation}
	where $\Gamma_t:=\{( s,x)\in \Gamma:\,  s\in (0,t)\}$.
	\end{teo}
In the energy balance \eqref{eq:enbalu} we recognize in the left-hand side the kinetic energy, the potential energy and a term corresponding to the evolution of the domain, while in the right-hand side we see the initial energy and the work of external forces. The integral over $\Gamma_t$ may be positive or negative, according to the geometry of $\mc O$. If the monotonicity condition \eqref{hypsetmonotone} is in force one has $\nu^t_{\mc O}\le0$; moreover a typical assumption is that the growth of the domains is subsonic, i.e.\ $|\nu^t_{\mc O}|\le|\nu^x_{\mc O}|$ (in this situation, the set $\mc O$ is usually called time-like, see \cite{Coop,BardCoop,DaPraZol, MaRuChu,Siko,Zol}). In applications to debonding models, where both conditions hold, the integral over $\Gamma_t$ can be thus interpreted as energy dissipated in the debonding process.
		\begin{rem}
		We point out that condition \eqref{bdryO} is a weak regularity assumption on the time-depen\-dence of the domain, and it does not hold in general for sets satisfying \eqref{hypset}. Indeed, if there is a time-discontinuity at time $t_0$, the boundary of $\mc O$ will contain a set of the form $\{t_0\}\times D$, for some set $D$. The Reader may think for instance to the simple one-dimensional example
		\begin{equation*}
			\Omega_t=\begin{cases}
				(0,1),&\text{if }t\in [0,1),\\
				(0,2),&\text{if }t\in[1,2],
			\end{cases}
		\end{equation*}
		in which the set $\{1\}\times(1,2)$ is contained in $\partial\mc O$. In Section~\ref{sec:fixeddomain} we will assume a stronger regularity condition, which will imply \eqref{bdryO} (see Remark~\ref{rmkdiffbdry}). 
	\end{rem}
	\begin{rem}
		The request of higher regularity \eqref{addreg} is needed to give a meaning to the term in \eqref{eq:enbalu} representing the energy variation due to the evolution of the domain, where $\nabla u$ has to be integrated along the lateral boundary $\Gamma$. A reformulation of this term, in such a way that the energy balance may be written for genuine weak solutions (i.e. just satisfying (i) in Definition~\ref{defweaksolu}), would be certainly desirable; unfortunately, to our better knowledge, a suitable rewriting of such term is still not available.
	\end{rem}

	\begin{rem}\label{rem:A1} 
		Under the additional assumptions of Remark~\ref{rem:nonhomogeneous}, the energy balance \eqref{eq:enbalu} changes in
		\begin{equation*}
			\begin{aligned}
				&\quad\,\frac 12\|\dot u(t)\|^2_{L^2(\Omega_t)}+\frac 12\left<A(t)\nabla u(t),\nabla u(t)\right>_{L^2(\Omega_t)}-\int_{\Gamma_t}\frac{\nu_{\mc O}^t}{2}\left[|\nabla u|^2_A-\left(\frac{\nu_{\mc O}^t}{|\nu_{\mc O}^x|}\right)^2|\nabla u|^2\right]\d \mc H^N\\
				&= \frac 12\|u_1\|^2_{L^2(\Omega_0)}+\frac 12\left<A(0)\nabla u_0,\nabla u_0\right>_{L^2(\Omega_0)}+ \int_{0}^{t}\!\!\!\langle f(s), \dot{u}(s)\rangle_{L^2(\Omega_s)} \d s+\frac 12\int_{0}^{t}\!\!\!\langle\dot{A}(s)\nabla u(s),\nabla u(s)\rangle_{L^2(\Omega_s)}\d s.
			\end{aligned}
		\end{equation*}	
	\end{rem}
	\begin{proof}[Proof of Theorem~\ref{teo energy balance}]
		By exploiting \eqref{addreg} we deduce that $u$ belongs to $H^2(\mc O)$ and thus satisfies 
		\begin{equation*}
			\ddot{u}(t,x) -\Delta u(t,x) = f(t,x), \qquad \text{for a.e. } (t,x)\in \mathcal{O}.
		\end{equation*}
		Multiplying the previous equation by a function 
		$\varphi \in L^2(0,T; H^1(\Omega_t))$ such that $\dot \varphi\in   L^2(0,T; L^2(\Omega_t))$, and integrating by parts in the ``space-time" domain $\mathcal{O}$, for all $t\in [0,T]$ we obtain:
		\begin{equation*}
			\begin{aligned}
				&\langle \dot{u}(t), \varphi (t) \rangle_{L^2(\Omega_t)} 
				-\langle u_1, \varphi (0) \rangle_{L^2(\Omega_0)} 
				-\int_{0}^{t}\langle \dot{u}(s),\dot{\varphi}(s) \rangle_{L^2(\Omega_s)} \d s +\int_{0}^{t}\langle \nabla u(s),\nabla \varphi(s) \rangle_{L^2(\Omega_s)}\d  s\\
				=&\int_{0}^{t}\langle f(s), \varphi(s) \rangle_{L^2(\Omega_s)}\d  s
				-\int_{\Gamma_t}\left( \dot{u}\,\nu_{\mc O}^t -\nabla u\cdot\nu_{\mc O}^x \right) \varphi \, \d\mathcal{H}^N.
			\end{aligned}
		\end{equation*}
		Thanks to the regularity provided by \eqref{addreg}, we can choose as test function $\varphi = \dot{u}$. We thus obtain
		\begin{equation}\label{eq:formalenbaln}
			\begin{aligned}
				&\norm{\dot{u}(t)}^2_{L^2(\Omega_t)}
				- \norm{u_1}^2_{L^2(\Omega_0)}
				-\int_{0}^{t}\langle \dot{u}(s),\ddot{u}(s) \rangle_{L^2(\Omega_s)} \d s +\int_{0}^{t}\langle \nabla u(s),\nabla \dot{u}(s) \rangle_{L^2(\Omega_s)}\d  s\\
				=&\int_{0}^{t}\langle f(s), \dot{u}(s) \rangle_{L^2(\Omega_s)}\d  s
				- \int_{\Gamma_t} (\dot{u})^2 \nu_{\mc O}^t\, \d\mathcal{H}^N + \int_{\Gamma_t} \left(\nabla u\cdot\nu_{\mc O}^x \right) \dot{u} \, \d\mathcal{H}^N.
			\end{aligned}
		\end{equation}
		Integrating by parts in time the integral terms in the first line, we get that
		\begin{subequations}\label{eq:subequnabla}
				\begin{align}
				\label{eq ddot u energy balance}
				\int_{0}^{t}\langle \dot{u}(s),\ddot{u}(s) \rangle_{L^2(\Omega_s)} \d s& = \frac{1}{2} \norm{\dot{u}(t)}^2_{L^2(\Omega_t)} -\frac{1}{2} \norm{u_1}^2_{L^2(\Omega_0)} +\frac{1}{2} \int_{\Gamma_t} (\dot{u})^2 \nu_{\mc O}^t\, \d\mathcal{H}^N,
				\\
				\label{eq nabla dot u energy balance}
				\int_{0}^{t}\langle \nabla u(s),\nabla \dot{u}(s) \rangle_{L^2(\Omega_s)}\d  s &= 
				\frac{1}{2} \norm{\nabla u(t)}^2_{L^2(\Omega_t)} -\frac{1}{2} \norm{\nabla u_0}^2_{L^2(\Omega_0)} + \frac{1}{2} \int_{\Gamma_t} |\nabla u|^2 \nu_{\mc O}^t\, \d\mathcal{H}^N.
			\end{align}
		\end{subequations}
		We now notice that, since $u\equiv0$ on $\Gamma$, it must hold
		\begin{equation*}\label{eq:releta}
			\dot{u}\,\nu_{\mc O}^x=\nu_{\mc O}^t \nabla u,\quad\mc H^N\text{-a.e. on }\Gamma,
		\end{equation*}
		which in particular implies the relations
		\begin{subequations}\label{eq:reldern}
				\begin{alignat}{3}
				&\dot{u}(\nabla u\cdot\nu_{\mc O}^x)&&=\nu_{\mc O}^t|\nabla u|^2, &&\quad\quad\mc H^N\text{-a.e. on }\Gamma,\\
			&(\dot{u})^2|\nu_{\mc O}^x|^2&&=(\nu_{\mc O}^t)^2|\nabla u|^2,&&\quad\quad\mc H^N\text{-a.e. on }\Gamma.
		\end{alignat}
		\end{subequations}
		By plugging \eqref{eq:subequnabla} and \eqref{eq:reldern} into \eqref{eq:formalenbaln}, we finally conclude that
		\begin{equation*}
			\begin{aligned}
				&\frac{1}{2} \norm{\dot{u}(t)}^2_{L^2(\Omega_t)} -\frac{1}{2} \norm{u_1}^2_{L^2(\Omega_0)}
				+ \frac{1}{2} \norm{\nabla u(t)}^2_{L^2(\Omega_t)} -\frac{1}{2} \norm{\nabla u_0}^2_{L^2(\Omega_0)}
				\\
				=&\int_{0}^{t}\langle f(s), \dot{u}(s) \rangle_{L^2(\Omega_s)}\d  s
				+ \int_{\Gamma_t}\frac{\nu_{\mc O}^t}{2}\left[1-\left(\frac{\nu_{\mc O}^t}{|\nu_{\mc O}^x|}\right)^2\right]|\nabla u|^2\d \mc H^N,
			\end{aligned}
		\end{equation*}
		thus the statement is proved.
	\end{proof}
	
	\subsection{An existence result} \label{sec:exist-discrete}
	We conclude this section by proposing a novel strategy to prove existence of solutions to the wave equations in noncylindrical domains; see \cite{CalvNovOrl} for a similar approach in the context of parabolic equations. It is based on time-discretization and it does not require any time-regularity on the growth of the sets $\Omega_t$. However it is now crucial to require that the family $\{\Omega_t\}_{t\in [0,T]}$ is nondecreasing, namely \eqref{hypsetmonotone}. Under these assumptions, we also obtain an energy inequality (in contrast, the energy balance was obtained before under stronger regularity hypotheses on the solution).\par
	We stress that our discretization procedure is substantially different (and from our point of view simpler and more intuitive) with respect to the (variant of) the classical minimizing movements approach used for hyperbolic problems, for instance applied in the context of dynamic fracture mechanics in \cite{DMToad}. Indeed, the latter relies on an iterative minimization of a suitable energy, followed by the construction of a piecewise affine interpolant. In our approach, instead, after the discretization of the time interval $[0,T]$ we consider the related piecewise constant evolution of the domains $\Omega_t$ and in each discrete-time interval we pick the solution of the wave equation in the corresponding cylindrical domain (see \eqref{eq:discrwave}). This allows us to employ well known results for the wave equation in cylindrical domains.
	
Before stating the result, we recall that, given a Banach space $X$, the set $C^0_w([0,T];X)$ denotes the space of functions $u\colon [0,T]\to X$ which are continuous with respect to the weak topology of $X$.
Notice that here $X$ is independent of time: in fact, we adopt the convention that the solutions of \eqref{eq:u} are extended to zero outside $\mc O$.
	
	\begin{teo}\label{teoexistencecylindermethod}
		Assume \eqref{hypset} and \eqref{eq:data}. Then there exists a weak solution $u$ of problem \eqref{eq:u} in the sense of Definition \ref{defweaksolu}. Moreover,
		\begin{align*}
			& u \in C^0_w([0,T];H^1_0(\Omega_T)),
			\\
			& \dot{u} \in L^\infty(0,T;L^2(\Omega_t))  \cap C^0_w([\bar t,T];L^2(\Omega_{\bar t})), \quad \text{for all } \bar t\in [0,T).
		\end{align*}
		Furthermore the following energy inequality holds true for every $t\in [0,T]$:
		\begin{equation}\label{eq:enineq}
			\frac 12\norm{\dot{u}(t)}^2_{L^2(\Omega_t)}+\frac 12\norm{\nabla {u}(t)}^2_{L^2(\Omega_t)}\le\frac 12\norm{u_1}^2_{L^2(\Omega_0)}+\frac 12\norm{\nabla {u}_0}^2_{L^2(\Omega_0)}+\int_{0}^{t}\scal{f(s)}{\dot{u}(s)}_{L^2(\Omega_s)}\d s.
		\end{equation}
	\end{teo}
	
	\begin{proof}
		We adopt a time discretisation argument: we consider a sequence of partitions of $[0,T]$ with vanishing size, namely for every $n\in \N$ we take $0=t^n_0<t^n_1<\dots<t^n_{k(n)}=T$ satisfying
		\begin{equation}\label{eq:finezza}
			\lim\limits_{n\to +\infty}\sup\limits_{k=1,\dots, k(n)}|t^n_k-t^n_{k-1}|=0.
		\end{equation}
		For every $k=1,\dots,{k(n)}$ we then take $u^n_k$ as the unique weak solution of the wave equation in the cylinder $(t^n_{k-1},t^n_{k})\times\Omega_{t^n_{k-1}}$ with initial data $u^n_{k-1}(t^n_{k-1})$ and $\dot{u}^n_{k-1}(t^n_{k-1})$ (with the convention $u^n_0(0)=u_0$ and $\dot{u}^n_0(0)=u_1$), namely
		\begin{equation}\label{eq:discrwave}
			\begin{cases}
				\ddot{u}^n_k-\Delta u^n_k=f, &\text{in }(t^n_{k-1},t^n_{k})\times\Omega_{t^n_{k-1}},\\
				u^n_k=0, &\text{in } (t^n_{k-1},t^n_{k})\times\partial\Omega_{t^n_{k-1}},\\
				u^n_k(t^n_{k-1})=u^n_{k-1}(t^n_{k-1}),\\
				\dot{u}^n_k(t^n_{k-1})=\dot{u}_{k-1}(t^n_{k-1}).
			\end{cases}
		\end{equation}
		We adopt the usual convention that $u^n_k$ is extended to $0$ in $\Omega^c_{t^n_{k-1}}$. Standard arguments show the following properties:
		\begin{itemize}
			\item[(a)] $u^n_k$ belongs to $ C^0([t^n_{k-1},t^n_{k}];H^1_0(\Omega_T))\cap C^1([t^n_{k-1},t^n_{k}];L^2(\Omega_T))$;
			\item[(b)] $u^n_k(t^n_{k-1})=u^n_{k-1}(t^n_{k-1})$ in the sense of $C^0([t^n_{k-1},t^n_{k}]; H^1_0(\Omega_T))$ and $\dot{u}^n_k(t^n_{k-1})=\dot{u}^n_{k-1}(t^n_{k-1})$ in the sense of $C^0([t^n_{k-1},t^n_{k}]; L^2(\Omega_T))$;
			\item[(c)] for every $\eta\in L^2(t^n_{k-1},t^n_{k};H^1_0(\Omega_T))\cap H^1(t^n_{k-1},t^n_{k};L^2(\Omega_T))$ s.t. $\eta(t)=0$ in $\Omega^c_{t^n_{k-1}}$ for a.e. $t\in (t^n_{k-1},t^n_{k})$ it holds:
			\begin{equation*}
				\begin{aligned}
					&\quad\, -\int_{t^n_{k-1}}^{t^n_{k}}\scal{\dot{u}^n_k( s)}{\dot{\eta}( s)}_{L^2(\Omega_T)}\d s+\int_{t^n_{k-1}}^{t^n_{k}}\scal{\nabla u^n_k( s)}{\nabla \eta( s)}_{L^2(\Omega_T)}\d s\\
					& =\int_{t^n_{k-1}}^{t^n_{k}}\scal{f( s)}{\eta( s)}_{L^2(\Omega_T)}\d s+\scal{\dot{u}^n_{k-1}(t^n_{k-1})}{{\eta}(t^n_{k-1})}_{L^2(\Omega_T)}-\scal{\dot{u}^n_{k}(t^n_{k})}{{\eta}(t^n_{k})}_{L^2(\Omega_T)}.
				\end{aligned}
			\end{equation*}
		\end{itemize}
		Furthermore we also have the energy balance:
		\begin{itemize}
			\item[(d)] for every $t\in[t^n_{k-1},t^n_{k}]$ it holds:
			\begin{equation}\label{eq:enbalk}
				\begin{aligned}
					&\quad\frac 12\norm{\dot{u}^n_k(t)}^2_{L^2(\Omega_{t^n_{k-1}})}+\frac 12\norm{\nabla {u}^n_k(t)}^2_{L^2(\Omega_{t^n_{k-1}})}\\
					&=\frac 12\norm{\dot{u}^n_{k-1}(t_{k-1})}^2_{L^2(\Omega_{t^n_{k-2}})}+\frac 12\norm{\nabla {u}^n_{k-1}(t_{k-1})}^2_{L^2(\Omega_{t^n_{k-2}})}+\int_{t^n_{k-1}}^{t}\scal{f( s)}{\dot{u}^n_k( s)}_{L^2(\Omega_{t^n_{k-1}})}\d s,
				\end{aligned}
			\end{equation}
		\end{itemize}
		where we extended also $f$ in the whole $(0,T)\times\Omega_T$ by setting $f\equiv 0$ outside $\mc O$. In particular, by recalling again that $u^n_k(t)$ vanishes in $\Omega^c_{t^n_{k-1}}$ and by summing \eqref{eq:enbalk} for $j=2,\dots, k$, we deduce for every $t\in[t^n_{k-1},t^n_{k}]$:
		\begin{equation}\label{eq:enbalk2}
			\begin{aligned}
				&\quad\,\frac 12\norm{\dot{u}^n_k(t)}^2_{L^2(\Omega_T)}+\frac 12\norm{\nabla {u}^n_k(t)}^2_{L^2(\Omega_T)}\\
				&=\frac 12\norm{u_1}^2_{L^2(\Omega_T)}{+}\frac 12\norm{\nabla u_0}^2_{L^2(\Omega_T)}{+}\int_{t^n_{k-1}}^{t}\!\!\!\!\!\!\!\scal{f( s)}{\dot{u}^n_k( s)}_{L^2(\Omega_T)}\d s{+}\!\sum_{j=2}^{k}\!\int_{t^n_{j-2}}^{t^n_{j-1}}\!\!\scal{f( s)}{\dot{u}^n_{j-1}( s)}_{L^2(\Omega_T)}\d s.
			\end{aligned}
		\end{equation}
		For every $n\in \N$ we now define 
		\begin{equation*}\label{eq:vn}
			u^n(t):=\begin{cases}
				u^n_k(t),&\text{ if }t\in [t^n_{k-1},t^n_k)\text{ for some }k=1,\dots,{k(n)},\\
				u^n_{k(n)}(T), &\text{ if }t=T.
			\end{cases}
		\end{equation*}
		By construction $u^n$ belongs to $C^0([0,T];H^1_0(\Omega_T))\cap C^1([0,T];L^2(\Omega_T))$ and  satisfies:
		\begin{itemize}
			\item[(a')] $u^n=0$ in $\bigcup\limits_{k=1}^{k(n)}[t^n_{k-1},t^n_k]\times\Omega^c_{t^n_{k-1}}\supseteq \bigcup\limits_{t\in[0,T]}\{t\}\times\Omega^c_t$;
			\item[(b')] $u^n(0)=u_0$ in the sense of $C^0([0,T];H^1_0(\Omega_T))$ and $\dot{u}^n(0)=u_1$ in the sense of $C^0([0,T];L^2(\Omega_T))$;
			\item[(c')] for every $\eta\in C^{\infty}_{\rm c}((0,T)\times \Omega_T)$ s.t. $\supp\eta\subseteq \bigcup\limits_{k=1}^{k(n)}[t^n_{k-1},t^n_k)\times\Omega_{t^n_{k-1}}$ it holds:
			\begin{equation}\label{eq:weakvn}
				-\int_{0}^{T}\scal{\dot{u}^n( s)}{\dot{\eta}( s)}_{L^2(\Omega_T)}\d s+\int_{0}^{T}\scal{\nabla u^n( s)}{\nabla \eta( s)}_{L^2(\Omega_T)}\d s=\int_{0}^{T}\scal{f( s)}{\eta( s)}_{L^2(\Omega_T)}\d s.
			\end{equation}
		\end{itemize} 
Furthermore the energy balance \eqref{eq:enbalk2} reads as follows:
		\begin{itemize}
			\item[(d')] for every $t\in[0,T]$ it holds:
			\begin{equation}\label{eq:enbaln}
				\frac 12\norm{\dot{u}^n(t)}^2_{L^2(\Omega_T)}\!{+}\frac 12\norm{\nabla {u}^n(t)}^2_{L^2(\Omega_T)}\!=\!\frac 12\norm{u_1}^2_{L^2(\Omega_T)}{+}\frac 12\norm{\nabla {u}_0}^2_{L^2(\Omega_T)}\!{+}\!\!\int_{0}^{t}\!\scal{f( s)}{\dot{u}^n( s)}_{L^2(\Omega_T)}\d s.
			\end{equation}
		\end{itemize}
		By a classical Gr\"onwall argument, since the forcing term $f$ is in $L^2((0,T)\times \Omega_T)$, we thus deduce
		\begin{equation*}
			\max\limits_{t\in [0,T]}\left(\frac 12\norm{\dot{u}^n(t)}^2_{L^2(\Omega_T)}+\frac 12\norm{\nabla {u}^n(t)}^2_{L^2(\Omega_T)}\right)\le C.
		\end{equation*}
		This implies the existence of $u\in L^\infty(0,T;H^1_0(\Omega_T))$, and of $u^*\in L^\infty(0,T;L^2(\Omega_T))$ such that, up to subsequences (not relabelled), we have
		\begin{equation*}
			u^n\rightharpoonup u,\quad\text{weakly in } L^2(0,T;H^1_0(\Omega_T)),\quad\text{ and }\quad\dot{u}^n\rightharpoonup u^*,\quad\text{weakly in } L^2(0,T;L^2(\Omega_T)).
		\end{equation*}
		It is standard to show that $u^*=\dot{u}$. Thus we deduce the existence of a function $u\in L^\infty(0,T;H^1_0(\Omega_T))$ with $\dot u \in L^\infty(0,T;L^2(\Omega_T))$ such that
		\begin{equation}\label{eq:convergence}
			u^n\rightharpoonup u,\quad\text{weakly in } L^2(0,T;H^1_0(\Omega_T))\cap H^1(0,T;L^2(\Omega_T)).
		\end{equation}
Notice that in \eqref{eq:convergence} the target spaces are independent of time; however,
		by (a') we easily deduce that $u\equiv 0$ outside $\mc O$, so we get the stronger conditions $u\in L^\infty(0,T;H^1_0(\Omega_t))$ and $\dot u \in L^\infty(0,T;L^2(\Omega_t))$. By the continuous embedding $L^\infty(0,T; H^1_0(\Omega_T))\cap H^1(0,T;L^2(\Omega_T))\subseteq C^0_w([0,T];H^1_0(\Omega_T))$ we also obtain $u\in C^0_w([0,T];H^1_0(\Omega_T))$.  \par
To complete the proof that $u$ satisfies Definition \ref{defweaksolu} we prove \eqref{def eq u weak sol}
by passing to the limit in \eqref{eq:weakvn} by means of \eqref{eq:convergence}.
Here, a technical issue is that the spaces of test functions in \eqref{def eq u weak sol} and in \eqref{eq:weakvn} are different. However, given a function $\eta \in L^2(0,T;H^1_0(\Omega_t))$ with $\dot\eta \in L^2(0,T;L^2(\Omega_t))$ and $\eta(T)=\eta(0)=0$, 
we can approximate it by a sequence of smooth functions $\eta^n$ as in (c'):
this readily follows thanks to \eqref{eq:finezza} and concludes the proof of \eqref{def eq u weak sol}.
We finally observe that for every $\bar t\in [0,T)$ the function $u$ is in particular a weak solution of the wave equation in the cylinder $(\bar t,T)\times\Omega_{\bar t}$, and thus it belongs to $C^1_w([\bar t,T];L^2(\Omega_{\bar t}))$.\par
		We are only left to prove the energy inequality \eqref{eq:enineq}. We integrate \eqref{eq:enbaln} between arbitrary times $\alpha,\beta$ with $0\le\alpha\le\beta\le T$. By \eqref{eq:convergence} and standard lower semicontinuity arguments, as $n\to +\infty$ we obtain
		\begin{align*}
			&\quad\, \int_\alpha^\beta\frac 12\norm{\dot{u}(t)}^2_{L^2(\Omega_T)}+\frac 12\norm{\nabla {u}(t)}^2_{L^2(\Omega_T)}\d t\\
			&\le(\beta-\alpha)\left(\frac 12\norm{u_1}^2_{L^2(\Omega_0)}+\frac 12\norm{\nabla {u}_0}^2_{L^2(\Omega_0)}\right)+\int_\alpha^\beta\int_{0}^{t}\scal{f( s)}{\dot{u}( s)}_{L^2(\Omega_T)}\d s\d t.
		\end{align*}
		By the arbitrariness of $\alpha$ and $\beta$, for a.e.\ $t\in[0,T]$ the above inequality yields 
		\begin{equation}\label{eq:enineqae}
			\frac 12\norm{\dot{u}(t)}^2_{L^2(\Omega_T)}+\frac 12\norm{\nabla {u}(t)}^2_{L^2(\Omega_T)}\le\frac 12\norm{u_1}^2_{L^2(\Omega_0)}+\frac 12\norm{\nabla {u}_0}^2_{L^2(\Omega_0)}+\int_{0}^{t}\scal{f( s)}{\dot{u}( s)}_{L^2(\Omega_T)}\d s.
		\end{equation}
We will now improve \eqref{eq:enineqae} by providing an energy inequality valid for every time. We fix $\bar t\in [0,T)$ and consider a sequence $t_k\searrow \bar t$ along which \eqref{eq:enineqae} is satisfied. Since $u\in C^0_w([0,T];H^1_0(\Omega_T))$, $\dot u\in C^0_w([\bar t,T];L^2(\Omega_{\bar t}))$ and $\Omega_{\bar t}\subseteq \Omega_T$, again by weak lower semicontinuity we deduce
		\begin{align*}
			&\quad\,\frac 12\norm{\dot{u}(\bar t\,)}^2_{L^2(\Omega_{\bar t})}+\frac 12\norm{\nabla {u}(\bar t\,)}^2_{L^2(\Omega_{\bar t})}\\
			&\le\liminf\limits_{k\to +\infty}\left(\frac 12\norm{\dot{u}(t_k)}^2_{L^2(\Omega_T)}+\frac 12\norm{\nabla {u}(t_k)}^2_{L^2(\Omega_T)}\right)\\
			&\le \frac 12\norm{u_1}^2_{L^2(\Omega_0)}+\frac 12\norm{\nabla {u}_0}^2_{L^2(\Omega_0)}+\int_{0}^{\bar t}\scal{f( s)}{\dot{u}( s)}_{L^2(\Omega_T)}\d s\\
			&=\frac 12\norm{u_1}^2_{L^2(\Omega_0)}+\frac 12\norm{\nabla {u}_0}^2_{L^2(\Omega_0)}+\int_{0}^{\bar t}\scal{f( s)}{\dot{u}( s)}_{L^2(\Omega_ s)}\d s.
		\end{align*}
Hence, \eqref{eq:enineq} is satisfied for all $\bar t\in [0,T)$. Its validity also in $\bar t=T$ follows by taking a larger final time $\widetilde T>T$, defining for instance $\Omega_t:=\Omega_T$ for $t\in (T,\widetilde T]$ and arguing in the same way.
	\end{proof}
	
	\section{Equivalent reformulation on a fixed domain}\label{sec:fixeddomain}
	
	In this section we recast problem \eqref{eq:u} into a hyperbolic problem in a fixed domain (see \eqref{eq:v} below).
To this end, we adapt the method of diffeomorphisms developed in \cite{Coop,BardCoop, DaPraZol} and employed more recently e.g.\ in \cite{Caponi, DMLuc, ZhouSunLi}.
\par 
	We thus assume \eqref{hypsetreg} and the existence of two functions 
	\begin{equation*}
		\Phi \colon[0,T] \times {\overline{\Omega}_0} \to \R^N, \qquad \Psi\colon\overline{\mc O}\to {\overline{\Omega}_0},
	\end{equation*}
	satisfying
	\begin{subequations}\label{assumptionsdiff}
	\begin{equation}
		\text{$\Phi(t,\Omega_0)=\Omega_t$ and $\Psi(t,\Omega_t)=\Omega_0$ for all $t\in [0,T]$},
	\end{equation}
			\begin{align}
			\label{assumptionPhi}
			\Phi(t,\Psi (t,x)) = x, \quad& \text{ for all } (t,x) \in \overline{\mc O},\\
			\label{assumptionPsi}
			\Psi(t,\Phi (t,y)) = y, \quad&\text{ for all } (t,y)\in [0,T]\times{\overline{\Omega}_0},\\
		\label{eq:phiidentity}
			\Phi(0,y) = y, \quad &\text{ for all } y \in {\overline{\Omega}_0}.
		\end{align}
	\end{subequations}
	We also assume that they fulfil the following assumptions:
	\begin{enumerate}[label=\textup{(H\arabic*)}]
		\item \label{H1} $\Phi,\Psi$ are of class $C^{1,1}$ on their domains of definition;
		\item \label{H2} $|\dot{\Phi}(t,y)|
		< 1$ for every $(t,y) \in [0,T] \times {\overline{\Omega}_0}$.
	\end{enumerate}
	Condition \ref{H2} ensures that the growth speed of the sets $\Omega_t$ is always strictly less than the speed of the travelling waves of problem \eqref{eq:u}; it is crucial in order to guarantee that the transformed problem \eqref{eq:v} is still hyperbolic (see \eqref{coerB}). 
	\begin{rem}\label{rmkdiffbdry}
		We notice that the existence of such diffeomorphisms automatically implies that the set $\mc O$ introduced in \eqref{eq:O} is open and with Lipschitz boundary; furthermore \eqref{bdryO} is valid. See also Lemma~\ref{lemma:diff} and Corollary~\ref{cor:O}.
	\end{rem}
		\begin{rem}\label{rem:A2}
			In the non-homogeneous case depicted in Remark~\ref{rem:nonhomogeneous} the wave speed is no more always equal to one. In this situation condition \ref{H2} rewrites as
			\begin{enumerate}[label=\textup{(H2A)}]
				\item \label{H2A} $|\dot{\Phi}(t,y)|
				< \sqrt{c_A}$ for every $(t,y) \in [0,T] \times {\overline{\Omega}_0}$,
			\end{enumerate}
		where $c_A>0$ is the positive constant appearing in \eqref{eq:ellA}.
		\end{rem}
	
	In Section~\ref{sec:application}, when we study higher regularity of solutions to problem \eqref{eq:u}, we will additionally require:
	\begin{enumerate}[label=\textup{(H1')}]
		\item \label{H1'} $\Phi,\Psi$ are of class $C^{2,1}$ on their domains of definition.
	\end{enumerate}

	In the following lemma we summarize some properties of the diffeomorphisms $\Phi$ and $\Psi$ needed in Theorem \ref{teoequivalenceweaksol} below.
	\begin{lemma}\label{lemma identities Psi Phi}
		Let $\Phi,\Psi $ be as in \eqref{assumptionsdiff} and satisfy \ref{H1}. Then, for almost every $(t,y) \in [0,T]\times{\overline{\Omega}_0}$ the following relations hold:
		\begin{subequations}
		\begin{align}
			\label{eq: DPsi DPhi=I}
			&\bullet\, D \Psi(t,\Phi(t,y)) D \Phi(t,y) = I,\\
			\label{eq: detDPsi detDPhi=1}
			&\bullet\, \det D \Psi(t,\Phi(t,y)) \det D \Phi(t,y) = 1,\\
			\label{eq: dot Psi}
			&\bullet\, \dot{\Psi}(t,\Phi(t,y)) = - D\Psi(t,\Phi(t,y)) \dot{\Phi}(t,y),\\
			\label{eq: nabla (detDPsi detDPhi)=0}
			&\bullet\, \nabla [\det D \Psi(t,\cdot)](\Phi(t,y)) \det D \Phi(t,y) = - \det D \Psi(t,\Phi(t,y)) D\Psi(t,\Phi(t,y))^T  \nabla \det D \Phi(t,y),\\
			&\label{eq: dot (detDPsi detDPhi)=0}\bullet\, \left(\partial_t [\det D \Psi(\cdot,\Phi(t,y))](t)  + \nabla [\det D \Psi(t,\cdot)](\Phi(t,y)) \cdot \dot{\Phi}(t,y) \right) \det D \Phi(t,y)
			\\
			& \quad= - \det D \Psi(t,\Phi(t,y))\, \partial_t \det D \Phi(t,y),\nonumber\\
			 \label{eq: nabla (detDPsi detDPhi) dotPhi =0}
			 &\bullet\, \nabla [\det D \Psi(t,\cdot)](\Phi(t,y)) \cdot \dot{\Phi}(t,y) \det D \Phi(t,y) = \dot{\Psi}(t,\Phi(t,y)) \cdot \nabla\det D \Phi(t,y) \det D \Psi(t,\Phi(t,y)),\\
			 \label{eq: dot detDPsi + div = 0}
			 &\bullet\, \partial_t \det D\Phi(t,y) + \div \left(\dot{\Psi}(t,\Phi(t,y)) \det D\Phi(t,y)  \right)  = 0.
		\end{align}
\end{subequations}
	In particular we notice that
	\begin{equation}\label{eq:detpos}
		\det D\Phi(t,y)>0,\quad\text{for every }(t,y) \in [0,T]\times{\overline{\Omega}_0}.
	\end{equation}
	\end{lemma}
	\begin{proof}
		Relations \eqref{eq: DPsi DPhi=I} and \eqref{eq: dot Psi} simply follow by differentiating \eqref{assumptionPsi} with respect to $y$ and $t$, respectively. Then, \eqref{eq: DPsi DPhi=I} easily implies \eqref{eq: detDPsi detDPhi=1} and \eqref{eq:detpos} by \eqref{eq:phiidentity}. Moreover, differentiating the identity \eqref{eq: detDPsi detDPhi=1} with respect to $y$ and using \eqref{eq: DPsi DPhi=I}, one obtains \eqref{eq: nabla (detDPsi detDPhi)=0}. Similarly, differentiating the identity \eqref{eq: detDPsi detDPhi=1} with respect to $t$ on gets \eqref{eq: dot (detDPsi detDPhi)=0}. Multiplying both sides of \eqref{eq: nabla (detDPsi detDPhi)=0} by $\dot{\Phi}(t,y)$ and using \eqref{eq: dot Psi}, one also deduces \eqref{eq: nabla (detDPsi detDPhi) dotPhi =0}. Finally, we prove \eqref{eq: dot detDPsi + div = 0}: by \eqref{assumptionPhi}, for a.e.\ $t\in[0,T]$ and $x\in {\overline{\Omega}_t}$ there holds
		\begin{equation*}
			\tr \left[ \frac{\d}{\d t} \left( D \Phi(t,\Psi (t,x)) \right) \right] = 0.
		\end{equation*}
		The above identity can be written in components as
		\begin{align*}
			\sum_{i,j=1}^{N} \partial_j \dot{\Phi}_i(t,\Psi(t,x)) \partial_i \Psi_j(t,x) &+ \sum_{i,j=1}^{N} \partial_j \Phi_i(t,\Psi(t,x)) \partial_i \dot{\Psi}_j(t,x) 
			\\
			&+ \sum_{i,j,k=1}^{N} \partial_k\partial_j \Phi_i(t,\Psi(t,x)) \dot{\Psi}_k(t,x) \partial_i \Psi_j(t,\Psi(t,x))=0.
		\end{align*}
		Setting $x = \Phi(t,y)$, we now get
		\begin{align*}
			\sum_{i,j=1}^{N} \partial_j \dot{\Phi}_i(t,y) \partial_i \Psi_j(t,\Phi(t,y)) &+ \sum_{i,j=1}^{N} \partial_j \Phi_i(t,y) \partial_i \dot{\Psi}_j(t,\Phi(t,y)) 
			\\
			&+ \sum_{i,j,k=1}^{N} \partial_k\partial_j \Phi_i(t,y) \dot{\Psi}_k(t,\Phi(t,y)) \partial_i \Psi_j(t,y)=0.
		\end{align*}
		Finally, we multiply the previous equality by $\det D \Phi(t,y)$ and apply the following Jacobi identity
		\begin{equation*}
			\partial_t\det M(t)=\det M(t)\tr\left[M(t)^{-1}\partial_t M(t)\right],
		\end{equation*}
		with $M(t)=D\Phi(t,y)$. Thus we deduce \eqref{eq: dot detDPsi + div = 0}.
	\end{proof}
	
	Given a weak solution $u$ of problem \eqref{eq:u}, we now consider the auxiliary function
	\begin{subequations}
	\begin{equation}\label{def eq v}
		v(t,y) := u(t, \Phi(t,y)), \quad \text{ for all } (t,x) \in [0,T]\times {\overline{\Omega}_0}.
	\end{equation}
	Equivalently,
	\begin{equation}\label{def eq u}
		u(t,x) = v(t, \Psi(t,x))  \quad \text{ for all } (t,x)\in \overline{\mc O}.
	\end{equation}
\end{subequations}
	This change of variables yields to the following problem with fixed domain:
	\begin{equation}\label{eq:v}
		\begin{cases}
			\ddot{v} - \div (B \nabla v) + a \cdot \nabla v - 2b \cdot \nabla \dot{v} = g , &\text{in } (0,T) \times \Omega_0,\\
			v=0, &\text{in }(0,T)\times\partial\Omega_0,\\
			v(0)=v_0, \\
			\dot{v}(0)=v_1,
		\end{cases}
	\end{equation}
	whose coefficients are given by
	\begin{subequations}\label{expcof}
	\begin{align}
		\label{eq:B}
		B(t,y) &:= D \Psi(t,\Phi(t,y)) D \Psi(t,\Phi(t,y))^T - \dot{\Psi}(t,\Phi(t,y)) \otimes  \dot{\Psi}(t,\Phi(t,y)),
		\\
		\label{eq:a}
		a(t,y) &:= - \left\{ B(t,y)^T \nabla \det D\Phi(t,y) +\partial_t \Big[ b(t,y) \det D\Phi(t,y) \Big]\right\} \,\det D\Psi(t,\Phi(t,y)),
		\\
		\label{eq:b}
		b(t,y) &:= - \dot{\Psi}(t,\Phi(t,y)),
	\end{align}	
the forcing term is
\begin{equation}\label{eq:g}
	g(t,y) := f(t,\Phi(t,y)),
\end{equation}
	and the initial data are defined by
	\begin{equation}\label{datav}
		v_0 := u_0, \quad v_1 := u_1 + \dot{\Phi}(0,\cdot) \cdot \nabla u_0 .
	\end{equation}
\end{subequations}
	\begin{rem}\label{rem:A3}
		If the equation under study is \eqref{eq:uA}, the only change in the new coefficients is given by
		\begin{equation*}
			B(t,y) = D \Psi(t,\Phi(t,y)) A(t,\Phi(t,y)) D \Psi(t,\Phi(t,y))^T - \dot{\Psi}(t,\Phi(t,y)) \otimes  \dot{\Psi}(t,\Phi(t,y)).
		\end{equation*}
		We also refer to \cite[Equation (2.29)]{DMLuc} for a comparison.
	\end{rem}
	The following proposition shows the regularity of the new data.
	
	\begin{prop}\label{propregularitycoeff}
		 Assume \eqref{hypsetreg}, \eqref{eq:data} and let $\Phi,\Psi$ be as in \eqref{assumptionsdiff} and satisfy \ref{H1}. Let relations \eqref{expcof} hold. Then
		\begin{subequations}\label{regcof}
		\begin{align}
			& B \in C^{0,1}([0,T] \times {\overline{\Omega}_0}; \R^{N\times N}_{\rm sym}),\label{regcofB}
			\\
			& a \in L^{\infty}([0,T] \times {\overline{\Omega}_0}; \R^N),\label{regcofa}
			\\
			& b \in C^{0,1}([0,T] \times {\overline{\Omega}_0};\R^N),\label{regcofb}\\
			&g\in L^2((0,T)\times \Omega_0),\\
			&v_0\in H^1_0(\Omega_0),\quad v_1\in L^2(\Omega_0).
		\end{align}
	\end{subequations}
		Moreover, if also \ref{H2} is satisfied, then $B$ is uniformly elliptic, i.e., there exists a positive costant $c_B>0$ such that for every $(t,y) \in [0,T]\times {\overline{\Omega}_0}$ one has
		\begin{equation}\label{coerB}
			(B(t,y) w) \cdot w \geq c_B |w|^2, \quad \text{ for all } w \in \R^N.
		\end{equation}
	If in addition $f\in H^1(\mc O)$ and \ref{H1'} is fulfilled, then there hold
	\begin{subequations}\label{moreregcof}
\begin{align}
	& B \in C^{1,1}([0,T] \times {\overline{\Omega}_0}; \R^{N\times N}_{\rm sym}),\label{moreregB}
	\\
	& a \in C^{0,1}([0,T] \times {\overline{\Omega}_0};\R^N),\label{morerega}
	\\
	& b \in C^{1,1}([0,T] \times {\overline{\Omega}_0};\R^N),\label{moreregb}\\
	&  g\in H^1(0,T;L^2(\Omega_0)).\label{moreregg}
\end{align}
	
\end{subequations}
	\end{prop} 
	
	\begin{proof}
	 Regularity properties \eqref{regcof} and \eqref{moreregcof} directly follow from the explicit expressions \eqref{expcof} together with \ref{H1} and \ref{H1'}, respectively. For the ellipticity property \eqref{coerB} we refer to \cite[Lemma B.3]{LazMolSol}.
	\end{proof}
	
	To deal with problem \eqref{eq:v} we introduce two equivalent notions of solution (see Proposition~\eqref{propequivweaksolv}), whose terminology is consistent with the one introduced in \cite{DMToad}. Definition~\ref{defweaksolv} is the analogue of Definition~\ref{defweaksolu} in the current context, while Definition~\ref{defweakspacesolv} does not involve integration by parts in time (as classical in the analysis of hyperbolic problems, see for instance the textbooks \cite{Evans,LionsMage}). The first notion is useful to show the equivalence between problem \eqref{eq:u} in a moving domain and problem \eqref{eq:v} in a fixed domain. The second notion is more suited to the Galerkin method and will be used in Section~\ref{sec:hyperbolic} to obtain higher regularity.
	
	\begin{defin}\label{defweaksolv}
		We say that $v\colon [0,T] \times {\overline{\Omega}_0} \to \R$ is a weak solution of problem \eqref{eq:v} with data \eqref{regcof} if
		\begin{enumerate}
			\item[(i)] $v \in L^2(0,T; H^1_0(\Omega_0))$ and $\dot{v} \in L^2(0,T; L^2(\Omega_0))$;
			\item[(ii)] $v(0)=v_0$ in the sense of $C^0([0,T]; L^2(\Omega_0))$ and $\dot{v}(0)=v_1$ in the sense of $C^0([0,T]; H^{-1}(\Omega_0))$;
			\item[(iii)]  $v$ satisfies
			\begin{equation}\label{def eq v weak sol}
				\begin{aligned}
					&\quad-\int_{0}^{T} \langle \dot{v}(t), \dot{\xi}(t) \rangle_{L^2(\Omega_0)} \d t + \int_{0}^{T} \langle B(t) \nabla v(t), \nabla \xi(t) \rangle_{L^2(\Omega_0)}\d t 
					\\
					&\quad+ \int_{0}^{T} \langle a(t) \cdot \nabla v(t), \xi(t) \rangle_{L^2(\Omega_0)}\d t
					+2 \int_{0}^{T}\langle \dot{v}(t), \div(b(t) \xi(t)) \rangle_{L^2(\Omega_0)}\d t
					\\
					&	= \int_{0}^{T} \langle g(t), \xi(t) \rangle_{L^2(\Omega_0)}\d t,
				\end{aligned}
			\end{equation}
			for every $\xi \in L^2(0,T; H^1_0(\Omega_0)) \cap H^1_0(0,T; L^2(\Omega_0))$.
		\end{enumerate}
	\end{defin}

	\begin{defin}\label{defweakspacesolv}
		We say that $v\colon[0,T] \times {\overline{\Omega}_0} \to \R$ is a strong-weak solution of problem \eqref{eq:v}  with data \eqref{regcof} if
		\begin{enumerate}
			\item[(i)] $v \in L^2(0,T; H^1_0(\Omega_0))$, $\dot{v} \in L^2(0,T; L^2(\Omega_0))$, and $\ddot{v} \in L^2(0,T; H^{-1}(\Omega_0))$;
			\item[(ii)] $v(0)=v_0$ in the sense of $C^0([0,T]; L^2(\Omega_0))$ and $\dot{v}(0)=v_1$ in the sense of $C^0([0,T]; H^{-1}(\Omega_0))$;
			\item[(iii)] $v$ satisfies 
			\begin{equation}\label{def eq v weak space sol}
				\begin{aligned}
					&\langle\ddot{v}(t), \phi\rangle _{H^1_0(\Omega_0)} + \langle B(t) \nabla v(t), \nabla \phi \rangle_{L^2(\Omega_0)}
					+ \langle a(t) \cdot \nabla v(t), \phi \rangle_{L^2(\Omega_0)}
					+2 \langle \dot{v}(t), \div(b(t) \phi) \rangle_{L^2(\Omega_0)}
					\\
					=& \langle g(t), \phi \rangle_{L^2(\Omega_0)},
				\end{aligned}
			\end{equation}
			for a.e. $t \in [0,T]$ and for every $\phi \in H^1_0(\Omega_0)$.
		\end{enumerate}
	\end{defin}
	The next proposition shows the equivalence between the two notions of solution just introduced.
	\begin{prop}\label{propequivweaksolv}
		Assume \eqref{regcof}. Then a function $v$ is a weak solution of problem \eqref{eq:v} in the sense of Definition \ref{defweaksolv} if and only if it is a strong-weak solution in the sense of Definition \ref{defweakspacesolv}.
	\end{prop}
	
	\begin{proof}
		First assume that $v$ is a strong-weak solution and fix  a function $\xi \in L^2(0,T; H^1_0(\Omega_0)) \cap H^1_0(0,T; L^2(\Omega_0)) $. Then $\xi(t) \in H^1_0(\Omega_0)$ for almost every $t \in [0,T]$ and \eqref{def eq v weak space sol} holds with $\phi = \xi(t)$. Integrating by parts in time we obtain that
		\begin{align*}
			\int_{0}^{T} \langle g(t), \xi(t) \rangle_{L^2(\Omega_0)}\d t = & -\int_{0}^{T} \langle\dot{v}(t), \dot{\xi}(t))\rangle_{L^2(\Omega_0)} \d t + \int_{0}^{T} \langle B(t) \nabla v(t), \nabla \xi(t) \rangle_{L^2(\Omega_0)}\d t 
			\\
			& + \int_{0}^{T} \langle a(t) \cdot \nabla v(t), \xi(t) \rangle_{L^2(\Omega_0)}\d t
			+2 \int_{0}^{T}\langle \dot{v}(t), \div(b(t) \xi(t)) \rangle_{L^2(\Omega_0)}\d t,
		\end{align*}
		and so we conclude that $v$ is a weak solution.
		
		We now prove the reverse implication. Let $v$ be a weak solution; we first prove that $\ddot{v}$ belongs to $L^2(0,T;H^{-1}(\Omega_0))$. Since $\dot{v} \in L^2(0,T;L^2(\Omega_0))$, a priori we know that $\ddot{v} \in H^{-1}(0,T;L^2(\Omega_0))$ as a distributional derivative. By definition it acts in the following way:
		\begin{equation*}
			\langle\ddot{v}, \xi \rangle_{H^1_0(0,T;L^2(\Omega_0))} = -\int_{0}^{T} \langle \dot{v}(t), \dot{\xi}(t) \rangle_{L^2(\Omega_0)} \d t \qquad \text{ for all } \xi \in H^1_0(0,T;L^2(\Omega_0)).
		\end{equation*}
		We now fix $\xi \in L^2(0,T;H^1_0(\Omega_0)) \cap H^1_0(0,T;L^2(\Omega_0))$, hence \eqref{def eq v weak sol} states that
		\begin{equation}\label{eq:equation}
		\begin{aligned}
			-\int_{0}^{T} \langle \dot{v}(t), \dot{\xi}(t) \rangle_{L^2(\Omega_0)} \d t =& -\int_{0}^{T} \langle B(t) \nabla v(t), \nabla \xi(t) \rangle_{L^2(\Omega_0)}\d t 
			- \int_{0}^{T} \langle a(t) \cdot \nabla v(t), \xi(t) \rangle_{L^2(\Omega_0)}\d t
			\\
			&-2 \int_{0}^{T}\langle \dot{v}(t), \div(b(t) \xi(t)) \rangle_{L^2(\Omega_0)}\d t  
			+ \int_{0}^{T} \langle g(t), \xi(t) \rangle_{L^2(\Omega_0)}\d t.
		\end{aligned}
\end{equation}
		Due to \eqref{regcof}, by developing the divergence term
		\begin{equation*}
			\div(b(t)\xi(t)) = \div(b(t)) \, \xi(t) + b(t) \cdot \nabla \xi(t),    
		\end{equation*}
		we conclude that the following inequalities hold:
		\begin{subequations}\label{ineq:ineq}
		\begin{align}
			\label{ineq ddot v regul 1}
			\left| \int_{0}^{T} \langle B(t) \nabla v(t), \nabla \xi(t) \rangle_{L^2(\Omega_0)}\d t \right| &\leq \norm{B}_{L^\infty((0,T) \times \Omega_0)} \norm{\nabla v}_{L^2(0,T;L^2(\Omega_0))} \norm{\nabla\xi}_{L^2(0,T;L^2(\Omega_0))},
			\\
			\label{ineq ddot v regul 2}
			\left| \int_{0}^{T} \langle a(t) \cdot \nabla v(t), \xi(t) \rangle_{L^2(\Omega_0)}\d t \right| &\leq \norm{a}_{L^\infty((0,T) \times \Omega_0)}
			\norm{\nabla v}_{L^2(0,T;L^2(\Omega_0))} \norm{\xi}_{L^2(0,T;L^2(\Omega_0))},
			\\
			\nonumber
			\left| \int_{0}^{T}\langle \dot{v}(t), \div(b(t) \xi(t)) \rangle_{L^2(\Omega_0)}\d t  \right| &\leq
			\left(\norm{\div(b)}_{L^\infty((0,T) \times \Omega_0)} + \norm{b}_{L^\infty((0,T) \times \Omega_0)}\right)
			\\
			\label{ineq ddot v regul 3}
			&\quad \times\norm{\dot{v}}_{L^2(0,T;L^2(\Omega_0))}
			\norm{\xi}_{L^2(0,T;H^1_0(\Omega_0))},
			\\
			\label{ineq ddot v regul 4}
			\left| \int_{0}^{T} \langle g(t), \xi(t) \rangle_{L^2(\Omega_0)}\d t \right| &\leq \norm{g}_{L^2(0,T;L^2(\Omega_0))} \norm{\xi}_{L^2(0,T;L^2(\Omega_0))}.
		\end{align}
\end{subequations}
		Hence there exists a constant $C>0$ such that
		\begin{equation*}
			|\langle\ddot{v}, \xi\rangle_{H^1_0(0,T;L^2(\Omega_0))}| \leq C \norm{\xi}_{L^2(0,T;H^1_0(\Omega_0))}.
		\end{equation*}
	By density of $L^2(0,T;H^1_0(\Omega_0)) \cap H^1_0(0,T;L^2(\Omega_0))$ in $L^2(0,T;H^1_0(\Omega_0))$, we conclude that $\ddot{v}$ is in $L^2(0,T; H^{-1}(\Omega_0))$.
		
	Employing an integration by parts in time in \eqref{eq:equation}, which now is allowed, we then deduce that
		\begin{equation}\label{eq v weak sol prop equiv}
			\begin{split}
				\int_{0}^{T} \langle g(t), \xi(t) \rangle_{L^2(\Omega_0)}\d t = & \int_{0}^{T} \langle\ddot{v}(t), \xi(t)\rangle_{H^1_0(\Omega_0)} \d t + \int_{0}^{T} \langle B(t) \nabla v(t), \nabla \xi(t) \rangle_{L^2(\Omega_0)}\d t
				\\
				& + \int_{0}^{T} \langle a(t) \cdot \nabla v(t), \xi(t) \rangle_{L^2(\Omega_0)}\d t
				+2 \int_{0}^{T}\langle \dot{v}(t), \div(b(t) \xi(t)) \rangle_{L^2(\Omega_0)}\d t,
			\end{split}
		\end{equation}
		for every $\xi \in L^2(0,T;H^1_0(\Omega_0))$. Now let $\{ \phi_n \}_{n \in \N} \subseteq H^1_0(\Omega_0)$ be a countable dense subset of $H^1_0(\Omega_0)$, and consider, for $n \in \N$ and $t\in [0,T]$, the following properties
		\begin{subequations}\label{list}
		\begin{align}
			&\bullet\, \lim\limits_{h \to 0^+} \int_{t}^{t+h}\langle\ddot{v}(s),\phi_n\rangle_{H^1_0(\Omega_0)} \d s =  \langle\ddot{v}(t), \phi_n\rangle_{H^1_0(\Omega_0)},
			\label{eqincrementalquotients1}
			\\
			&\bullet\, \lim\limits_{h \to 0^+} \int_{t}^{t+h}\langle B(s) \nabla v(s), \nabla \phi_n \rangle_{L^2(\Omega_0)} \d s = \langle B(t) \nabla v(t), \nabla \phi_n \rangle_{L^2(\Omega_0)},
			\label{eqincrementalquotients2}
			\\
			&\bullet\, \lim\limits_{h \to 0^+} \int_{t}^{t+h}\langle a(s) \cdot \nabla v(s), \phi_n \rangle_{L^2(\Omega_0)}\d s =  \langle a(t) \cdot \nabla v(t), \phi_n \rangle_{L^2(\Omega_0)},
			\label{eqincrementalquotients3}
			\\
			&\bullet\, \lim\limits_{h \to 0^+} \int_{t}^{t+h}\langle \dot{v}(s), \div(b(s) \phi_n) \rangle_{L^2(\Omega_0)}\d s =  \langle \dot{v}(t), \div(b(t) \phi_n) \rangle_{L^2(\Omega_0)},
			\label{eqincrementalquotients4}
			\\
			&\bullet\, \lim\limits_{h \to 0^+} \int_{t}^{t+h}\langle g(s), \phi_n \rangle_{L^2(\Omega_0)}\d s =  \langle g(t), \phi_n \rangle_{L^2(\Omega_0)}.
			\label{eqincrementalquotients5}
		\end{align}
\end{subequations}
		For every $n \in \N$ we now define the following set
		\begin{equation*}
			A_n := \left\{ t \in [0,T] \,|\, \text{relations } \eqref{list} \text{ hold}
			\right\}.
		\end{equation*}
		By the regularity of $v$ and of the data, we have that $A_n$ has full measure; hence the set
		\begin{equation*}
			Z:= \bigcup_{n \in \N} Z_n, \quad\text{with }\quad Z_n := [0,T] \setminus A_n,
		\end{equation*}
		has null measure. By considering the functions
		\begin{equation*}
			\psi^h_n(s):= \frac{1}{h}\, \phi_n \,\chi_{[t,t+h]}(s) \quad \text{ for all } h>0, n \in \N, t \in [0,T] \setminus Z,
		\end{equation*}
		and testing \eqref{eq v weak sol prop equiv} by $\xi = \psi^h_n \in L^2(0,T;H^1_0(\Omega_0))$, we thus obtain
		\begin{align*}
			&\frac{1}{h}\int_{t}^{t+h}\langle\ddot{v}(s), \phi_n\rangle_{H^1_0(\Omega_0)} \d s + \frac{1}{h} \int_{t}^{t+h}\langle B(s) \nabla v(s), \nabla \phi_n \rangle_{L^2(\Omega_0)} \d s 
			\\
			&+ \frac{1}{h} \int_{t}^{t+h}\langle a(s) \cdot \nabla v(s), \phi_n \rangle_{L^2(\Omega_0)}\d s 
			+\frac{2}{h}\int_{t}^{t+h}\langle \dot{v}(s), \div(b(s) \phi_n) \rangle_{L^2(\Omega_0)}\d s  
			\\
			=& \frac{1}{h} \int_{t}^{t+h} \langle g(s), \phi_n \rangle_{L^2(\Omega_0)}\d t, \qquad\qquad \text{ for all } h>0, n \in \N, t \in [0,T] \setminus Z.
		\end{align*}
		Letting $h \to 0^+$, for $t \in [0,T]\setminus Z$ we now get
		\begin{align*}
			&\langle\ddot{v}(t), \phi_n\rangle _{H^1_0(\Omega_0)} + \langle B(t) \nabla v(t), \nabla \phi_n \rangle_{L^2(\Omega_0)}
			+ \langle a(t) \cdot \nabla v(t), \phi_n \rangle_{L^2(\Omega_0)}
			+2 \langle \dot{v}(t), \div(b(t) \phi_n) \rangle_{L^2(\Omega_0)}\\
			=& \langle g(t), \phi_n \rangle_{L^2(\Omega_0)},
		\end{align*}
		for every $n \in \N$. By the density of $\{ \phi_n \}_{n \in \N}$ in $H^1_0(\Omega_0)$, we finally conclude that $v$ is a strong-weak solution. Hence, the proof is complete.
	\end{proof}
	
	The main result of the section is contained in the following theorem, which states that problems \eqref{eq:u} and \eqref{eq:v} are actually equivalent.
	
	\begin{teo}\label{teoequivalenceweaksol}
			 Assume \eqref{hypsetreg}, \eqref{eq:data} and let $\Phi,\Psi$ be as in \eqref{assumptionsdiff} and satisfy \ref{H1}. Then $u$ is a weak solution of problem \eqref{eq:u} in the sense of Definition \ref{defweaksolu} if and only if the corresponding function $v$ defined as in \eqref{def eq v} is a weak solution of problem \eqref{eq:v} with data \eqref{expcof} in the sense of Definition \ref{defweaksolv}.
	\end{teo}
	
	\begin{proof}
		
		Let $u$ be a weak solution of problem \eqref{eq:u} and let $v$ be defined by the change of variables $\Phi$ as in \eqref{def eq v}. Let
		\begin{equation*}
			\xi \in L^2(0,T; H^1_0(\Omega_0)) \cap H^1_0(0,T; L^2(\Omega_0)).
		\end{equation*}
		Then consider the test function defined by
		\begin{equation*}
			\eta(t,x) = \xi(t,\Psi(t,x)) \det D \Psi(t,x) \quad \text{for a.e. } (t,x) \in \mathcal{O}.
		\end{equation*}
		We observe that by hypothesis \ref{H1}, clearly $\eta \in L^2(0,T;H^1_0(\Omega_t))$ with $\dot\eta \in L^2(0,T;L^2(\Omega_t))$ and $\eta(T)=\eta(0)=0$.
		By \eqref{def eq u} the following relations hold  for a.e. $(t,x) \in \mc O$:
		\begin{align*}
			&\bullet\,\dot{u}(t,x) = \dot{v}(t,\Psi(t,x)) + \nabla v (t,\Psi(t,x)) \cdot \dot{\Psi}(t,x),
			\\
			&\bullet\,\nabla u(t,x) = 
			D\Psi(t,x)^T\,\nabla v (t,\Psi(t,x)).
		\end{align*}
		Hence, by \eqref{def eq u weak sol} we get
		\begin{equation}\label{eq solweakvchangevar}
			\begin{aligned}
				&\underbrace{-\int_{0}^{T} \left\langle  \dot{v}(t,\Psi(t,\cdot)) + \nabla v (t,\Psi(t,\cdot)) \cdot \dot{\Psi}(t,\cdot)), \frac{\d}{\d t}  \Big[\xi(t,\Psi(t,\cdot)) \det D \Psi(t,\cdot) \Big] \right\rangle_{L^2(\Omega_t)} \, \d t}_{J_1}
				\\
				&\underbrace{+ \int_{0}^{T}  \left\langle D\Psi(t,\cdot)^T \,\nabla v (t,\Psi(t,\cdot)), \nabla \Big[\xi(t,\Psi(t,\cdot)) \det D \Psi(t,\cdot)\Big] \right\rangle_{L^2(\Omega_t)} \d t}_{J_2}
				\\
				=& \int_{0}^{T} \langle f(t,\cdot), \xi(t,\Psi(t,\cdot)) \det D \Psi(t,\cdot)\rangle_{L^2(\Omega_t)} \d t.
			\end{aligned}
		\end{equation}
		Then, by the change of variables $x= \Phi (t,y)$ and using the identity \eqref{eq: detDPsi detDPhi=1} of Lemma \ref{lemma identities Psi Phi}, we get 
		\begin{equation}\label{eq f=g}
			\int_{0}^{T} \langle f(t,\cdot), \xi(t,\Psi(t,\cdot)) \det D \Psi(t,\cdot) \rangle_{L^2(\Omega_t)} \d t = \int_{0}^{T} \langle g(t,\cdot), \xi(t,\cdot)\rangle_{L^2(\Omega_0)} \d t.
		\end{equation}
		In order to conclude the proof, we shall prove that $J_1+J_2$ coincides with the left-hand side of \eqref{def eq v weak sol}. We start by considering the term $J_2$. Expanding the term $\nabla \Big[\xi(t,\Psi(t,\cdot)) \det D \Psi(t,\cdot) \Big]$ yields to 
		\begin{equation*}
			\nabla \Big[\xi(t,\Psi(t,\cdot)) \det D \Psi(t,\cdot) \Big] = D \Psi(t,\cdot)^T \nabla \xi(t,\Psi(t,\cdot)) \det D \Psi(t,\cdot) +  \xi(t,\Psi(t,\cdot)) \, \nabla \det  D \Psi(t,\cdot).
		\end{equation*}
		Performing again the change of variables $x= \Phi(t,y)$, by \eqref{eq: detDPsi detDPhi=1} we obtain
		\begin{align*}
			J_2 =& \int_{0}^{T} \langle D\Psi(t,\Phi(t,\cdot))^T \nabla v (t,\cdot), D \Psi(t,\Phi(t,\cdot))^T \nabla \xi(t,\cdot) \rangle_{L^2(\Omega_0)} \d t
			\\
			& +  \int_{0}^{T} \langle D\Psi(t,\Phi(t,\cdot))^T \nabla v (t,\cdot),  \xi(t,\cdot) \, \nabla\det  D \Psi(t,\Phi(t,\cdot)) \det D\Phi(t,\cdot) \rangle_{L^2(\Omega_0)} \d t.
		\end{align*}
		Then, by \eqref{eq: nabla (detDPsi detDPhi)=0} in Lemma \ref{lemma identities Psi Phi} we deduce that  
		\begin{equation*}
			\begin{aligned}
				J_2 =& \int_{0}^{T} \langle D\Psi(t,\Phi(t,\cdot)) D\Psi(t,\Phi(t,\cdot))^T \nabla v (t,\cdot) , \nabla \xi(t,\cdot)  \rangle_{L^2(\Omega_0)} \d t
				\\
				&- \int_{0}^{T} \langle D\Psi(t,\Phi(t,\cdot)) D\Psi(t,\Phi(t,\cdot))^T \nabla v (t,\cdot) , \xi(t,\cdot) \, \nabla\det D\Phi(t,\cdot) \det  D \Psi(t,\Phi(t,\cdot)) \rangle_{L^2(\Omega_0)} \d t.
			\end{aligned}
		\end{equation*}
		We now consider the term $J_1$. Expanding the term $\frac{\d}{\d t}  \Big[\xi(t,\Psi(t,\cdot)) \det D \Psi(t,\cdot) \Big]$ yields to 
		\begin{align*}
			\frac{\d}{\d t}  \Big[\xi(t,\Psi(t,\cdot)) \det D \Psi(t,\cdot) \Big] &= 
			\dot{\xi}(t,\Psi(t,\cdot)) \det D\Psi(t,\cdot) 
			\\
			&\quad  + \nabla\xi (t,\Psi(t,\cdot))  \cdot \dot{\Psi}(t,\cdot) \det D\Psi(t,\cdot) + \xi(t,\Psi(t,\cdot)) \, \partial_t \det D\Psi(t,\cdot).
		\end{align*}
		Arguing as before, we obtain
		\begin{align*}
			\allowdisplaybreaks
			J_1 =& - \int_{0}^{T} \langle\dot{v}(t,\cdot), \dot{\xi} (t,\cdot)\rangle_{L^2(\Omega_0)} \d t -  \int_{0}^{T}  \langle \dot{v}(t,\cdot), \nabla\xi(t,\cdot) \cdot \dot{\Psi}(t, \Phi(t,\cdot)) \rangle_{L^2(\Omega_0)} \d t
			\\
			&-  \int_{0}^{T} \langle  \dot{v}(t,\cdot), \xi(t,\cdot) \, \partial_t \left[\det D \Psi(t,\cdot)\right](\Phi(t,\cdot)) \det D\Phi(t,\cdot) \rangle_{L^2(\Omega_0)} \d t 
			\\
			&- \int_{0}^{T} \langle \nabla v (t,\cdot) \cdot \dot{\Psi}(t,\Phi(t,\cdot)) , \dot{\xi}(t,\cdot) \rangle_{L^2(\Omega_0)} \d t
			\\
			&- \int_{0}^{T} \langle \nabla v (t,\cdot) \cdot \dot{\Psi}(t,\Phi(t,\cdot)), \nabla\xi(t,\cdot)  \cdot \dot{\Psi}(t,\Phi(t,\cdot)) \rangle_{L^2(\Omega_0)} \, \d t
			\\
			&- \int_{0}^{T} \langle \nabla v (t,\cdot) \cdot \dot{\Psi}(t,\Phi(t,\cdot)) ,\xi(t,\cdot) \,\partial_t [\det D \Psi(t,\cdot)](\Phi(t,\cdot))  \det D\Phi(t,\cdot) \rangle_{L^2(\Omega_0)} \d t.
		\end{align*} 
		Then, by the identity \eqref{eq: dot (detDPsi detDPhi)=0} of Lemma \ref{lemma identities Psi Phi} we deduce that
		\begin{equation}\label{eq T_1}
			\allowdisplaybreaks
			\begin{aligned}
				J_1 =& - \int_{0}^{T} \langle\dot{v}(t,\cdot), \dot{\xi} (t,\cdot)\rangle_{L^2(\Omega_0)} \d t -  \int_{0}^{T}  \langle \dot{v}(t,\cdot), \nabla\xi(t,\cdot) \cdot \dot{\Psi}(t, \Phi(t,\cdot)) \rangle_{L^2(\Omega_0)} \d t
				\\
				& + \int_{0}^{T} \langle \dot{v}(t,\cdot) , \xi(t,\cdot) \, \partial_t \det D\Phi(t,\cdot) \det D \Psi(t,\Phi(t,\cdot)) \rangle_{L^2(\Omega_0)} \d t 
				\\
				& + \int_{0}^{T} \langle \dot{v}(t,\cdot), \xi(t,\cdot) \, \nabla [\det D \Psi(t,\cdot)](\Phi(t,\cdot)) \cdot \dot{\Phi}(t,\cdot) \det D \Phi(t,\cdot)\rangle_{L^2(\Omega_0)} \d t 
				\\
				&- \int_{0}^{T} \langle \nabla v (t,\cdot) \cdot \dot{\Psi}(t,\Phi(t,\cdot)), \dot{\xi}(t,\cdot) \rangle_{L^2(\Omega_0)} \d t
				\\
				&- \int_{0}^{T} \langle \dot{\Psi}(t,\Phi(t,\cdot)) \otimes  \dot{\Psi}(t,\Phi(t,\cdot)) \nabla v (t,\cdot), \nabla\xi(t,\cdot) \rangle_{L^2(\Omega_0)}\d t
				\\
				& + \int_{0}^{T} \langle\nabla v (t,\cdot) \cdot \dot{\Psi}(t,\Phi(t,\cdot)), \xi(t,\cdot) \,\partial_t \det D\Phi(t,\cdot) \det D \Psi(t,\Phi(t,\cdot))  \rangle_{L^2(\Omega_0)} \d t
				\\
				& + \int_{0}^{T}\langle \nabla v (t,\cdot)\cdot \dot{\Psi}(t,\Phi(t,\cdot)), \xi(t,\cdot) \,\nabla [\det D \Psi(t,\cdot)](\Phi(t,\cdot)) \cdot \dot{\Phi}(t,\cdot) \det D \Phi(t,\cdot) \rangle_{L^2(\Omega_0)} \d t.
			\end{aligned}
		\end{equation}
		We now notice that, in force of the relation \eqref{eq: nabla (detDPsi detDPhi) dotPhi =0} of Lemma \ref{lemma identities Psi Phi}, we can rewrite the last summand of \eqref{eq T_1} as follows:
		\begin{align*}
			&\int_{0}^{T} \langle \nabla v (t,\cdot)\cdot \dot{\Psi}(t,\Phi(t,\cdot)), \xi(t,\cdot) \,\nabla [\det D \Psi(t,\cdot)](\Phi(t,\cdot)) \cdot \dot{\Phi}(t,\cdot) \det D \Phi(t,\cdot) \rangle_{L^2(\Omega_0)} \d t
			\\
			=& \int_{0}^{T} \langle\dot{\Psi}(t,\Phi(t,\cdot)) \otimes  \dot{\Psi}(t,\Phi(t,\cdot)) \nabla v(t,\cdot), \xi(t,\cdot) \,\nabla\det D \Phi(t,\cdot)\det D \Psi(t,\Phi(t,\cdot)) \rangle_{L^2(\Omega_0)}\d t.
		\end{align*}
		The fourth summand of \eqref{eq T_1} can be instead rewritten as
		\begin{align*}
			&\int_{0}^{T} \langle  \dot{v}(t,\cdot), \xi(t,\cdot) \, \nabla [\det D \Psi(t,\cdot)](\Phi(t,\cdot)) \cdot \dot{\Phi}(t,\cdot) \det D \Phi(t,\cdot) \rangle_{L^2(\Omega_0)} \d t 
			\\
			 =& \int_{0}^{T} \langle  \dot{v}(t,\cdot), \xi(t,\cdot) \, \dot{\Psi}(t,\Phi(t,\cdot)) \cdot \nabla\det D \Phi(t,\cdot) \det D \Psi(t,\Phi(t,\cdot)) \rangle_{L^2(\Omega_0)} \d t. 
		\end{align*}
		By using \eqref{eq: dot detDPsi + div = 0} in Lemma \ref{lemma identities Psi Phi} (keeping in mind also \eqref{eq: detDPsi detDPhi=1}), the sum of the third and the fourth summand in \eqref{eq T_1} gives
		\begin{equation}\label{eq T_1 third and fourth}
			\begin{aligned}
				& \int_{0}^{T} \langle  \dot{v}(t,\cdot), \xi(t,\cdot) \, \partial_t \det D\Phi(t,\cdot) \det D \Psi(t,\Phi(t,\cdot)) \rangle_{L^2(\Omega_0)} \d t 
				\\
				& + \int_{0}^{T} \langle \dot{v}(t,\cdot), \xi(t,\cdot) \,  \dot{\Psi}(t,\Phi(t,\cdot)) \cdot \nabla\det D \Phi(t,\cdot) \det D \Psi(t,\Phi(t,\cdot)) \rangle_{L^2(\Omega_0)} \d t 
				\\
				=& -  \int_{0}^{T} \left\langle \dot{v}(t,\cdot),  \xi(t,\cdot) \div \left(\dot{\Psi}(t, \Phi(t,\cdot))\right) \right\rangle_{L^2(\Omega_0)}\d t.
			\end{aligned}
		\end{equation}
		We finally consider the second and the fifth summand of \eqref{eq T_1}. We split the term $\dot{\Psi}(t,\Phi(t,y)) \dot{\xi}(t,y)$ into
		\begin{equation*}
			\dot{\Psi}(t,\Phi(t,y)) \dot{\xi}(t,y) = \frac{\d}{\d t}\left( \dot{\Psi}(t,\Phi(t,y)) \xi(t,y)\right) - \left( \frac{\d}{\d t} \dot{\Psi}(t,\Phi(t,y))\right) \xi(t,y),
		\end{equation*}
		and rewrite the term $\xi(t,y) \div \left(\dot{\Psi}(t, \Phi(t,y))\right)  + \nabla\xi(t,y) \cdot \dot{\Psi}(t, \Phi(t,y))$ (cf. the second summand of \eqref{eq T_1} and right-hand side of \eqref{eq T_1 third and fourth}) as 
		\begin{equation*}
			\xi(t,y) \div \left(\dot{\Psi}(t, \Phi(t,y))\right)  + \nabla\xi(t,y) \cdot \dot{\Psi}(t, \Phi(t,y)) = \div \left(\dot{\Psi}(t, \Phi(t,y))\xi(t,y)\right).
		\end{equation*}
		Finally, integrating by parts in space and in time, we conclude that
		\begin{equation*}
			\begin{aligned}
				&-  \int_{0}^{T} \langle  \dot{v} (t,\cdot), \nabla\xi(t,\cdot) \cdot \dot{\Psi}(t, \Phi(t,\cdot)) \rangle_{L^2(\Omega_0)} \d t
				-  \int_{0}^{T} \left\langle \dot{v}(t,\cdot), \xi(t,\cdot) \div \left(\dot{\Psi}(t, \Phi(t,\cdot))\right) \right\rangle_{L^2(\Omega_0)} \d t
				\\
				&- \int_{0}^{T} \langle \nabla v (t,\cdot) \cdot \dot{\Psi}(t,\Phi(t,\cdot)), \dot{\xi}(t,\cdot) \rangle_{L^2(\Omega_0)} \d t\\
				 =& 
				-  2\int_{0}^{T} \left\langle  \dot{v}(t,\cdot) , \div \left( \dot{\Psi}(t, \Phi(t,\cdot)) \xi(t,\cdot)\right) \right\rangle_{L^2(\Omega_0)}\!\!\!\!\!\d t
				+ \int_{0}^{T} \left\langle \nabla v (t,\cdot) \cdot \left(\frac{\d}{\d t} \dot{\Psi}(t,\Phi(t,\cdot))\right), \xi(t,\cdot) \right\rangle_{L^2(\Omega_0)}\!\!\!\!\!\d t.
			\end{aligned} 
		\end{equation*}
		Hence, recalling the expressions \eqref{eq:B}, \eqref{eq:a} and \eqref{eq:b}, we deduce that
		\begin{equation*}
			\begin{aligned}
				J_1 + J_2 =& -\int_{0}^{T} \langle \dot{v}(t), \dot{\xi}(t) \rangle_{L^2(\Omega_0)} \d t + \int_{0}^{T} \langle B(t) \nabla v(t), \nabla \xi(t) \rangle_{L^2(\Omega_0)}\d t 
				\\
				&+ \int_{0}^{T} \langle a(t) \cdot \nabla v(t), \xi(t) \rangle_{L^2(\Omega_0)}\d t
				+2 \int_{0}^{T}\langle \dot{v}(t), \div(b(t) \xi(t)) \rangle_{L^2(\Omega_0)}\d t,
			\end{aligned}
		\end{equation*}
		which, due to \eqref{eq solweakvchangevar} and \eqref{eq f=g}, yields the conclusion.
		
		Using the very same argument, one can prove the reverse implication, thus the proof is complete.
	\end{proof}
	
	\section{Nonautonomous hyperbolic equations}\label{sec:hyperbolic}
	In this section we focus on the hyperbolic problem \eqref{eq:v}, independently of its relation with the original problem given by the explicit formulas \eqref{expcof}. We employ the classical Galerkin method in order to obtain the higher regularity we are looking for. In Subsection~\ref{sec:firstestimate} we show the first basic estimates, which also provide a way to prove existence of strong-weak solutions to the problem under consideration. In Subsection~\ref{sec high reg} we refine such estimates, strenghtening the assumptions on the data, finally deducing more regularity for the solutions previously obtained.

	We thus consider problem \eqref{eq:v} with nonautonomous coefficients $B$, $a$, $b$, forcing term $g$ and initial data $v_0$, $v_1$ satisfying \eqref{regcof} and \eqref{coerB}. We tacitly assume throughout the whole section that 
	\begin{equation}\label{eq:Omega0}
		\text{the set $\Omega_{0}\subseteq \R^N$ is nonempty, open, bounded and with Lipschitz boundary.}
	\end{equation}
	Moreover, the definitions of weak solutions and strong-weak solutions are the ones given in Definitions~\ref{defweaksolv} and \ref{defweakspacesolv}, respectively. 
	
	\subsection{Galerkin approximation} The Galerkin method consists in projecting problem \eqref{eq:v} onto finite-dimensional spaces, and on finding uniform estimates on the lower dimensional problems which allow to retrieve information on the infinite-dimensional one.
	
	To this aim, let $\{w_k\}_{k \in \N} \subseteq H^2(\Omega_0) \cap H^1_0(\Omega_0)$ be the set of eigenfunctions of $-\Delta$ in $H^1_0(\Omega_0)$ normalized in $L^2(\Omega_0)$. It is a standard fact that they form an orthogonal basis of $H^1_0(\Omega_0)$ and an orthonormal basis of $L^2(\Omega_0)$. Furthermore, by the very definition, for every $k\in \N$ they fulfil
	\begin{equation}\label{eq:eigen}
		\langle \phi, w_k \rangle_{L^2(\Omega_0)}= \frac{\langle \nabla \phi , \nabla w_k\rangle_{L^2(\Omega_0)}}{\norm{\nabla w_k}^2_{L^2(\Omega_0)}},\quad\text{for all }\phi\in H^1_0(\Omega_0).
	\end{equation}	
	For every $m \in \N$, we seek functions $d^m_k \in H^2(0,T)$ such that the function defined by
\begin{equation}\label{eq def vm}
	v^m (t) := \sum_{k=1}^{m} d^m_k(t) w_k \in H^2(0,T; H^2(\Omega_0) \cap H^1_0(\Omega_0))
\end{equation}
satisfies for every $k= 1, \dots, m$ and for almost every $t\in [0,T]$ the finite-dimensional version of problem \eqref{eq:v}, namely
\begin{equation}\label{eq:pdevm}
	\begin{aligned}
		&\langle\ddot{v}^m(t), w_k\rangle
		_{H^1_0(\Omega_0)}+ \langle B(t) \nabla v^m(t), \nabla w_k \rangle_{L^2(\Omega_0)}
		+ \langle a(t) \cdot \nabla v^m(t), w_k \rangle_{L^2(\Omega_0)}
		-2 \langle b(t) \cdot\nabla \dot{v}^m(t), w_k \rangle_{L^2(\Omega_0)}\\
		=& \langle g(t), w_k \rangle_{L^2(\Omega_0)},
	\end{aligned}
\end{equation}
with initial conditions
\begin{subequations}
\begin{align}
	\label{d^m_k(0)}&d^m_k(0) =  \langle v_0, w_k \rangle_{L^2(\Omega_0)},
	\\
	\label{dot d^m_k(0)}&\dot{d}^m_k(0) = \langle v_1, w_k \rangle_{L^2(\Omega_0)}.
\end{align}
\end{subequations}
Comparing \eqref{eq def vm} and \eqref{eq:pdevm}, we obtain an auxiliary ordinary differential equation whose coefficients (time-dependent) are given by
\begin{alignat*}{3}
	&B_{lk}(t) := \langle B(t) \nabla w_l, \nabla w_k \rangle_{L^2(\Omega_0)},
	&&\quad a_{lk}(t) := \langle a(t)\cdot \nabla w_l, w_k \rangle_{L^2(\Omega_0)},&&
	\\
	&b_{lk}(t) := \langle b(t) \cdot \nabla w_l , w_k \rangle_{L^2(\Omega_0)},
	&&\quad g_{k}(t) := \langle g(t), w_k \rangle_{L^2(\Omega_0)},&&\, \text{for a.e.\ $t \in [0,T]$ and all $l,k \in \N$.}
\end{alignat*}
Then, by classical results on linear ordinary differential equations, for every $m \in \N$ there exists a unique m-tuple of functions
\begin{equation*}
	d^m = (d^m_1, \dots, d^m_m) \in (H^2(0,T))^m  ,  
\end{equation*} 
satisfying
\begin{equation}\label{eq:d}
	\begin{cases}\displaystyle
		\ddot{d}^m_k(t) -2 \sum_{l=1}^{m} b_{lk}(t)\, \dot{d}^m_l(t) + \sum_{l=1}^{m} \left(B_{lk}(t) +a_{lk}(t) \right) d^m_l(t) = g_k(t), \quad\text{for a.e. }  t \in [0,T],
		\\\displaystyle
		d^m_k(0) =  \langle v_0, w_k \rangle_{L^2(\Omega_0)},
		\\\displaystyle
		\dot{d}^m_k(0) = \langle v_1, w_k \rangle_{L^2(\Omega_0)},
	\end{cases}
\end{equation}
so that the corresponding function $v^m$ defined by \eqref{eq def vm} satisfies \eqref{eq:pdevm} for almost every $t \in [0,T]$.

\subsection{First estimates and existence}\label{sec:firstestimate}
The goal now is finding suitable uniform estimates on the functions $v^m$. We start with the following proposition.
\begin{prop}\label{propunifboundvm}
	Assume \eqref{regcof} and \eqref{coerB}. Then there exists a constant $C>0$ (independent of $m \in \N$) such that
	\begin{equation}\label{ineq: esssup vm}
		\sup_{0\leq t\leq T} \left(\norm{\dot{v}^m(t)}_{L^2(\Omega_0)}^2 +\norm{v^m(t)}_{H^1_0(\Omega_0)}^2\right) \leq C.
	\end{equation}
	Moreover,
	\begin{align}\label{eq:boundvdott}
		\ddot{v}^m \text{ is uniformly bounded in } L^2(0,T;H^{-1}(\Omega_0)).
	\end{align}
\end{prop}
\begin{proof}
	The proof is rather standard, see for instance \cite{Evans}. We however present it in detail, since similar estimates will be employed in the (more involved) proof of Proposition~\ref{teoH2regularityvm}.\par 
	Consider the equation \eqref{eq:pdevm} satisfied by $v^m$, multiply it by $\dot{d}^m_k(t)$ and sum from $k=1,\dots,m$. Fixing $t \in [0,T]$ and integrating with respect to time over $(0,t)$, we thus obtain 
	\begin{align*}
		&\quad\underbrace{\int_{0}^{t}\langle\ddot{v}^m(s), \dot{v}^m(s)\rangle _{H^1_0(\Omega_0)} \d s}_{J_1}+ \underbrace{\int_{0}^{t}\langle B(s) \nabla v^m(s), \nabla \dot{v}^m(s) \rangle_{L^2(\Omega_0)}\d s}_{J_2}
		\\
		&+\underbrace{\int_{0}^{t}\langle a(s) \cdot \nabla v^m(s), \dot{v}^m(s) \rangle_{L^2(\Omega_0)}\d s}_{J_3}
		-\underbrace{2 \int_{0}^{t}\langle b(s) \cdot\nabla \dot{v}^m(s), \dot{v}^m(s) \rangle_{L^2(\Omega_0)}\d s}_{J_4}\\
		=& \underbrace{\int_{0}^{t}\langle g(s), \dot{v}^m(s) \rangle_{L^2(\Omega_0)}\d s}_{J_5}.
	\end{align*}
	Let us consider each of the terms $J_1,\dots,J_5$ defined above. For the term $J_1$ we have
	\begin{equation*}
		\int_{0}^{t}\langle\ddot{v}^m(s), \dot{v}^m(s)\rangle _{H^1_0(\Omega_0)} \d s = \frac{1}{2} \norm{\dot{v}^m(t)}_{L^2(\Omega_0)}^2 -\frac{1}{2} \norm{\dot{v}^m(0)}_{L^2(\Omega_0)}^2.
	\end{equation*}
	We notice, moreover, that $\norm{\dot{v}^m(0)}_{L^2(\Omega_0)}^2 \leq \norm{v_1}_{L^2(\Omega_0)}^2$ by \eqref{dot d^m_k(0)}.
	
	As for the term $J_2$, by the symmetry of the matrix $B$ and integrating by parts in time, we have
	\begin{align*}
		\int_{0}^{t}\langle B(s) \nabla v^m(s), \nabla \dot{v}^m(s) \rangle_{L^2(\Omega_0)}\d s =&  \frac{1}{2} \langle B(t) \nabla v^m(t), \nabla v^m(t) \rangle_{L^2(\Omega_0)} - \frac{1}{2} \langle B(0) \nabla v^m(0), \nabla v^m(0) \rangle_{L^2(\Omega_0)}
		\\
		&-\frac{1}{2}\int_{0}^{t}\langle \dot{B}(s) \nabla v^m(s), \nabla v^m(s) \rangle_{L^2(\Omega_0)}\d s.
	\end{align*}
	Moreover, by \eqref{coerB} and \eqref{regcofB}, we deduce that 
	\begin{align*}
		&\bullet\,\frac{1}{2} \langle B(t) \nabla v^m(t), \nabla v^m(t) \rangle_{L^2(\Omega_0)} \geq \frac{c_B}{2} \norm{\nabla v^m(t)}^2_{L^2(\Omega_0)},
		\\
		&\bullet\, \frac{1}{2}\left| \langle B(0) \nabla v^m(0), \nabla v^m(0) \rangle_{L^2(\Omega_0)}\right| \leq \frac{1}{2} \|B(0)\|_{L^{\infty}(\Omega_0)}\norm{\nabla v^m(0)}^2_{L^2(\Omega_0)},
		\\
		&\bullet\,\left|\frac{1}{2}\int_{0}^{t}\langle \dot{B}(s) \nabla v^m(s), \nabla v^m(s) \rangle_{L^2(\Omega_0)}\d s\right| \leq \frac{1}{2}
		\|\dot{B}\|_{L^\infty((0,T)\times\Omega_0)}
		\int_{0}^{t} \norm{\nabla v^m(s)}^2_{L^2(\Omega_0)} \d s.
	\end{align*}
	We notice, in particular, that $\norm{\nabla v^m(0)}^2_{L^2(\Omega_0)} \leq \norm{\nabla v_0}^2_{L^2(\Omega_0)}$ by \eqref{eq:eigen} and \eqref{d^m_k(0)}.\par
	As for $J_3$, by \eqref{regcofa} we have
	\begin{equation*}
		\left|\int_{0}^{t}\langle a(s) \cdot \nabla v^m(s), \dot{v}^m(s) \rangle_{L^2(\Omega_0)}\d s\right| \leq \norm{a}_{L^\infty((0,T)\times\Omega)} \int_{0}^{t} \left( \norm{\dot{v}^m(s)}_{L^2(\Omega_0)}^2 + \norm{\nabla v^m(s)}_{L^2(\Omega_0)}^2 \right)\d s.
	\end{equation*}	
	Regarding $J_4$, we observe that for each $s \in (0,t)$ one has $\dot{v}^m(s) \in H^1_0(\Omega_0)$, and hence integrating by parts in space and exploiting \eqref{regcofb} we get
	\begin{align*}
		2\left| \int_{0}^{t}\langle b(s) \cdot\nabla \dot{v}^m(s), \dot{v}^m(s) \rangle_{L^2(\Omega_0)}\d s\right| =& \left|\int_{0}^{t}\langle b(s), \nabla|\dot{v}^m(s)|^2 \rangle_{L^1(\Omega_0)}\d s \right|
		\\
		=& \left|-\int_{0}^{t}\langle \div(b(s)), |\dot{v}^m(s)|^2 \rangle_{L^1(\Omega_0)}\d s\right| \\
		\leq &
		\norm{\div(b)}_{L^\infty((0,T)\times\Omega_0)} \int_{0}^{t}\norm{\dot{v}^m(s)}_{L^2(\Omega_0)}^2\d s.
	\end{align*}	
	Finally, for the term $J_5$, by Young's inequality we readily deduce that 
	\begin{equation*}
		\left|\int_{0}^{t}\langle g(s), \dot{v}^m(s) \rangle_{L^2(\Omega_0)}\d s\right| \leq \frac{1}{2}\norm{g}^2_{L^2(0,T;L^2(\Omega_0)}
		+ \frac{1}{2}\int_{0}^{t} \norm{\dot{v}^m(s)}_{L^2(\Omega_0)}^2 \d s.
	\end{equation*}
	By the estimates obtained for $J_1,\dots,J_5$, we deduce that there exist positive constants $c_1$ and $c_2$ such that for every $t \in (0,T)$ there holds
	\begin{equation*}
		\norm{\dot{v}^m(t)}_{L^2(\Omega_0)}^2 + \frac{c_B}{2} \norm{\nabla v^m(t)}_{L^2(\Omega_0)}^2 \leq c_1 + c_2 \int_{0}^{t} \left( \norm{\dot{v}^m(s)}_{L^2(\Omega_0)}^2 + \norm{\nabla v^m(s)}_{L^2(\Omega_0)}^2 \right) \d s.
	\end{equation*}
	By a classical Gr\"onwall argument together with Poincaré inequality, we deduce the esistence of a constant $C>0$ such that \eqref{ineq: esssup vm} in the statement holds true.
	
	In order to conclude the proof, it remains to prove \eqref{eq:boundvdott}. Fix $w \in H^1_0(\Omega_0)$, such that $\norm{w}_{H^1_0(\Omega_0)} \leq 1$. Now write $w= w^m_{(1)}+w^m_{(2)}$, where $w^m_{(1)} \in \mathrm{span} \{w_k\}_{k=1}^m$ and $\langle w^m_{(2)} , w_k\rangle_{L^2(\Omega_0)}= 0$ for every $k \in \{1,\dots,m\}$. Due to \eqref{eq:eigen} we have $\|w^m_{(1)}\|_{H^1_0(\Omega_0)} \leq 1$.
	Then by \eqref{eq def vm} and \eqref{eq:pdevm}, we obtain
	\begin{align*}
		\langle\ddot{v}^m(t), w\rangle_{H^1_0(\Omega_0)} = \langle \ddot{v}^m(t), w^m_{(1)}\rangle_{L^2(\Omega_0)} = & - \langle B(t) \nabla v^m(t), \nabla w^m_{(1)} \rangle_{L^2(\Omega_0)}
		- \langle a(t) \cdot \nabla v^m(t), w^m_{(1)} \rangle_{L^2(\Omega_0)}
		\\
		&+2 \langle b(t) \cdot\nabla \dot{v}^m(t), w^m_{(1)} \rangle_{L^2(\Omega_0)}
		+ \langle g(t), w^m_{(1)} \rangle_{L^2(\Omega_0)}
		\\
		= & - \langle B(t) \nabla v^m(t), \nabla w^m_{(1)} \rangle_{L^2(\Omega_0)}
		- \langle a(t) \cdot \nabla v^m(t), w^m_{(1)} \rangle_{L^2(\Omega_0)}
		\\
		& - 2 \langle \dot{v}^m(t), \div (b(t) w^m_{(1)}) \rangle_{L^2(\Omega_0)}
		+ \langle g(t), w^m_{(1)} \rangle_{L^2(\Omega_0)}.
	\end{align*}
	Then using again \eqref{regcof}, exploiting \eqref{ineq: esssup vm} and recalling that $\|w^m_{(1)}\|_{H^1_0(\Omega_0)} \leq 1$, the previous equation yields to 
	\begin{equation*}
		\int_{0}^{T} \norm{\ddot{v}^m(s)}^2_{H^{-1}(\Omega_0)} \d s \leq {C},
	\end{equation*}
	for a suitable positive constant ${C}$. We do not detail here the estimates, which are similar to \eqref{ineq:ineq} in Proposition \ref{propequivweaksolv}.
\end{proof}

As a simple corollary we deduce the following result.

\begin{teo}\label{cor}
	Assume \eqref{regcof} and \eqref{coerB}. Then there exists a strong-weak solution $v$ for problem \eqref{eq:v} in the sense of Definition \ref{defweakspacesolv} which fulfils
	\begin{equation}\label{regg}
	\begin{aligned}
		&v\in C^0_w([0,T];H^1_0(\Omega_0)),\\
		&\dot v \in C^0_w([0,T];L^2(\Omega_0)).
	\end{aligned}
\end{equation}
\end{teo}
\begin{proof}
	By the uniform bounds obtained in Proposition~\ref{propunifboundvm} one deduces the existence of a subsequence (not relabelled) such that
	\begin{equation*}
		v^m\wto v,\quad\text{ weakly in }L^2(0,T;H^1_0(\Omega_0))\cap H^1(0,T;L^2(\Omega_0))\cap H^2(0,T;H^{-1}(\Omega_0)).
	\end{equation*}
	By integrating \eqref{eq:pdevm} with respect to time, letting $m\to +\infty$ and recalling that $\{w_k\}_{k\in \N}$ is a basis of $H^1_0(\Omega_0)$, it is easy to conclude that the limit function $v$ is a strong-weak solution to problem \eqref{eq:v}. Regularity property \eqref{regg} follows by classical embeddings as in the proof of Theorem~\ref{teoexistencecylindermethod}.
\end{proof}

By slightly strenghtening the assumptions on the coefficients, a uniqueness result is also available. Notice that \eqref{eq:slightreg} below is implied by \eqref{morerega} and \eqref{moreregb}.

\begin{prop}\label{prop:unique}
	In addition to \eqref{regcof} and \eqref{coerB}, assume that 
	\begin{equation}\label{eq:slightreg}
	 a\in C^{0,1}([0,T]\times {\overline{\Omega}_0};\R^N)\quad\text{ and }\quad 	\div b\in C^{0,1}([0,T];L^\infty(\Omega_0)).
	\end{equation}
Then the strong-weak solution to problem \eqref{eq:v} is unique.
\end{prop}
\begin{proof}
	The proof is based on an argument introduced by Ladyzenskaya in \cite{Lady}. For details we refer to \cite[Theorem~{3.10}]{DMLuc}.
\end{proof}

For the sake of completeness, we also present a result regarding the energy balance for weak solutions to problem \eqref{eq:v}, which can be proved by following \cite[Proposition~3.11]{DMLuc}. Unfortunately this energy balance is not easily transferable to problem \eqref{eq:u} with moving domains.
\begin{prop}\label{prop energy balance}
	Assume \eqref{regcof} and \eqref{coerB} and let $v$ be a weak solution of problem \eqref{eq:v} satisfying \eqref{regg}. Then, for every $t \in [0,T]$ we have
	\begin{equation}\label{eq energy balance}
		\frac{1}{2} \|\dot{v}(t)\|^2_{L^2(\Omega_0)} + \frac{1}{2} \langle B(t) \nabla v(t), \nabla v(t)\rangle_{L^2(\Omega_0)}= \frac{1}{2} \|v_1\|^2_{L^2(\Omega_0)} + \frac{1}{2} \langle B(0) \nabla v_0, \nabla v_0\rangle_{L^2(\Omega_0)} + {\mathcal{R}}[v](t),
	\end{equation}
	where the remainder is given by
	\begin{equation*}\label{eq remainder energy balance}
		\begin{aligned}
			{\mathcal{R}}[v](t) =& \int_{0}^{t} \frac{1}{2}\langle \dot{B}(s) \nabla v(s), \nabla v(s) \rangle_{L^2(\Omega_0)} - \langle a(s) \cdot \nabla v(s), \dot{v}(s) \rangle_{L^2(\Omega_0)}\d s
			\\
			&+ \int_{0}^{t}\left(-\langle \div b(s) , |\dot{v}(s)|^2 \rangle_{L^1(\Omega_0)} + \langle g(s), \dot{v}(s) \rangle_{L^2(\Omega_0)}\right)\d s.
		\end{aligned}
	\end{equation*}
\end{prop}
	\begin{rem}\label{rmk:reg}
		The energy balance \eqref{eq energy balance} allows us to slightly increase the regularity of the solution $v$ obtained in Theorem~\ref{cor}, which actually fulfils
		\begin{equation*}
			\begin{aligned}
				&v\in C^0([0,T];H^1_0(\Omega_0)),\\
				&\dot v \in C^0([0,T];L^2(\Omega_0)).
			\end{aligned}
		\end{equation*}
	This easily follows by the weak lower semicontinuity (with respect to $t$) of the left-hand side of \eqref{eq energy balance} (we recall \eqref{regcofB} and \eqref{coerB}) combined with the continuity of the right-hand side, which indeed yields convergence of the norms.
	\end{rem}

	\subsection{Further estimates and higher regularity}\label{sec high reg}
	We now improve previous estimates assuming that the data of the problem satisfy the stronger assumptions \eqref{moreregcof}. As a byproduct we deduce more regularity for strong-weak solutions to problem \eqref{eq:v}. 
	As for the initial data we require:
	\begin{equation}\label{regulardata}
		v_0 \in H^2(\Omega_0)\cap H^1_0(\Omega_0),\quad\text{ and }\quad v_1 \in H^1_0(\Omega_0).
	\end{equation}
We also need to assume that
	\begin{equation}\label{setregular}
		\Omega_0 \text{ is convex or of class }C^2.
	\end{equation}	
The latter assumption is used in Lemma \ref{lemmagrisvard} below.

\begin{rem}
In the next section we will apply our results to a family of moving domains $\Omega_t$ satisfying \eqref{hypsetreg} and \eqref{setregular}. From this point of view,
it may be surprising to notice that the convexity assumption on $\Omega_0$ allows one to circumvent other regularity assumptions on the domains. In particular, one may consider a family $\Omega_t$ of merely Lipschitz, nonconvex domains, such that they can be mapped into a single convex domain through changes of variables as in \eqref{assumptionsdiff}, satisfying \ref{H1} and \ref{H2}. 
Unfortunately it is difficult to characterize the class of sets fulfilling such property, however this method gives in principle the possibility to deal with irregular domains in concrete cases.
\end{rem}

	\begin{rem}\label{remhigherregcoeff}
	Under these stronger assumptions, the functions $d^m_k$ solving \eqref{eq:d} are of class $H^3(0,T)$, so $v^m$ belongs to $H^3(0,T;H^2(\Omega_0)\cap H^1_0(\Omega_0))$.
	\end{rem}
	
	We start by stating a lemma on elliptic regularity which will be used in Proposition~\ref{teoH2regularityvm} below. 
	\begin{lemma}\label{lemmagrisvard}
		Assume \eqref{setregular} and let $\widetilde{B} \in C^{1,1}({\overline{\Omega}_0}; \R^{N\times N}_{\rm sym})$ be elliptic with ellipticity constant $c_{\widetilde{B}}$. Then there exists a positive constant $\widetilde{D}$, depending only on $\Omega_0, c_{\widetilde{B}}$, and $\|\widetilde{B}\|_{C^{1,1}({\overline{\Omega}_0}; \R^{N\times N}_{\rm sym})}$, such that
		\begin{equation*}
			 \widetilde{D} \norm{\div (\widetilde{B}\nabla z)}_{L^2(\Omega_0)} \geq \norm{z}_{H^2(\Omega_0)}  
\quad \text{for all } z \in H^2(\Omega_0)\cap H^1_0(\Omega_0).
		\end{equation*}
	\end{lemma}
	
	\begin{proof}
		 This is a classical result on elliptic regularity and can be proved through the difference quotient technique, as e.g. in \cite{Grisvard}. In particular, the case of $\Omega_0$ of class $C^2$ is contained in \cite[Sect. 2.3-2.4]{Grisvard}. Furthermore, if $\Omega_0$ is convex \emph{and} of class $C^2$, we stress that the constant $\widetilde{D}$ may be chosen in such a way that the dependence on $\Omega_0$ only involves its diameter (this can be found in \cite[Sect. 3.1]{Grisvard}). The dependence just on the diameter allows one to extend the result to the case of $\Omega_0$ merely convex without further reguarity of the boundary, by a standard method of approximation by convex $C^2$-subdomains (see \cite[Sect. 3.2]{Grisvard}).  
	\end{proof}
	
	With this tool at disposal, we are now in a position to deduce higher uniform estimates for the functions $v^m$. 
	
	\begin{prop}\label{teoH2regularityvm}
		Assume \eqref{coerB}, \eqref{moreregcof}, \eqref{setregular} and \eqref{regulardata}. Then, there exists a constant $D>0$ (independent of $m \in \N$) such that
		\begin{equation*}
			\sup_{0\leq t\leq T} \left(\norm{\ddot{v}^m(t)}_{L^2(\Omega_0)}^2 +\norm{\dot{v}^m(t)}_{H^1_0(\Omega_0)}^2 +  \norm{v^m(t)}_{H^2(\Omega_0)}^2\right) \leq D.
		\end{equation*}
	\end{prop}
	
	\begin{proof}
		Defining $V^m:= \dot{v}^m$ and recalling Remark~\ref{remhigherregcoeff} we know that
		\begin{equation*}
			V^m\in H^2(0,T;H^2(\Omega_0)\cap H^1_0(\Omega_0)). 
		\end{equation*} 
	By differentiating \eqref{eq:pdevm} with respect to time, we hence obtain
		\begin{equation}\label{eq:pdeVm}
			\begin{aligned}
				&\langle\ddot{V}^m(t), w_k\rangle _{H^1_0(\Omega_0)} + \langle \dot{B}(t) \nabla v^m(t), \nabla w_k \rangle_{L^2(\Omega_0)} + \langle B(t) \nabla V^m(t), \nabla w_k \rangle_{L^2(\Omega_0)} 
				\\
				& + \langle \dot{a}(t) \cdot \nabla v^m(t), w_k \rangle_{L^2(\Omega_0)} 
				+ \langle a(t) \cdot \nabla V^m(t), w_k \rangle_{L^2(\Omega_0)}
				-2 \langle \dot{b}(t) \cdot\nabla V^m(t), w_k \rangle_{L^2(\Omega_0)}  
				\\
				&- 2 \langle b(t) \cdot\nabla \dot{V}^m(t), w_k \rangle_{L^2(\Omega_0)}
				= \langle \dot{g}(t), w_k \rangle_{L^2(\Omega_0)}. 
			\end{aligned}
		\end{equation}
		Now multiply \eqref{eq:pdeVm} by $\ddot{d}^m_k(t)$ and sum from $k=1,\dots, m$. Fixing $t \in [0,T]$ and integrating with respect to time over $(0,t)$, we obtain 
		\begin{equation*}
			\begin{aligned}
				&\underbrace{\int_{0}^{t}\langle\ddot{V}^m(s), \dot{V}^m(s)\rangle _{H^1_0(\Omega_0)} \d s}_{J_1}+\underbrace{\int_{0}^{t}\langle \dot{B}(s) \nabla v^m(s), \nabla \dot{V}^m(s) \rangle_{L^2(\Omega_0)}\d s}_{J_2} 
				\\
				&+\underbrace{\int_{0}^{t}\langle B(s) \nabla V^m(s), \nabla \dot{V}^m(s) \rangle_{L^2(\Omega_0)}\d s}_{J_3}+
				\underbrace{\int_{0}^{t}\langle \dot{a}(s) \cdot \nabla v^m(s), \dot{V}^m(s) \rangle_{L^2(\Omega_0)} \d s}_{J_4}
				\\
				&+\underbrace{\int_{0}^{t}\langle a(s) \cdot \nabla V^m(s), \dot{V}^m(s) \rangle_{L^2(\Omega_0)}\d s}_{J_5}
				-\underbrace{2 \int_{0}^{t}\langle \dot{b}(s) \cdot\nabla V^m(s), \dot{V}^m(s) \rangle_{L^2(\Omega_0)}\d s}_{J_6}
				\\
				&-\underbrace{ 2 \int_{0}^{t}\langle b(s) \cdot\nabla \dot{V}^m(s), \dot{V}^m(s) \rangle_{L^2(\Omega_0)}\d s}_{J_7}
				= \underbrace{\int_{0}^{t}\langle \dot{g}(s), \dot{V}^m(s) \rangle_{L^2(\Omega_0)}\d s}_{J_8}. 
			\end{aligned}
		\end{equation*}
		Next we estimate each of the terms in the previous identity. As for $J_1$, the following holds:
		\begin{equation*}
			\int_{0}^{t}\langle \ddot{V}^m(s), \dot{V}^m(s)\rangle _{H^1_0(\Omega_0)} \d s = \frac{1}{2} \|\dot{V}^m(t)\|_{L^2(\Omega_0)}^2 -\frac{1}{2} \|\dot{V}^m(0)\|_{L^2(\Omega_0)}^2.
		\end{equation*}
		We claim that 
		\begin{equation}\label{eq:claim}
			\begin{aligned}
					\|\dot{V}^m(0)\|_{L^2(\Omega_0)}&\le \|g(0)\|_{L^2(\Omega_0)}+\|a(0)\|_{L^\infty(\Omega_0)}\|\nabla v_0\|_{L^2(\Omega_0)} +2\|b(0)\|_{L^\infty(\Omega_0)}\|\nabla v_1\|_{L^2(\Omega_0)}\\
					&\quad\,+\|DB(0)\|_{L^\infty(\Omega_0)}\|\nabla v_0\|_{L^2(\Omega_0)}+C\|B(0)\|_{L^\infty(\Omega_0)}\| v_0\|_{H^2(\Omega_0)}.
			\end{aligned}
		\end{equation}
		We assume for the moment the claim is true and we continue by estimating $J_2$: integrating by parts in time, we have
		\begin{align*}
			&\int_{0}^{t}\langle \dot{B}(s) \nabla v^m(s), \nabla \dot{V}^m(s) \rangle_{L^2(\Omega_0)}\d s 
			\\
			=&  \langle \dot{B}(t) \nabla v^m(t), \nabla V^m(t) \rangle_{L^2(\Omega_0)} - \langle \dot{B}(0) \nabla v^m(0), \nabla V^m(0) \rangle_{L^2(\Omega_0)}
			\\
			&-\int_{0}^{t}\langle \ddot{B}(s) \nabla v^m(s), \nabla V^m(s) \rangle_{L^2(\Omega_0)}\d s -\int_{0}^{t}\langle \dot{B}(s) \nabla V^m(s), \nabla V^m(s) \rangle_{L^2(\Omega_0)}\d s.
		\end{align*}
		Furthermore, by the uniform bound for $\nabla v^m$ provided in Proposition \ref{propunifboundvm}, by \eqref{moreregB} and by Young's weighted inequality, the following estimates hold:
		\begin{align}
		&\begin{aligned}\label{epsilon}
			\bullet\,\left|\langle \dot{B}(t) \nabla v^m(t), \nabla V^m(t) \rangle_{L^2(\Omega_0)}\right| &\leq \|\dot{B}\|_{L^\infty((0,T)\times{\Omega_0})} \|\nabla v^m(t)\|_{L^2(\Omega_0)}  \|\nabla V^m(t)\|_{L^2(\Omega_0)} 
			\\
			&\leq \frac C2\|\dot{B}\|_{L^\infty((0,T)\times{\Omega_0})}\left( \frac{1}{\varepsilon} + \varepsilon  \|\nabla V^m(t)\|^2_{L^2(\Omega_0)}\right), \qquad \text{ for all } \varepsilon>0, 
		\end{aligned}	\\
		&\begin{aligned}
			\bullet\,\left|\langle \dot{B}(0) \nabla v^m(0), \nabla V^m(0) \rangle_{L^2(\Omega_0)}\right| &\leq \|\dot{B}(0)\|_{L^\infty({\Omega_0})} \|\nabla v^m(0)\|_{L^2(\Omega_0)}  \|\nabla V^m(0)\|_{L^2(\Omega_0)},\nonumber
			\\
			&\le \|\dot{B}(0)\|_{L^\infty({\Omega_0})} \|\nabla v_0\|_{L^2(\Omega_0)}  \|\nabla v_1\|_{L^2(\Omega_0)},\nonumber
		\end{aligned}\\			
		&\begin{aligned}
			\bullet\,\left|\int_{0}^{t}\langle \ddot{B}(s) \nabla v^m(s), \nabla V^m(s) \rangle_{L^2(\Omega_0)}\d s\right| \leq\frac 12 \|\ddot{B}\|_{L^\infty((0,T)\times {\Omega_0})} \left(CT+\int_{0}^{t} \|\nabla V^m(s)\|^2_{L^2(\Omega_0)}\d s\right), \nonumber
		\end{aligned}
			\\
			&\begin{aligned}
			\bullet\,	\left|\int_{0}^{t}\langle \dot{B}(s) \nabla V^m(s), \nabla V^m(s) \rangle_{L^2(\Omega_0)}\d s\right| \leq \|\dot{B}\|_{L^{\infty}((0,T)\times {\Omega_0})} \int_{0}^{t} \|\nabla V^m(s)\|^2_{L^2(\Omega_0)}\d s.\nonumber
			\end{aligned}
		\end{align}
		As for the term $J_3$, by the symmetry of the matrix $B$ and integrating by parts in time, we have
		\begin{align*}
			&\int_{0}^{t}\langle B(s) \nabla V^m(s), \nabla \dot{V}^m(s) \rangle_{L^2(\Omega_0)}\d s\\
			=&  \frac{1}{2} \langle B(t) \nabla V^m(t), \nabla V^m(t) \rangle_{L^2(\Omega_0)} - \frac{1}{2} \langle B(0) \nabla V^m(0), \nabla V^m(0) \rangle_{L^2(\Omega_0)}\\
			&-\frac{1}{2}\int_{0}^{t}\langle \dot{B}(s) \nabla V^m(s), \nabla V^m(s) \rangle_{L^2(\Omega_0)}\d s.
		\end{align*}
		Moreover, by the ellipticity of $B$ and again by \eqref{moreregB}, we deduce that 
		\begin{align*}
			&\bullet\,\frac{1}{2} \langle B(t) \nabla V^m(t), \nabla V^m(t) \rangle_{L^2(\Omega_0)} \geq \frac{c_B}{2} \norm{\nabla V^m(t)}^2_{L^2(\Omega_0)},
			\\
			&\bullet\,\frac{1}{2} \left|\langle B(0) \nabla V^m(0), \nabla V^m(0) \rangle_{L^2(\Omega_0)}\right| \leq \frac 12\|B(0)\|_{L^\infty(\Omega_0)}\|\nabla V^m(0)\|^2_{L^2(\Omega_0)}\le\|B(0)\|_{L^\infty(\Omega_0)}\|\nabla v_1\|^2_{L^2(\Omega_0)},
			\\
			&\bullet\, \left|\frac{1}{2}\int_{0}^{t}\langle \dot{B}(s) \nabla V^m(s), \nabla V^m(s) \rangle_{L^2(\Omega_0)}\d s\right| \leq \frac 12
			\|\dot{B}\|_{L^\infty((0,T)\times\Omega_0)}
			\int_{0}^{t} \norm{\nabla V^m(s)}^2_{L^2(\Omega_0)} \d s.
		\end{align*}
		We now focus on $J_4$ and $J_5$: by using \eqref{morerega} and by the uniform bound for $\nabla v^m$ provided by Proposition \ref{propunifboundvm} we obtain that
		\begin{align*}
			\bullet\,\left|\int_{0}^{t}\langle \dot{a}(s) \cdot \nabla v^m(s), \dot{V}^m(s) \rangle_{L^2(\Omega_0)} \d s\right| &\leq\frac 12 \|\dot{a}\|_{L^\infty((0,T)\times {\Omega_0})} \left(CT+\int_{0}^{t} \|\dot V^m(s)\|^2_{L^2(\Omega_0)}\d s\right),
			\\
			\bullet\,\left|\int_{0}^{t}\langle a(s) \cdot \nabla V^m(s), \dot{V}^m(s) \rangle_{L^2(\Omega_0)}\d s\right| &\leq\frac 12 \|a\|_{L^{\infty}([0,T]\times\Omega_0)} \int_{0}^{t} \!\left( \|\dot{V}^m(s)\|_{L^2(\Omega_0)}^2 + \norm{\nabla V^m(s)}_{L^2(\Omega_0)}^2 \right)\d s.
		\end{align*}
		As for the terms $J_6$ and $J_7$, due to \eqref{moreregb} and by observing that for each $s \in (0,t)$ one has $\dot{V}^m(s) \in H^1_0(\Omega_0)$, we obtain 
		\begin{align*}
			\bullet\,\left|\int_{0}^{t}\langle \dot{b}(s) \cdot\nabla V^m(s), \dot{V}^m(s) \rangle_{L^2(\Omega_0)}\d s\right| \leq &\,\frac 12  \|\dot{b}\|_{L^{\infty}((0,T)\times\Omega_0)} 
			\\
			& \times \int_{0}^{t} \left( \|\dot{V}^m(s)\|_{L^2(\Omega_0)}^2 + \norm{\nabla V^m(s)}_{L^2(\Omega_0)}^2 \right)\d s,
			\\
			\bullet\,\left|\int_{0}^{t}\langle b(s) \cdot\nabla \dot{V}^m(s), \dot{V}^m(s) \rangle_{L^2(\Omega_0)}\d s \right|=&  \left|\int_{0}^{t}\langle b(s), \nabla|\dot{V}^m(s)|^2 \rangle_{L^1(\Omega_0)}\d s \right|
			\\
			=& \left|\int_{0}^{t}\langle \div(b(s)), |\dot{V}^m(s)|^2 \rangle_{L^1(\Omega_0)}\d s \right|
			\\
			\leq &
			\norm{\div(b)}_{L^\infty((0,T)\times\Omega_0)} \int_{0}^{t}\|\dot{V}^m(s)\|_{L^2(\Omega_0)}^2\d s.
		\end{align*}
		Finally for the term $J_8$, by the regularity of $\dot{g}$ and by Young's inequality, we readily deduce that
		\begin{equation*}
			\left|\int_{0}^{t}\langle \dot{g}(s), \dot{V}^m(s) \rangle_{L^2(\Omega_0)}\d s\right| \leq \frac{1}{2} \norm{\dot g}^2_{L^2(0,T;L^2(\Omega_0)} + \frac{1}{2}\int_{0}^{t} \|\dot{V}^m(s)\|_{L^2(\Omega_0)}^2 \d s.
		\end{equation*}
		By the estimates obtained for $J_1,\dots,J_8$, by \eqref{eq:claim} and by a suitable choice of $\varepsilon$ in \eqref{epsilon}, we deduce that there exist positive constants $c_1,c_2$ and $c$ such that
		\begin{equation*}
			\|\dot{V}^m(t)\|_{L^2(\Omega_0)}^2 + c \|\nabla V^m(t)\|_{L^2(\Omega_0)}^2 \leq c_1 + c_2 \int_{0}^{t} \left( \|\dot{V}^m(s)\|_{L^2(\Omega_0)}^2 + \|\nabla V^m(s)\|_{L^2(\Omega_0)}^2 \right)\d s.
		\end{equation*}
		By using Gr\"onwall's inequality together with Poincaré inequality, we thus deduce the existence of a constant $C>0$ such that
		\begin{equation*}
			\sup_{0\leq t\leq T} \left(\|\dot{V}^m(t)\|_{L^2(\Omega_0)}^2 +\|V^m(t)\|_{H^1_0(\Omega_0)}^2\right) \leq C,
		\end{equation*}
		which yields, by the definition of $V^m$, to
		\begin{equation}\label{ineq: esssup vm *}
			\sup_{0\leq t\leq T} \left(\|\ddot{v}^m(t)\|_{L^2(\Omega_0)}^2 +\|\dot{v}^m(t)\|_{H^1_0(\Omega_0)}^2\right) \leq C.
		\end{equation}
		
		In order to conclude the proof we have to prove the claim \eqref{eq:claim} and to provide a bound on the $H^2$-norm of $v^m(t)$ uniformly with respect to $t$. 
		Let us fix $t \in [0,T]$. We define by $g^m(t)$ the $L^2$ projection of the function $g(t)$ on the finite dimensional space spanned by $w_1, \dots, w_m$. Then, by \eqref{eq:pdevm} and by using the fact that $\{w_k\}_{k\in \N}$ is an orthonormal basis of $L^2(\Omega_0)$ and an orthogonal basis of $H^1_0(\Omega_0)$, we deduce that
		\begin{equation*}
			\begin{aligned}
				- \langle \div (B(t) \nabla v^m(t)), \varphi \rangle_{L^2(\Omega_0)}
				= \langle g^m(t) - \ddot{v}^m(t) - a(t) \cdot \nabla v^m(t) + 2 b(t) \cdot\nabla \dot{v}^m(t), \varphi \rangle_{L^2(\Omega_0)},
			\end{aligned}
		\end{equation*}
		for every $t\in [0,T]$ and $\varphi \in H^1_0(\Omega_0)$. In particular, for every $t\in [0,T]$ we have
		\begin{equation}\label{eq: pdevm div Delta vm 2}
			- \div (B(t) \nabla v^m(t))
			=  g^m(t) - \ddot{v}^m(t) - a(t) \cdot \nabla v^m(t) + 2 b(t) \cdot\nabla \dot{v}^m(t),
		\end{equation}
		in the sense of distributions. By choosing $t=0$ in \eqref{eq: pdevm div Delta vm 2} and recalling that $\dot{V}^m(0)=\ddot{v}^m(0)$ we can thus estimate
		\begin{align*}
			\|\dot{V}^m(0)\|_{L^2(\Omega_0)}&\le \|g^m(0)\|_{L^2(\Omega_0)}+\|a(0)\|_{L^\infty(\Omega_0)}\|\nabla v^m(0)\|_{L^2(\Omega_0)} +2\|b(0)\|_{L^\infty(\Omega_0)}\|\nabla \dot{v}^m(0)\|_{L^2(\Omega_0)}\\
			&\quad\, +\|\div (B(0) \nabla v^m(0))\|_{L^2(\Omega_0)}\\
			& \le \|g(0)\|_{L^2(\Omega_0)}+\|a(0)\|_{L^\infty(\Omega_0)}\|\nabla v_0\|_{L^2(\Omega_0)} +2\|b(0)\|_{L^\infty(\Omega_0)}\|\nabla v_1\|_{L^2(\Omega_0)}\\
			&\quad\, +\|DB(0)\|_{L^\infty(\Omega_0)}\|\nabla v_0\|_{L^2(\Omega_0)}+\|B(0)\|_{L^\infty(\Omega_0)}\|v^m(0)\|_{H^2(\Omega_0)}.
		\end{align*}
		By arguing as in \cite[Chapter 7.1, Theorem 5]{Evans}, one can deduce that $\|v^m(0)\|_{H^2(\Omega_0)}\le C\|v_0\|_{H^2(\Omega_0)}$ for a suitable constant $C>0$, and so claim \eqref{eq:claim} is proved.
			
		 We now observe that, by Lemma \ref{lemmagrisvard}, there exists a positive constant $\widetilde{D}$ (depending on $\Omega_0$, on the uniform constant of ellipticity of $B$, which is independent on time, and on $\norm{B}_{C^{1,1}([0,T] \times {\overline{\Omega}_0})}$) such that
		\begin{equation}\label{inequ grisvard vm}
			\norm{v^m(t)}_{H^2(\Omega_0)} \leq \widetilde{D} \norm{\div (B(t) \nabla v^m(t))}_{L^2(\Omega_0)},\quad\text{for all }t\in [0,T].
		\end{equation}
		Then, by \eqref{eq: pdevm div Delta vm 2}, \eqref{inequ grisvard vm} and since $\norm{g^m(t)}_{L^2(\Omega_0)} \leq \norm{g(t)}_{L^2(\Omega_0)}$, we deduce that
		\begin{align*}
			&\quad\norm{v^m(t)}_{H^2(\Omega_0)}
			\\
			&\leq \widetilde{D}\left( \norm{g^m(t)}_{L^2(\Omega_0)} + \norm{ \ddot{v}^m(t)}_{L^2(\Omega_0)} +\norm{ a(t) \cdot \nabla v^m(t)}_{L^2(\Omega_0)} + \norm{ 2 b(t) \cdot\nabla \dot{v}^m(t)}_{L^2(\Omega_0)}\right)
			\\
			&\leq \widetilde{D}\left( \norm{g(t)}_{L^2(\Omega_0)} + \norm{ \ddot{v}^m(t)}_{L^2(\Omega_0)} + \|a\|_{L^{\infty}((0,T)\times\Omega_0)} \norm{\nabla v^m(t)}_{L^2(\Omega_0)} \right)
			\\
			& \quad + 2\widetilde{D}
			\left( \|b\|_{L^{\infty}((0,T)\times\Omega_0)} \norm{\nabla \dot{v}^m(t)}_{L^2(\Omega_0)} \right).
		\end{align*}
		We now use \eqref{ineq: esssup vm *}, \eqref{morerega}, \eqref{moreregb}, \eqref{moreregg}, and the uniform bounds provided by Proposition \ref{propunifboundvm} to conclude that
		\begin{equation*}
			\sup_{0\leq t\leq T} \norm{v^m(t)}_{H^2(\Omega_0)} \leq C,
		\end{equation*}
	for some constant $C>0$ and so the statement is proved.
	\end{proof}
	
	As a corollary we now obtain the main result of the section:
	\begin{teo}\label{teoH2regularityv}
	Assume \eqref{coerB}, \eqref{moreregcof}, \eqref{setregular} and \eqref{regulardata}. Then there exists a unique strong-weak solution $v$ of problem \eqref{eq:v} which satisfies:
		\begin{equation}\label{reg}
				\begin{aligned}
				&v \in L^\infty(0,T;H^2(\Omega_0)\cap H^1_0(\Omega_0)),
				\\
				&\dot{v} \in L^\infty(0,T; H^1_0(\Omega_0)) ,
				\\
				&\ddot{v} \in L^\infty(0,T; L^2(\Omega_0)).
			\end{aligned}
		\end{equation}
	\end{teo}
	\begin{proof}
		The estimates provided by Proposition~\ref{teoH2regularityvm} allow us to conclude that there exists a function $v$ fulfilling \eqref{reg} such that, up to a subsequence which we still denote by $v^m$, the following convergences hold:
		\begin{align*}
			& v^{m} \rightharpoonup v \quad\text{weakly in } L^2(0,T;H^2(\Omega_0)\cap H^1_0(\Omega_0))
			,
			\\
			& \dot{v}^{m} \rightharpoonup \dot{v} \quad\text{weakly in } L^2(0,T; H^1_0(\Omega_0)),
			\\
			& \ddot{v}^{m} \rightharpoonup \ddot{v} \quad\text{weakly in } L^2(0,T; L^2(\Omega_0)).
		\end{align*}
	The fact that $v$ is a strong-weak solution follows arguing as in Theorem~\ref{cor}. Uniqueness is instead ensured by Proposition~\ref{prop:unique}.
	\end{proof}
	\begin{rem}\label{rmk:highreg}
		By standard results on interpolation of spaces, property \eqref{reg} easily yields $v\in C^0([0,T];H^1_0(\Omega_0))$ and $\dot{v}\in C^0([0,T];L^2(\Omega_0))$. We actually point out that \eqref{reg} may also be improved by replacing $L^\infty$ with $C^0$ (as in Remark~\ref{rmk:reg}); this can be seen e.g. by exploiting the Semigroup approach (see \cite{DaPraZol} or \cite{CapLucTas}). However, this (slightly) stronger regularity will not be needed for the purposes of the present paper.
	\end{rem}

	\section{Existence, uniqueness and energy balance}\label{sec:application}
	
	In this section we come back to the original problem \eqref{eq:u}, still employing the approach of diffeomorphisms. In this way we are able to improve Theorem~\ref{teo energy balance}, by getting existence and uniqueness of weak solutions and by showing a more precise version of the energy balance (see Theorem~\ref{teo higher regularity u}). To this aim we need a lemma, for which we refer to \cite[Proposition 17.1]{Maggi} for details.
	\begin{lemma}\label{lemma:diff}
		Let $E\subseteq\R^M$ be an open set with Lipschitz boundary and let $\mathfrak{f}\colon \overline{E}\to \R^M$ be a diffeomorphism of class $C^1$ with inverse $\mathfrak{g}\in C^1$. Then $\mathfrak{f}(E)$ is an open set with Lipschitz boundary, and $\partial(\mathfrak{f}(E))=\mathfrak{f}(\partial E)$. Moreover
		\begin{equation}\label{eq:normal}
			\nu_{\mathfrak{f}(E)}(z)=\frac{D\mathfrak{g}(z)^T\nu_E(\mathfrak{g}(z))}{|D\mathfrak{g}(z)^T\nu_E(\mathfrak{g}(z))|},\quad\text{for $\mc H^{M-1}$-a.e. }z\in \partial (\mathfrak{f}(E)),
		\end{equation}
	and
	\begin{equation}\label{eq:normalintegral}
		\int_{\partial\mathfrak{f}(E)}\mathfrak{h}\,\nu_{\mathfrak{f}(E)}\d \mc H^{M-1}=\int_{\partial E}(\mathfrak{h}\circ \mathfrak f)|\det D\mathfrak{f}|(D\mathfrak{g}\circ\mathfrak{f})^T\nu_{E}\d \mc H^{M-1},
	\end{equation}
	for all $\mathfrak{h}\in L^1(\mathfrak{f}(E))$.
	\end{lemma}
	
	\begin{cor}\label{cor:O}
		Assume \eqref{hypsetreg} and let $\Phi,\Psi$ be as in \eqref{assumptionsdiff} and satisfy \ref{H1}. Then the set $\mc O$ introduced in \eqref{eq:O} is open with Lipschitz boundary, and \eqref{bdryO} is satisfied.\par 
		Furthermore one has
		\begin{equation}\label{eq:nuO}
			\nu_{\mc O}(t,x)=(\nu_{\mc O}^t(t,x),\nu_{\mc O}^x(t,x))=\frac{(-\omega(t,x),\nu_{\Omega_t}(x))}{\sqrt{1+\omega(t,x)^2}},\quad\text{for all $t\in [0,T]$ and $\mc H^{N-1}$-a.e. }x\in \partial \Omega_t,
		\end{equation}
		where we introduced the scalar normal velocity
		\begin{equation}\label{eq:omega}
			\omega(t,x):=\dot{\Phi}(t,\Psi(t,x))\cdot\nu_{\Omega_t}(x),\quad \text{ for all $t\in [0,T]$ and $\mc H^{N-1}$-a.e. }x\in \partial \Omega_t.
		\end{equation}
		Moreover the following identity holds true for every $\mathfrak{h}\in L^1(\Gamma)$:
		\begin{equation}\label{eq:integralgamma}
			\int_{\Gamma}\mathfrak{h}\,\nu_{\mc O}\d \mc H^N=\int_{0}^{T}\int_{\partial\Omega_t}\mathfrak{h}(t,x) \begin{pmatrix}
				-\omega(t,x)\\ \nu_{\Omega_t}(x)
			\end{pmatrix}\d \mc H^{N-1}(x)\d t.
		\end{equation}
	\end{cor}
	\begin{rem}
		We point out that actually the scalar normal velocity $\omega$ does not depend on the choice of the diffeomorphisms $\Phi$ and $\Psi$, but it is intrinsically related to the set $\mc O$ (and so to the family $\{\Omega_t\}_{t\in [0,T]}$). Indeed by \eqref{eq:nuO} we have
		\begin{equation}\label{eq:relomega}
			\omega(t,x)=-\frac{\nu_{\mc O}^t(t,x)}{|\nu_{\mc O}^x(t,x)|}, \,\text{ for all $t\in [0,T]$ and $\mc H^{N-1}$-a.e. }x\in \partial \Omega_t.
		\end{equation}
	Moreover, it is immediate to check that the following statements hold:
	\begin{itemize}
		\item if \ref{H2} is in force, then $\|\omega\|_{L^\infty(\Gamma)}<1$;
		\item if \eqref{hypsetmonotone} is in force, then $\omega(t,x)\ge 0$ for all $t\in [0,T]$ and $\mc H^{N-1}$-a.e. $x\in \partial \Omega_t$.
	\end{itemize}
	Indeed, the former fact is a direct consequence of the explicit form \eqref{eq:omega}, while the latter can be inferred by \eqref{eq:relomega} since under \eqref{hypsetmonotone} one has $\nu_{\mc O}^t(t,x)\ge 0$.
	\end{rem}
	\begin{proof}[Proof of Corollary~\ref{cor:O}]
		We introduce the maps $\widetilde{\Phi}\colon [0,T]\times{\overline{\Omega}_0}\to \overline{\mc O}$ and $\widetilde{\Psi}\colon \overline{\mc O}\to [0,T]\times{\overline{\Omega}_0}$ defined as
		\begin{equation*}
			\widetilde{\Phi}(t,y):=(t,\Phi(t,y)),\quad\text{and}\quad \widetilde{\Psi}(t,x):=(t,\Psi(t,x)).
		\end{equation*}
		It is immediate to check that $\overline{\mc O}=\widetilde{\Phi}([0,T]\times{\overline{\Omega}_0})$ and that, due to \ref{H1}, $\widetilde{\Phi}$ is a diffeomorphism of class $C^{1,1}$ with inverse $\widetilde{\Psi}\in C^{1,1}$. By Lemma~\ref{lemma:diff} the statement regarding $\mc O$ is hence verified.\par 
		By using \eqref{eq:normal} with the function $\Phi(t,\cdot)$ we now observe that for all $t\in[0,T]$ there holds
			\begin{equation}\label{eq:nut}
				\nu_{\Omega_t}(x)=\frac{D\Psi(t,x)^T\nu_{\Omega_0}(\Psi(t,x))}{|D\Psi(t,x)^T\nu_{\Omega_0}(\Psi(t,x))|},\quad\text{for $\mc H^{N-1}$-a.e. }x\in \partial \Omega_t.
			\end{equation}
	By using the same formula with $\widetilde{\Phi}$ we instead obtain:
	\begin{equation}\label{eq:nuOcomputation}
		\nu_{\mc O}(t,x)=\frac{D_{(t,x)}\widetilde\Psi(t,x)^T\nu_{(0,T)\times\Omega_0}(t,\Psi(t,x))}{|D_{(t,x)}\widetilde\Psi(t,x)^T\nu_{(0,T)\times\Omega_0}(t,\Psi(t,x))|},\quad\text{for $\mc H^{N}$-a.e. $(t,x)\in \Gamma$}.
	\end{equation}
	By recalling \eqref{eq: dot Psi} we now compute
	\begin{align}\label{eq:matrix}
		D_{(t,x)}\widetilde\Psi(t,x)^T\nu_{(0,T)\times\Omega_0}(t,\Psi(t,x))&=\begin{bmatrix}
			\begin{array}{c|c}
			1 & \dot{\Psi}(t,x)\\
			\hline 0 & D\Psi(t,x)^T
		\end{array}
\end{bmatrix}\begin{pmatrix}
0\\ \nu_{\Omega_0}(\Psi(t,x))
\end{pmatrix}=\begin{pmatrix}
\dot{\Psi}(t,x)\cdot \nu_{\Omega_0}(\Psi(t,x))\nonumber\\ D\Psi(t,x)^T\nu_{\Omega_0}(\Psi(t,x))
\end{pmatrix}\\
&= \begin{pmatrix}
	-\dot{\Phi}(t,\Psi(t,x))\cdot D\Psi(t,x)^T\nu_{\Omega_0}(\Psi(t,x))\\ D\Psi(t,x)^T\nu_{\Omega_0}(\Psi(t,x))
\end{pmatrix}.
\end{align}
Since $D_{(t,x)}\widetilde\Psi(t,x)^T\in\R^{(N+1)\times (N+1)}$, in the first equality above we have gathered its components using
the row vector $\dot{\Psi}(t,x)\in\R^N$ and the matrix $D\Psi(t,x)^T\in\R^{N\times N}$.
By plugging the last equality into \eqref{eq:nuOcomputation} and recalling \eqref{eq:nut} and \eqref{eq:omega} we get \eqref{eq:nuO}.\par 
In order to prove \eqref{eq:integralgamma} we exploit \eqref{eq:normalintegral} for the function $\widetilde{\Phi}$ deducing
\begin{align*}
	\int_{\Gamma}\mathfrak{h}\,\nu_{\mc O}\d \mc H^N&=\int_{(0,T)\times\partial\Omega_0}\mathfrak{h}(\widetilde{\Phi}(t,y))|\det D_{(t,y)}\widetilde{\Phi}(t,y)|D_{(t,x)}\widetilde{\Psi}(\widetilde{\Phi}(t,y))^T\nu_{(0,T)\times\Omega_0}(t,y)\d\mc H^{N}(t,y)\\
	&=\int_{0}^{T}\int_{\partial\Omega_0}\mathfrak{h}(t,\Phi(t,y))\det D\Phi(t,y)\begin{pmatrix}
		-\dot{\Phi}(t,y)\cdot D\Psi(t,\Phi(t,y))^T\nu_{\Omega_0}(y)\\ D\Psi(t,\Phi(t,y))^T\nu_{\Omega_0}(y))
	\end{pmatrix}\d\mc H^{N-1}(y)\d t,
\end{align*}
where we used the fact that $|\det D_{(t,y)}\widetilde{\Phi}(t,y)|=\det D\Phi(t,y)$ and \eqref{eq:matrix} with $x=\Phi(t,y)$. By using again \eqref{eq:normalintegral} with $\Phi(t,\cdot)$ for the integral over $\partial\Omega_0$ we conclude by recalling the formula \eqref{eq:omega} for $\omega$.
	\end{proof}

	Combining the results of Sections~\ref{sec:fixeddomain}, \ref{sec:hyperbolic} and exploiting the previous corollary we can now state the following theorem, which rigorously extends Theorem~\ref{teo energy balance}:
	\begin{teo}\label{teo higher regularity u}
		Assume \eqref{hypsetreg}, \eqref{setregular} and let $\Phi,\Psi$ be as in \eqref{assumptionsdiff} and satisfy \ref{H1'} and \ref{H2}. Let the forcing term $f$ be in $H^1(\mc O)$ and assume the initial data satisfy 
		\begin{equation}\label{eq:regindata}
			u_0 \in H^2(\Omega_0) \cap H^1_0(\Omega_0), \quad \text{ and }\quad  u_1 + \dot{\Phi}(0,\cdot) \cdot \nabla u_0 \in H^1_0(\Omega_0).
		\end{equation}
		Then there exists a unique weak solution $u$ of problem \eqref{eq:u} in the sense of Definition \ref{defweaksolu}, which satisfies 
		\begin{align*}
			&u \in L^\infty(0,T;H^2(\Omega_t)\cap H^1_0(\Omega_t)),
			\\
			&\dot{u} \in L^\infty(0,T; H^1(\Omega_t)),
			\\
			&\ddot{u} \in L^\infty(0,T; L^2(\Omega_t)).
		\end{align*}
		Moreover for every $t \in [0,T]$ the following energy balance holds true:
		\begin{equation}\label{eq:enbaluimproved}
			\begin{aligned}
				&\frac 12\|\dot u(t)\|^2_{L^2(\Omega_t)}+\frac 12\|\nabla u(t)\|^2_{L^2(\Omega_t)}+\int_{0}^{t}\int_{\partial\Omega_s}\frac{\omega(s,x)}{2}(1{-}\omega(s,x)^2)\left(\frac{\partial u}{\partial\nu_{\Omega_s}}(s,x)\right)^2\d\mc H^{N-1}(x)\d s\\
				=& \frac 12\|u_1\|^2_{L^2(\Omega_0)}+\frac 12\|\nabla u_0|^2_{L^2(\Omega_0)}+ \int_{0}^{t}\langle f(s), \dot{u}(s)\rangle_{L^2(\Omega_s)} \d s,
			\end{aligned}
		\end{equation}
	where the scalar normal velocity $\omega$ has been introduced in \eqref{eq:omega}.
	\end{teo}
\begin{rem}\label{rmk:bdry}
	We point out that since $u_0\in H^2(\Omega_0) \cap H^1_0(\Omega_0)$ there must hold $\nabla u_0=\frac{\partial u_0}{\partial \nu_{\Omega_0}}\nu_{\Omega_0}$ on $\partial\Omega_0$. So the second compatibility condition in \eqref{eq:regindata} is actually equivalent to
	\begin{equation*}
			 u_1 \in H^1(\Omega_0)\qquad\quad\text{and}\qquad\quad u_1+\omega(0,\cdot)\frac{\partial u_0}{\partial \nu_{\Omega_0}}=0\quad \text{on }\partial\Omega_0.
	\end{equation*}
	For the same reason we deduce that
		\begin{equation}\label{eq:gradientnormal}
			\nabla u(t,x)=\frac{\partial u}{\partial\nu_{\Omega_t}}(t,x)\nu_{\Omega_t}(x), \quad\text{ for a.e. $t\in [0,T]$ and $\mc H^{N-1}$-a.e. }x\in \partial\Omega_t.
		\end{equation}
\end{rem}

	\begin{rem}\label{rem:A4}
		Under the additional assumptions of Remark~\ref{rem:nonhomogeneous} and \ref{H2A}, the energy balance \eqref{eq:enbaluimproved} now reads as in Remark~\ref{rem:A1} with the boundary term taking the form
		\begin{equation*}
			\int_{0}^{t}\int_{\partial\Omega_s}\frac{\omega(s,x)}{2}\left(|\nu_{\Omega_s}(x)|^2_{A(s,x)}-\omega(s,x)^2\right)\left(\frac{\partial u}{\partial\nu_{\Omega_s}}(s,x)\right)^2\d\mc H^{N-1}(x)\d s.
		\end{equation*}
	\end{rem}

	\begin{proof}[Proof of Theorem~\ref{teo higher regularity u}]
		Existence and uniqueness of the weak solution $u$ satisfying the regularity properties in the statement follow by combining Propositions~\ref{propregularitycoeff}, \ref{propequivweaksolv} and Theorems~\ref{teoequivalenceweaksol}, \ref{teoH2regularityv}. \par 
		To prove \eqref{eq:enbaluimproved} we just need to show that
		\begin{equation*}
			-\int_{\Gamma_t}\frac{\nu_{\mc O}^t}{2}\left[1-\left(\frac{\nu_{\mc O}^t}{|\nu_{\mc O}^x|}\right)^2\right]|\nabla u|^2\d \mc H^N=\int_{0}^{t}\int_{\partial\Omega_s}\frac{\omega(s,x)}{2}(1-\omega(s,x)^2)\left(\frac{\partial u}{\partial\nu_{\Omega_s}}(s,x)\right)^2\d\mc H^{N-1}(x)\d s,
		\end{equation*}
		and then exploit \eqref{eq:enbalu}. The above equality easily follows by applying \eqref{eq:integralgamma} together with \eqref{eq:relomega} and \eqref{eq:gradientnormal}.
	\end{proof}
	
	\subsection{Moving boundary conditions}\label{sec:movingbdrycond}
	We now show how our result can be applied to problems driven by time-dependent boundary conditions, which often appear in mechanical models of debonding. To this aim we ask that the boundary of the set $\Omega_t$ is composed by a fixed part $\Lambda^1$ and a moving one $\Lambda^2_t$. Namely we require that
	\begin{equation}\label{eq:fixedbdry}
		\begin{gathered}
		\text{for all $t\in [0,T]$ there holds } \partial\Omega_t=\Lambda^1\cup\Lambda_t^2,\text{ where $\Lambda^1$ and $\Lambda^2_t$ are $\mc H^{N-1}$-measurable sets,}\\
		\text{with $\Lambda^1$ independent of time $t$, and satisfying $\mc H^{N-1}(\Lambda^1\cap \Lambda^2_t)=0$}.
	\end{gathered}
	\end{equation}
 	Consistently with the previous notation we define
 	\begin{equation*}
 		\Gamma^1:=(0,T)\times \Lambda^1,\quad\text{ and }\quad \Gamma^2:=\bigcup_{t\in(0,T)}\{t\}\times \Lambda^2_t,
 	\end{equation*}
 	so that $\Gamma=\Gamma^1\cup \Gamma^2$. For the sake of clarity, we instead denote by $\nu_{\Lambda^1}$ and $\nu_{\Lambda^2_t}$ the outward unit normal to $\Omega_t$ restricted to $\Lambda^1$ and $\Lambda^2_t$, respectively.\par
 	From a mechanical point of view, it is meaningful to prescribe a time-dependent external loading $W$ on the fixed boundary $\Gamma^1$, while homogeneous Dirichlet boundary conditions are assumed on the moving boundary $\Gamma^2$. We thus consider the problem
	\begin{equation}\label{eq:U}
		\begin{cases}
			\ddot{U}(t,x)-\Delta U(t,x)=0, &(t,x)\in\mc O,\\
			U(t,x)=W(t,x), &(t,x)\in \Gamma^1,\\
			U(t,x)=0, & (t,x)\in \Gamma^2,\\
			U(0,x)=U_0(x),& x \in \Omega_0,\\
			\dot{U}(0,x)=U_1(x), & x \in \Omega_0.
		\end{cases}
	\end{equation}
	The notion of weak solution for problem \eqref{eq:U} is given by Definition~\ref{defweaksolu}, with the obvious changes regarding the boundary conditions. Such problem can be tackled by assuming that $W$ is the trace on $\Gamma^1$ of a regular function, still denoted by $W$, which is everywhere defined and vanishes on $\Gamma^2$ (as stated in Theorem~\ref{teo:U}). Indeed, by considering 
	\begin{equation}\label{eq:uw}
		u(t,x):=U(t,x)-W(t,x),
	\end{equation}
one can resort to Theorem~\ref{teo higher regularity u}.
	\begin{teo}\label{teo:U}
		Assume \eqref{hypsetreg}, \eqref{setregular}, \eqref{eq:fixedbdry} and let $\Phi,\Psi$ be as in \eqref{assumptionsdiff} and satisfy \ref{H1'} and \ref{H2}. Let the external loading $W$ and the initial data satisfy:
		\begin{align*}
			&\bullet\,	W\in H^3(0,T;L^2_{\rm loc}(\R^N))\cap H^2(0,T;H^1_{\rm loc}(\R^N))\cap H^1(0,T;H^2_{\rm loc}(\R^N))\cap L^2(0,T;H^3_{\rm loc}(\R^N)),\\
			&\quad\,\text{such that }\quad W=0\text{ on }\Gamma^2;\\
			&\bullet\, U_0 \in H^2(\Omega_0), \quad \text{ such that }\quad U_0=W(0,\cdot) \text{ on $\Lambda^1$ and }U_0=0 \text{ on $\Lambda^2_0$};\\
			&\bullet\, U_1\in H^1(\Omega_0), \quad \text{ such that }\quad U_1=\dot W(0,\cdot)\text{ on $\Lambda^1$ and } U_1 + \omega(0,\cdot) \frac{\partial U_0}{\partial\nu_{\Lambda^2_0}}=0 \text{ on $\Lambda^2_0$}.
		\end{align*}
		Then there exists a unique weak solution $U$ of problem \eqref{eq:U}, which satisfies:
		\begin{align*}
			&U \in L^\infty(0,T;H^2(\Omega_t)),
			\\
			&\dot{U} \in L^\infty(0,T; H^1(\Omega_t)),
			\\
			&\ddot{U} \in L^\infty(0,T; L^2(\Omega_t)).
		\end{align*}
		Moreover for every $t \in [0,T]$ the following energy balance holds true:
		\begin{equation*}
			\begin{aligned}
				&\frac 12\|\dot U(t)\|^2_{L^2(\Omega_t)}+\frac 12\|\nabla U(t)\|^2_{L^2(\Omega_t)}+\int_{0}^{t}\int_{\Lambda_s^2}\frac{\omega(s,x)}{2}(1{-}\omega(s,x)^2)\left(\frac{\partial U}{\partial\nu_{\Lambda^2_s}}(s,x)\right)^2\d\mc H^{N-1}(x)\d s\\
				= &\frac 12\|U_1\|^2_{L^2(\Omega_0)}+\frac 12\|\nabla U_0|^2_{L^2(\Omega_0)}+ \int_{0}^{t} \int_{\Lambda^1}\dot W(s,x) \frac{\partial U}{\partial\nu_{\Lambda^1}}(s,x) \d\mc H^{N-1}(x)\d s.
			\end{aligned}
		\end{equation*}
	\end{teo}

	\begin{rem}\label{rem:A5}
		Under the additional assumptions of Remark~\ref{rem:nonhomogeneous} and \ref{H2A}, in the energy balance above, besides the usual changes on the potential energy and the presence of the forcing term due to $\dot{A}$, 
the boundary terms related to $\Lambda^1$ and $\Lambda^2_t$ need to be replaced by
		\begin{equation*}
			\int_{0}^{t} \int_{\Lambda^1}\dot W(s,x)|\nu_{\Lambda^1}(x)|^2_{A(s,x)} \frac{\partial U}{\partial\nu_{\Lambda^1}}(s,x) \d\mc H^{N-1}(x)\d s,
		\end{equation*}
		and
		\begin{equation*}
			\int_{0}^{t}\int_{\Lambda_s^2}\frac{\omega(s,x)}{2}\left(|\nu_{\Lambda_s^2}(x)|^2_{A(s,x)}-\omega(s,x)^2\right)\left(\frac{\partial U}{\partial\nu_{\Lambda^2_s}}(s,x)\right)^2\d\mc H^{N-1}(x)\d s.
		\end{equation*}
	\end{rem}

	\begin{proof}[Proof of Theorem~\ref{teo:U}]
		 By considering the function $u$ defined in \eqref{eq:uw} it is easy to see that, from the point of view of weak solutions, problem \eqref{eq:U} is equivalent to problem \eqref{eq:u} with data
		 \begin{equation}\label{eq:dataw}
		 \begin{gathered}
		 	f(t,x)=\Delta W(t,x)-\ddot W(t,x),\\
		 	u_0(x)=U_0(x)-W(0,x),\quad\text{ and }\quad u_1(x)=U_1(x)-\dot W(0,x).
		 \end{gathered}
	 \end{equation}
	 	By the assumptions on $W, U_0, U_1$ we infer that $f\in H^1(\mc O)$ and that \eqref{eq:regindata} holds (see also Remark~\ref{rmk:bdry}), and so Theorem~\ref{teo higher regularity u} yields existence, uniqueness and regularity of the weak solution $U$.\par
		The energy balance instead follows from \eqref{eq:enbaluimproved} by exploiting \eqref{eq:uw} and \eqref{eq:dataw} and after some simple (but tedious) manipulation via integration by parts.
	\end{proof}
	
	\subsection{Examples of moving domains}
	We finally collect some examples of sets $\Omega_t$ satisfying \eqref{hypsetreg} (and sometimes also \eqref{hypsetmonotone}) for which we can explicitly construct diffeomorphisms $\Phi$ and $\Psi$ satisfying \eqref{assumptionsdiff} and the regularity conditions \ref{H1'} and \ref{H2}.
	\subsubsection{1-dimensional setting}\label{ex:1d}
	 Let $\ell\in C^{2,1}([0,T])$ satisfy
	\begin{equation*}
		\ell(t)>0,\quad\text{and}\quad |\dot{\ell}(t)|<1,\quad\text{for all }t\in [0,T],
	\end{equation*}
	and consider the sets $\Omega_t:=(0,\ell(t))$. By defining
	\begin{equation*}
		\Phi(t,y):=\frac{\ell(t)}{\ell(0)}y,\qquad\Psi(t,x):=\frac{\ell(0)}{\ell(t)}x,
	\end{equation*}
	it is elementary to verify the validity of all the assumptions of Theorem~\ref{teo higher regularity u} and also of \eqref{eq:fixedbdry}. If in addition $\ell$ is nondecreasing, then also \eqref{hypsetmonotone} is satisfied. In this situation there hold
	\begin{equation}\label{eq:omega1d}
		\omega(t,0)=0\quad\text{ and }\quad\omega(t,\ell(t))=\dot\ell(t),\quad\text{for every }t\in [0,T].
	\end{equation}
	This setting has been analysed in \cite{DMLazNar,LazNar, LazNarQuas, LazNarInitiation, Riv, RivQuas, RivNar}, where $\ell$ is just required to be Lipschitz. Their argument strongly relies on an explicit representation of solutions of the wave equation provided by d'Alembert's formula, which holds true only in dimension one.
	
	\subsubsection{Homothetic transformations}\label{ex:homo}
	Let $\Omega_0\subseteq\R^N$ satisfy \eqref{eq:Omega0}, \eqref{setregular} and let $\lambda\in C^{2,1}([0,T])$ satisfy $\lambda(0)=1$ and
	\begin{equation*}
		\lambda(t)>0,\quad\text{and}\quad |\dot{\lambda}(t)|\max\limits_{y\in {\overline{\Omega}_0}}|y|<1,\quad\text{for all }t\in [0,T].
	\end{equation*}
Consider the sets
\begin{equation*}
	\Omega_t:=\lambda(t)\Omega_0=\{x\in \R^N:\, x=\lambda(t)y\text{ for some }y\in \Omega_0\},
\end{equation*}
and the diffeomorphisms
\begin{equation*}
	\Phi(t,y):=\lambda(t)y,\qquad\text{and}\qquad \Psi(t,x):=\frac{x}{\lambda(t)}.
\end{equation*}
It is again immediate to check all the assumptions of Theorem~\ref{teo higher regularity u}. If moreover $\lambda$ is nondecreasing and $\Omega_0$ is positively balanced, meaning that
\begin{equation}\label{eq:balanced}
	\varepsilon \Omega_0\subseteq \Omega_0,\quad\text{for every }\varepsilon\in (0,1),
\end{equation}
then also \eqref{hypsetmonotone} is satisfied. Notice that \eqref{eq:balanced} is related to the position of $\Omega_0$ with respect to the origin. It is for instance fulfilled by star-shaped sets at the origin. \par
In this setting we have
\begin{equation*}
	\omega(t,x)=\dot\lambda(t)\frac{x}{\lambda(t)}\cdot\nu_{\Omega_0}\left(\frac{x}{\lambda(t)}\right), \quad\text{for all }t\in [0,T] \text{ and for }\mc H^{N-1}\text{-a.e. }x\in \partial\Omega_t=\lambda(t)\partial\Omega_0.
\end{equation*}
In the particular case $\Omega_0=B_1(0)$ we infer $\omega(t,x)=\dot{\lambda}(t)$ and we get a radial symmetry similar to the one studied in \cite{LazMolSol}. \par 
If instead $\Omega_0$ is a tetrahedron of the form
\begin{equation}\label{220620}
	\Omega_0=\{y\in \R^N:\, y\cdot n<1\text{ and }y_i>0\text{ for all }i=1,\dots,N\},
\end{equation}
where $n\in\R^N$ is a unit vector with positive components, then 
\begin{equation*}
	\omega(t,x)=\begin{cases}
		0,&\text{if }x_i=0\text{ for some }i=1,\dots,N,\\
		\dot{\lambda}(t), &\text{if }x\cdot n=\lambda(t),
	\end{cases}\qquad\text{for every }t\in [0,T].
\end{equation*}
We may also treat the case of an octahedron 
\begin{equation*}
	\Omega_0=\{y\in \R^N:\, y\cdot n_i<1,\text{ for }i=1,\dots,2^N\},
\end{equation*}
where $n_i\in \R^N$ is a unit vector belonging to the $i-$th orthant for all $i=1,\dots,2^N$; in this situation there holds $\omega(t,x)=\dot\lambda(t)$.

\subsubsection{Sublevel sets}\label{ex:sublevels}
Let $g\in C^0(\R^N)\cap C^{3,1}(\R^N\setminus\{g=0\})$ be a nonnegative function satisfying
\begin{equation*}
	\nabla g(x)\neq 0,\quad\text{ for every }x\in \R^N\setminus\{g=0\},
\end{equation*}
and assume there exists $R>0$ such that the sublevel set $\{g<R\}$ is bounded; then consider a nondecreasing function $\rho\in C^{2,1}([0,T])$ such that
\begin{equation*}
	0<\rho(0)\le\rho(T)<R.
\end{equation*}
 We now define the sets
\begin{equation*}
	\Omega_t:=\{R-\rho(t)<g<R\},
\end{equation*}
which fulfill assumptions \eqref{hypset}, \eqref{setregular} and also \eqref{eq:fixedbdry}. An analogous choice, with the obvious changes, is $\Omega_t=\{r<g<\rho(t)\}$ with $r\in (0,\rho(0))$.\par 
We thus need to build the diffeomorphisms $\Phi$, $\Psi$ satisfying \eqref{assumptionsdiff}, \ref{H1'} and \ref{H2}. To this aim let us consider the vector field $X\colon [0,T]\times (\R^N\setminus\{g=0\})\to \R^N$ defined by
\begin{equation*}
	X(t,x):=\frac{\dot\rho(t)}{\rho(t)}(g(x)-R)\frac{\nabla g(x)}{|\nabla g(x)|^2},
\end{equation*}
and notice that $X$ is of class $C^{1,1}([0,T];C^{2,1}(\R^N\setminus\{g<\varepsilon\}))$ for all $\varepsilon>0$. The map $\Phi$ is now defined as follows: $\Phi(t,y)$ is the evolution at time $t$ of the point $y\in {\overline{\Omega}_0}$ through the flow
\begin{equation*}
	\begin{cases}
		\dot{\Phi}(s,y)=X(s,\Phi(s,y)),&s\in (0,T),\\
		\Phi(0,y)=y.
	\end{cases}
\end{equation*} 
By standard results on ODEs, for all $y\in {\overline{\Omega}_0}$ there exists a time $T_y\in (0,T]$ such that there exists a unique local solution $\Phi(\cdot, y)\in C^{2,1}([0,T_y])$ of the Cauchy problem above. We now show that actually the solution is global (i.e. $T_y=T$) by proving that $\Phi(t,\Omega_0)=\Omega_t$ for all $t\in [0,T]$. Indeed this ensures that the flow is always well contained in the region where $g$ does not vanish, and hence where the vector field $X$ does not blow up.\par 
To this aim we notice that for all $t\in [0,T_y]$ we have
\begin{align*}
	g(\Phi(t,y))=g(y)+\int_{0}^{t}\nabla g(\Phi(s,y))\cdot\dot{\Phi}(s,y)\d s=g(y)+\int_{0}^{t}\frac{\dot\rho(s)}{\rho(s)}(g(\Phi(s,y))-R)\d s.
\end{align*}
By setting $F_y(t):=g(\Phi(t,y))-R$ and by differentiating the above equality we get
\begin{equation}\label{eq:system}
	\begin{cases}\displaystyle
		\dot F_y(t)=\frac{\dot\rho(t)}{\rho(t)}F_y(t),& t\in [0,T_y],\\
		F_y(0)=g(y)-R.
	\end{cases}
\end{equation}
The only solution to \eqref{eq:system} is given by $F_y(t)=\frac{\rho(t)}{\rho(0)}F_y(0)$, which finally implies
\begin{equation}\label{eq:levels}
	g(\Phi(t,y))-R=\frac{\rho(t)}{\rho(0)}(g(y)-R),\quad\text{for every }t\in [0,T_y].
\end{equation}
This implies that $\Phi(t,\cdot)$ maps level sets of $g$ in level sets of $g$ and hence that $\Phi(t,\Omega_0)= \Omega_t$; recalling that $R>\rho(T)$, we now infer that $\Phi$ is well-defined on the whole $[0,T]\times {\overline{\Omega}_0}$ and it belongs to $C^{2,1}([0,T];C^{2,1}({\overline{\Omega}_0}))$.\par 
We can now introduce the function $\Psi\colon \overline{\mc O}\to [0,T]\times {\overline{\Omega}_0}$ as the \lq\lq space\rq\rq-inverse of $\Phi$, namely $\Psi(t,x):=[\Phi(t,\cdot)]^{-1}(x)$, so that \eqref{assumptionsdiff} is fulfilled by construction. By the Inverse Function Theorem and the Implicit Function Theorem, from the regularity of $\Phi$ we also deduce \ref{H1'}.\par 
To finally have \ref{H2} we need to require in addition that
\begin{equation}\label{eq:addi}
	\dot{\rho}(t) < \min\limits_{x\in {\overline{\Omega}_t}}|\nabla g(x)|,\quad\text{ for every }t\in[0,T] ;
\end{equation}
indeed, by also exploiting \eqref{eq:levels}, it yields
\begin{equation*}
	|\dot{\Phi}(t,y)|=\frac{\dot{\rho}(t)}{\rho(t)}\frac{(R-g(\Phi(t,y)))}{|\nabla g(\Phi(t,y))|}=\frac{\dot{\rho}(t)}{|\nabla g(\Phi(t,y))|}\frac{R-g(y)}{\rho(0)}\le \frac{\dot{\rho}(t)}{|\nabla g(\Phi(t,y))|}<1.
\end{equation*}
The scalar normal velocity can be finally computed as follows:
\begin{equation*}
	\omega(t,x)=X(t,x)\cdot \nu_{\Omega_t}(x)=\begin{cases}\displaystyle
		\frac{\dot{\rho}(t)}{|\nabla g(x)|},&\text{if }g(x)=R-\rho(t),\\
			0,&\text{if }g(x)=R,
	\end{cases}\qquad\text{for every }t\in [0,T].
\end{equation*}

With the particular choice $g(x)=|x|$ the sets $\Omega_t$ are annuli, so we recover the radial case already mentioned above. Since now $|\nabla g(x)|=1$, condition \eqref{eq:addi} reads $\dot{\rho}(t)<1$ and the scalar normal velocity is
\begin{equation}\label{eq:omegarad}
\omega(t,x)=\begin{cases}\displaystyle
	\dot{\rho}(t),&\text{if }|x|=R-\rho(t),\\
	0,&\text{if }|x|=R,
\end{cases}\qquad\text{for every }t\in [0,T].
\end{equation}
The radial case was analysed in \cite{LazMolSol} with slightly lower regularity assumptions, by reducing the problem to a one-dimensional one, as in example \ref{ex:1d}.

\bigskip

The situations described in examples \ref{ex:homo} and \ref{ex:sublevels} above refer to debonding models where one assumes to know a priori the possible shapes of the debonding front and of the debonded set. The unknown to be determined is the evolution law providing the time when a certain debonded front is reached. This is analogous to models with prescribed crack in fracture mechanics.

A simple concrete application is the following: 
in dimension two, the possible debonded regions are triangles
\[ 
	\Omega_t=\{x\in \R^2:\, x\cdot n<\lambda(t)\text{ and }x_i>0\text{ for }i=1,2\},
\]
with $n\in\R^N$ a unit vector with positive components and
$\lambda\in C^{2,1}([0,T];[1,+\infty))$ nondecreasing; cf.\ \eqref{220620}.
Boundary conditions are given e.g.\ on the segments 
$\{x=(x_1,0):\, 0<x\cdot n<\lambda(0)/2\}$
and
$\{x=(0,x_2):\, 0<x\cdot n<\lambda(0)/2\}$.
(This represents an external load pulling the film from a region close to the vertex $(0,0)$.)  
Such example can be treated with both methods shown in examples \ref{ex:homo} and \ref{ex:sublevels}.
As in the one-dimensional case of example \ref{ex:1d}, the possible debonding fronts are parallel lines (modelling a material that can be detached only in a certain direction). However, this example is genuinely two-dimensional, since the debonding front is not parallel to the edges where the boundary condition is imposed, which results in nontrivial reflections in the propagation of waves in the debonded region.

	\section{Application to dynamic debonding}\label{sec:debonding}
	In this last section we propose a proper formulation of a dynamic debonding model. Differently from the previous sections, in this setting also the evolution of the sets $t\mapsto\Omega_t$ is unknown, and it has to be recovered by means of energetic considerations which involve the solution $u$ of problem \eqref{eq:u} in an implicit and complex way. Here we rigorously define the dynamic energy release rate in a general framework, i.e. without any ansatz on the shape of the domains. This allows to state the energetic principle governing the evolution, called (dynamic) Griffith criterion. \par 
We assume that the energy needed to debond a portion of film parametrized on a (measurable) set $E\subseteq\R^N$ is given by
	\begin{equation}\label{eq:energydissipated}
		\int_{E}\kappa(x)\d x,
	\end{equation}
	where $\kappa\in C^0(\R^N)$ is a positive function, representing the \emph{toughness} of the glue between the film and the substrate.\par
	The next lemma shows how the integral in \eqref{eq:energydissipated} varies when the domain of integration depends on time and its evolution is given.
	\begin{lemma}
		Assume \eqref{hypsetreg} and let $\Phi,\Psi$ be as in \eqref{assumptionsdiff} and satisfy \ref{H1}. Given any function $\kappa\in C^0(\R^N)$, for every $t\in [0,T]$ there holds
		\begin{subequations}
		\begin{equation}\label{eq:kappa}
			\int_{\Omega_t}\kappa(x)\d x=\int_{\Omega_0}\kappa(x)\d x+\int_{0}^{t}\int_{\partial\Omega_s}\omega(s,x)\kappa(x)\d\mc H^{N-1}(x)\d s.
		\end{equation}
		In particular we have
		\begin{equation}\label{eq:measure}
			|\Omega_t|=|\Omega_0|+\int_{0}^{t}\int_{\partial\Omega_s}\omega(s,x)\d\mc H^{N-1}(x)\d s.
		\end{equation}
\end{subequations}
	\end{lemma}
	\begin{proof}
		We recall that by Corollary~\ref{cor:O} the set $\mc O$ is open, with Lipschitz boundary, and satisfies \eqref{bdryO}. Setting $\mc O_t:=\{(s,x)\in \mc O:\, s\in (0,t)\}$, by an application of the fundamental theorem of calculus in $\R^{N+1}$ we thus deduce that
		\begin{align*}
			0&=\int_{\mc O_t}\frac{\partial}{\partial s}\kappa(x)\d s\d x=\int_{\partial O_t}\kappa\,\nu_{\mc O}^t\d \mc H^N\\
			&=\int_{\Omega_t}\kappa(x)\d x-\int_{\Omega_0}\kappa(x)\d x+\int_{\Gamma_t}\kappa\,\nu_{\mc O}^t\d \mc H^N.
		\end{align*}
		We now conclude by using \eqref{eq:integralgamma}.
	\end{proof}
	
	\subsection{Dynamic energy release rate, maximum dissipation principle and Griffith criterion}
Let now $u$ be the weak solution found in Theorem~\ref{teo higher regularity u} for a given nondecreasing family $\Omega_t$ (i.e. also satisfying \eqref{hypsetmonotone}). 
The \emph{dynamic energy release rate} \cite{Fre90} is the opposite of the (infinitesimal) energy variation due to the change in time of the domain, without accounting for the energy variation due to the evolution of the external forces.
Recalling the energy balance \eqref{eq:enbaluimproved}, we thus consider the internal energy (kinetic and potential) subtracted with the work of external forces:
	\begin{equation}\label{eq:E}
		\mc E(t):= \frac 12\|\dot u(t)\|^2_{L^2(\Omega_t)}+\frac 12\|\nabla u(t)\|^2_{L^2(\Omega_t)}-\int_{0}^{t}\langle f(s), \dot{u}(s)\rangle_{L^2(\Omega_s)} \d s.
	\end{equation}
We now provide the definition of dynamic energy release rate.
	\begin{defin}
		For $t\in [0,T]$, we define the \emph{dynamic energy release rate} of the debonding model as
		\begin{equation*}
			\mc G(t):=\lim\limits_{h \to 0^+}-\frac{\mc E(t+h)-\mc E(t)}{|\Omega_{t+h}\setminus\Omega_t|},
		\end{equation*}
		whenever such limit exists.
	\end{defin}
	Due to the energy balance \eqref{eq:enbaluimproved} together with \eqref{eq:measure}, by means of \eqref{eq:E} we thus infer that the dynamic energy release rate can be computed as follows:
	\begin{equation*}
		\mc G(t)=-\frac{\dot{\mc E}(t)}{\frac{\d}{\d t}|\Omega_{(\cdot)}|(t)}=\frac{\displaystyle \int_{\partial\Omega_t}\!\!\! \frac{\omega(t,x)}{2}(1-\omega(t,x)^2)\left(\frac{\partial u}{\partial\nu_{\Omega_t}}(t,x)\right)^2\!\!\!\!\d\mc H^{N-1}(x)}{\displaystyle \int_{\partial\Omega_t} \!\!\!\omega(t,x)\d\mc H^{N-1}(x)},\text{ if }\int_{\partial\Omega_t} \!\!\!\omega(t,x)\d\mc H^{N-1}(x)>0.
	\end{equation*}
	In the one-dimensional setting and in the radial case (see Sections~\ref{ex:1d} and~\ref{ex:sublevels}), analyzed in \cite{DMLazNar, RivNar} and \cite{LazMolSol} respectively, the integrands in the latter formula actually do not depend on $x$ due to the symmetry of the problem; in other words, in those cases the released energy is the same at each point of the boundary $\partial\Omega_t$. In contrast, in the general situation here depicted the released energy may be different from point to point. It is thus convenient to introduce the \emph{density} of the dynamic energy release rate, which is obtained by localizing the above formula around a point $x\in \partial\Omega_t$ as in the following definition.
\begin{defin}
		Given $t\in [0,T]$ and $x\in \partial\Omega_t$ for which $\alpha:=\omega(t,x)>0$, the \emph{dynamic energy release rate density} at the point $(t,x)$ with speed $\alpha\in (0,1)$ is defined by
		\begin{equation}\label{eq:dynenrel}
			G_\alpha(t,x):=\lim\limits_{r\to 0^+}\frac{\displaystyle \int_{\partial\Omega_t\cap B_r(x)} \frac{\omega(t)}{2}(1-\omega(t)^2)\left(\frac{\partial u}{\partial\nu_{\Omega_t}}(t)\right)^2\d\mc H^{N-1}}{\displaystyle \int_{\partial\Omega_t\cap B_r(x)} \omega(t)\d\mc H^{N-1}}=\frac 12 (1-\alpha^2)\left(\frac{\partial u}{\partial\nu_{\Omega_t}}(t,x)\right)^2.
		\end{equation}
	If $\alpha=0$, the dynamic energy release rate density is extended by continuity, setting:
	\begin{equation*}
		G_0(t,x):=\frac 12 \left(\frac{\partial u}{\partial\nu_{\Omega_t}}(t,x)\right)^2.
	\end{equation*}
	\end{defin}
\begin{rem} We now provide the explicit expression of the dynamic energy release rate in the one-dimensional setting and in the radial one, recalling that in those cases the dynamic energy release rate coincides with its density and recovering the formulas obtained in \cite{DMLazNar, RivNar, LazMolSol}. In the 1-d situation, by means of \eqref{eq:omega1d}, it is easy to check that
	\begin{equation*}
		\mc G(t)=G_{\dot{\ell}(t)}(t,\ell(t))=\frac 12 (1-\dot{\ell}(t)^2)u_x(t,\ell(t))^2,\quad \text{for a.e. }t\in [0,T].
	\end{equation*}
	In the radial case the weak solution $u(t,\cdot)$ has radial symmetry \cite{LazMolSol}, hence the dynamic energy release rate density $G_\alpha(t,\cdot)$ is radial as well. Thus, by considering $u^{\rm rad}(t,r):=u(t,x)$ and $G^{\rm rad}_\alpha(t,r):=G_\alpha(t,x)$ for $r=R-|x|$ and by using \eqref{eq:omegarad}, we obtain
	\begin{equation*}
		\mc G(t)=G^{\rm rad}_{\dot{\rho}(t)}(t,\rho(t))=\frac 12 (1-\dot{\rho}(t)^2)u^{\rm rad}_r(t,\rho(t))^2,\quad \text{for a.e. }t\in [0,T].
	\end{equation*}
\end{rem}
\begin{rem}
	We notice that the dynamic energy release rate density can be written in an equivalent way by using the relation
	\begin{equation*}
		\dot{u}(t,x)+\omega(t,x)\frac{\partial u}{\partial\nu_{\Omega_t}}(t,x)=0, \quad\text{ for a.e. $t\in [0,T]$ and $\mc H^{N-1}$-a.e. }x\in \partial\Omega_t,
	\end{equation*}
which follows since $u\equiv 0$ on $\Gamma$. Indeed, from the above equality we deduce
\begin{equation}\label{eq:equivDEER}
	\begin{aligned}
	G_{\omega(t,x)}(t,x)&=\frac 12 (1-\omega(t,x)^2)\left(\frac{\partial u}{\partial\nu_{\Omega_t}}(t,x)\right)^2=\frac 12\frac{1-\omega(t,x)}{1+\omega(t,x)}\left[(1+\omega(t,x))\frac{\partial u}{\partial\nu_{\Omega_t}}(t,x)\right]^2\\
	&=\frac 12\frac{1-\omega(t,x)}{1+\omega(t,x)}\left[\frac{\partial u}{\partial\nu_{\Omega_t}}(t,x)-\dot{u}(t,x)\right]^2.
\end{aligned}
\end{equation}
This will be used in Proposition~\ref{prop:Griff}.
\end{rem}
	Given a positive toughness $\kappa\in C^0(\R^N)$, we now postulate that during the evolution process the following energy balance is satisfied:
	\begin{equation}\label{eq:enbalkappa}		
			\mc E(t)+\int_{\Omega_t\setminus\Omega_0}\kappa(x)\d x= \mc E(0),\qquad\text{ for every }t\in [0,T].		
	\end{equation}
	By comparing \eqref{eq:enbalkappa}, \eqref{eq:enbaluimproved}, \eqref{eq:kappa} and \eqref{eq:dynenrel} we observe that the energy is conserved if one requires
	\begin{equation*}
		\omega(t,x)\kappa(x)=\omega(t,x)G_{\omega(t,x)}(t,x),\quad\text{for a.e. $t\in [0,T]$ and for $\mc H^{N-1}$-a.e. $x\in \partial\Omega_t$}.
	\end{equation*}
	However, the above condition is not sufficient to determine a proper evolution of the sets $\Omega_t$, indeed $\omega\equiv 0$ (i.e. $\Omega_t\equiv\Omega_0$) is always an admissible choice.\par 
A stronger requirement is the following \emph{local maximum dissipation principle}, which essentially says that $\Omega_t$ grows whenever it is possible, while preserving the energy balance:
	\begin{equation}\label{eq:mdp}
		\begin{gathered}
				\omega(t,x)=\max\{\alpha\in[0,1):\, \alpha\,\kappa(x)=\alpha G_\alpha(t,x)\},\\
				\text{for a.e. $t\in [0,T]$ and for $\mc H^{N-1}$-a.e. $x\in \partial\Omega_t$}.
		\end{gathered}
	\end{equation}
We refer to \cite{DMLazNar, LazMolSol, RivNar} for a discussion.
The next proposition states two equivalent forms of the local maximum dissipation principle:
the first one is the (local) dynamic Griffith criterion, the second one consists of two equivalent equations for the scalar normal velocity, involving the normal derivative of the displacement $u$. Note that the condition $\omega<1$ in \eqref{eq:Griffith} corresponds to the physical requirement that the speed of growth of the domain is subsonic.
	\begin{prop}\label{prop:Griff}
		Let $\kappa\in C^0(\R^N)$ be positive. Assume \eqref{hypsetmonotone} and the hypotheses of Theorem~\ref{teo higher regularity u}, and let $u$ be the unique weak solution of problem \eqref{eq:u}. Then, the following three conditions are equivalent:
		\begin{itemize}
\item the local maximum dissipation principle \eqref{eq:mdp} holds true;
			\item the \emph{local dynamic Griffith criterion} holds true, namely
			\begin{equation}\label{eq:Griffith}
				\begin{cases}
					0\le\omega(t,x)<1,\\
					G_{\omega(t,x)}(t,x)\le \kappa(x),\\
					\omega(t,x)\left[G_{\omega(t,x)}(t,x)- \kappa(x)\right]=0,
				\end{cases}
			\text{for a.e. $t\in [0,T]$ and for $\mc H^{N-1}$-a.e. $x\in \partial\Omega_t$};
			\end{equation}
			\item for a.e. $t\in [0,T]$ and for $\mc H^{N-1}$-a.e. $x\in \partial\Omega_t$ the scalar normal velocity $\omega$ is given by
			\begin{subequations}\label{eq:omexplicit}
			\begin{equation}\label{eq:omexplicita}
				\omega(t,x)=\begin{cases}
					\displaystyle\sqrt{1-\frac{2\kappa(x)}{\left(\frac{\partial u}{\partial\nu_{\Omega_t}}(t,x)\right)^2}},&\text{if }\left(\frac{\partial u}{\partial\nu_{\Omega_t}}(t,x)\right)^2> 2\kappa(x),\\
					0, &\text{otherwise},
				\end{cases}
	\end{equation}
or equivalently
	\begin{equation}
			\label{eq:omexplicitb} \omega(t,x)=\max\left\{\frac{\left[\frac{\partial u}{\partial\nu_{\Omega_t}}(t,x)-\dot{u}(t,x)\right]^2-2\kappa(x)}{\left[\frac{\partial u}{\partial\nu_{\Omega_t}}(t,x)-\dot{u}(t,x)\right]^2+2\kappa(x)},\ 0\right\}.
			\end{equation}		
	\end{subequations}
		\end{itemize}
	\end{prop}
\begin{proof}
	Let us first fix a pair $(t,x)$. We then observe that the set 
	\begin{equation*}
		A(t,x):= \{\alpha\in[0,1):\, \alpha\kappa(x)=\alpha G_\alpha(t,x)\},
	\end{equation*}
	appearing in \eqref{eq:mdp} consists of at most two elements; indeed by recalling \eqref{eq:dynenrel} and since $\kappa(x)>0$ it is easy to check that
	\begin{alignat*}{3}
		A(t,x)&=
			\left\{0, \sqrt{1-\frac{2\kappa(x)}{\left(\frac{\partial u}{\partial\nu_{\Omega_t}}(t,x)\right)^2}}\right\}, \quad && \text{if }\left(\frac{\partial u}{\partial\nu_{\Omega_t}}(t,x)\right)^2> 2\kappa(x),\\
		A(t,x)&=	\{0\}, &&\text{otherwise}.
	\end{alignat*}
	Hence the equivalence between \eqref{eq:mdp} and \eqref{eq:omexplicita} is proved.  Analogously, \eqref{eq:mdp} and \eqref{eq:omexplicitb} turns out to be equivalent by employing \eqref{eq:equivDEER}. The equivalence between \eqref{eq:omexplicita} and \eqref{eq:Griffith} is straightforward by exploiting the explicit form \eqref{eq:dynenrel}, so we conclude.
\end{proof}

	\begin{rem}\label{rem:A6} 
		We conclude this section by listing  the changes that must be taken into account in the case of the hyperbolic equation \eqref{eq:uA}. 		The energy $\mc E$ in \eqref{eq:E} now takes the form
		\begin{equation*}
			\begin{aligned}
				\mc E(t)=&\quad\, \frac 12\|\dot u(t)\|^2_{L^2(\Omega_t)}+\frac 12\langle A(t)\nabla u(t), \nabla u(t)\rangle_{L^2(\Omega_t)}-\int_{0}^{t}\langle f(s), \dot{u}(s)\rangle_{L^2(\Omega_s)} \d s\\
				&-\frac 12\int_{0}^{t}\langle\dot{A}(s)\nabla u(s),\nabla u(s)\rangle_{L^2(\Omega_s)}\d s.
			\end{aligned}
		\end{equation*}
	As a consequence the dynamic energy release rate density becomes
		\begin{equation*}
		G_\alpha(t,x)=\frac 12 (|\nu_{\Omega_t}(x)|^2_{A(t,x)}-\alpha^2)\left(\frac{\partial u}{\partial\nu_{\Omega_t}}(t,x)\right)^2,
	\end{equation*}
	and so
	\begin{equation*}
		\begin{aligned}
			G_{\omega(t,x)}(t,x)&=\frac 12 (|\nu_{\Omega_t}(x)|^2_{A(t,x)}-\omega(t,x)^2)\left(\frac{\partial u}{\partial\nu_{\Omega_t}}(t,x)\right)^2\\
			&=\frac 12\frac{|\nu_{\Omega_t}(x)|_{A(t,x)}-\omega(t,x)}{|\nu_{\Omega_t}(x)|_{A(t,x)}+\omega(t,x)}\left[|\nu_{\Omega_t}(x)|_{A(t,x)}\frac{\partial u}{\partial\nu_{\Omega_t}}(t,x)-\dot{u}(t,x)\right]^2.
		\end{aligned}
	\end{equation*}
This implies that equations \eqref{eq:omexplicit} has to be rewritten as
\begin{equation*}
	\begin{aligned}
	\omega(t,x)&=\begin{cases}
		\displaystyle\sqrt{|\nu_{\Omega_t}(x)|_{A(t,x)}^2-\frac{2\kappa(x)}{\left(\frac{\partial u}{\partial\nu_{\Omega_t}}(t,x)\right)^2}},&\text{if }\left(|\nu_{\Omega_t}(x)|_{A(t,x)}\frac{\partial u}{\partial\nu_{\Omega_t}}(t,x)\right)^2> 2\kappa(x),\\
		0, &\text{otherwise},
	\end{cases}\\
&=\max\left\{|\nu_{\Omega_t}(x)|_{A(t,x)}\frac{\left[|\nu_{\Omega_t}(x)|_{A(t,x)}\frac{\partial u}{\partial\nu_{\Omega_t}}(t,x)-\dot{u}(t,x)\right]^2-2\kappa(x)}{\left[|\nu_{\Omega_t}(x)|_{A(t,x)}\frac{\partial u}{\partial\nu_{\Omega_t}}(t,x)-\dot{u}(t,x)\right]^2+2\kappa(x)},\ 0\right\},
\end{aligned}
\end{equation*}
and also that the first line in the local dynamic Griffith criterion becomes
\begin{equation*}
	0\le\omega(t,x)<|\nu_{\Omega_t}(x)|_{A(t,x)}.
\end{equation*}	The resulting changes in Definition~\ref{def:coupled} below are straightforward.
		\end{rem}

	\subsection{Formulation of the coupled problem}
We are now in the position to provide a proper formulation of a dynamic debonding model, by combining the wave equation \eqref{eq:u} with the local maximum dissipation principle \eqref{eq:mdp} (or, equivalently, with the local dynamic Griffith criterion \eqref{eq:Griffith}, or with \eqref{eq:omexplicit}). 
We point out that the resulting system features a strong coupling: indeed, the evolution of the domain of the wave equation is governed by \eqref{eq:omexplicit},
which in turn depends on the solution $u$ to the wave equation itself.\par 
	A solution to the dynamic debonding model is defined as follows.
	\begin{defin}\label{def:coupled}
		Given the following data:
		\begin{itemize}
			\item $\Omega_0\subseteq\R^N$ satisfying \eqref{eq:Omega0} and \eqref{setregular},
			\item $\kappa\in C^0(\R^N)$ satisfying $\kappa(x)>0$ for all $x\in \R^N$,
			\item $f\in H^1(0,T;L^2_{\rm loc}(\R^N))\cap L^2(0,T;H^1_{\rm loc}(\R^N))$,
			\item $u_0\in H^2(\Omega_0)\cap H^1_0(\Omega_0)$ and $u_1\in H^1(\Omega_0)$ satisfying 
			\begin{equation}\label{eq:compatibility}
				\begin{aligned}
				\text{either}\quad&u_1(x)=0\quad\text{and}\quad\left(\frac{\partial u_0}{\partial\nu_{\Omega_0}}(x)\right)^2\le 2\kappa(x),\\
				\text{or}\quad&u_1(x)\neq0,\quad \left(\frac{\partial u_0}{\partial\nu_{\Omega_0}}(x)\right)^2-u_1(x)^2=2\kappa(x)\quad \text{and}\quad\frac{\frac{\partial u_0}{\partial\nu_{\Omega_0}}(x)}{u_1(x)}<-1,
			\end{aligned}
			\end{equation}
		for $\mc H^{N-1}$-a.e. $x\in \partial\Omega_0$,
		\end{itemize}
	we say that an evolution $[0,T]\ni t\mapsto (u(t),\Omega_t)$ is a weak solution of the coupled problem \eqref{eq:u}\&\eqref{eq:mdp} if the following conditions are satisfied:
	\begin{enumerate}
		\item there exists a map $\Phi\colon [0,T]\times{\overline{\Omega}_0}\to \R^N$ with \lq\lq space-inverse\rq\rq{} $\Psi(t,\cdot)$ satisfying \eqref{assumptionsdiff}, \ref{H1'} and \ref{H2}, for which
		\begin{equation*}
			\Omega_t=\Phi(t,\Omega_0),\quad\text{for every }t\in [0,T];
		\end{equation*}
	\item $u$ is the weak solution to problem \eqref{eq:u} with forcing term $f$ and initial data $u_0$, $u_1$;
	\item the local maximum dissipation principle \eqref{eq:mdp} is satisfied, or equivalently the scalar normal velocity $	\omega(t,x)=\dot\Phi(t,\Psi(t,x))\cdot\nu_{\Omega_t}(x)$ fulfils one of the two (equivalent) equations in \eqref{eq:omexplicit} for a.e. $t\in [0,T]$ and for $\mc H^{N-1}$-a.e. $x\in \partial\Omega_t$.
	\end{enumerate}
	\end{defin}
	\begin{rem}
		We point out that the definition makes sense. Indeed condition (1) ensures the regularity \eqref{hypsetreg} of the family $\{\Omega_t\}_{t\in [0,T]}$. Furthermore, by simple computations one can check that the compatibility condition \eqref{eq:compatibility} is actually equivalent to 
		\begin{equation*}
			u_1+\omega(0,\cdot)\frac{\partial u_0}{\partial \nu_{\Omega_0}}=0,\quad \text{ on }\partial\Omega_0,
		\end{equation*}
	where $\omega(0,\cdot)$ is defined by \eqref{eq:omexplicit}. Hence, by Remark~\ref{rmk:bdry} one can apply Theorem~\ref{teo higher regularity u},
concluding that the wave equation \eqref{eq:u} has a unique weak solution $u$, whose regularity allows one to give a meaning to the normal derivative $\frac{\partial u}{\partial\nu_{\Omega_t}}$ at the boundary, appearing in \eqref{eq:omexplicit}. Finally, notice that \eqref{eq:omexplicit} also implies the monotonicity property \eqref{hypsetmonotone}.
	\end{rem}
	\begin{rem}
		Definition~\ref{def:coupled} can be adapted to the case of moving boundary conditions described in Subsection~\ref{sec:movingbdrycond}, with minor modifications. In this setting $\Omega_0$ also satisfies \eqref{eq:fixedbdry} (at $t=0$), and the external loading fulfils $W\equiv 0$ in a neighborhood of $(0,T)\times \Lambda^2_0$. The compatibility conditions on $U_0$ (and on $U_1$ on $\Lambda^1$) are the same of Theorem~\ref{teo:U}, while \eqref{eq:compatibility} has to be valid for $U_0$ and $U_1$ on $\Lambda^2_0$. Finally, (3) is prescribed only on $\Lambda^2_t$, while on $\Lambda^1$ there must hold $\omega(t,x)\equiv 0$.
	\end{rem}

We conclude the paper by showing how Definition~\ref{def:coupled} covers the particular cases of the 1-dimensional and radial models already analysed in \cite{DMLazNar,RivNar} and \cite{LazMolSol}, respectively (see also Sections~\ref{ex:1d} and~\ref{ex:sublevels}). In fact, in those papers the notion of solution to the coupled problem is given in a slightly different form, and the existence is obtained by exploiting d'Alembert's formula. We prove that, if the initial data are well-prepared, the solution found in the above mentioned works fulfils Definition~\ref{def:coupled}, at least for short times.

\begin{teo}\label{Thm coupled prob dim one}
	Let $\ell_0 >0$ and let $\kappa\in C^{1,1}_{\rm loc}([\ell_0,+\infty))$ satisfy $\kappa(x)>0$ for all $x\in [\ell_0,+\infty)$.
	Assume that
	$f\in C^{0,1}([0,T]\times [0,+\infty))$, $u_0\in C^{2,1}([0,\ell_0])$ and $u_1\in C^{1,1}([0,\ell_0])$ satisfy
	\begin{subequations}\label{eq:compdim1}
		\begin{equation}\label{eq:compatibility in dim one-1}
			u_0(0) = 0,\quad u_0(\ell_0) = 0 \quad\text{and}\quad u_1(0) = 0,
		\end{equation}
		and 
		\begin{equation}\label{eq:compatibility in dim one-2}
			u_1(\ell_0)\neq0,\quad {u}'_0(\ell_0)^2 - u_1(\ell_0)^2=2\kappa(\ell_0),\quad\frac{{u}'_0(\ell_0)}{u_1(\ell_0)}<-1.
		\end{equation}
	\end{subequations}
	Then, there exist $T^\ast \in (0,T]$ and a unique weak solution $t\mapsto(u(t),(0,\ell(t)))$ to the coupled problem \eqref{eq:u}\&\eqref{eq:mdp} in $[0,T^*]$ in the sense of Definition \ref{def:coupled}.
\end{teo}

\begin{proof}
	By \cite[Theorem 4.6]{RivNar}, there exists a unique pair $(u,\ell)$ such that:
	\begin{enumerate}
		\item[(i)] $u$ is a weak solution to the wave equation with forcing term $f$ and initial data $u_0$ and $u_1$ in the moving domain $\bigcup_{t\in(0,T)}\{t\}\times(0,\ell(t))$;
		\item [(ii)] $\ell(0)=\ell_0$ and in a right neighborhood of $0$, the function $\ell$ is a solution of the following ODE: 
		\begin{equation}\label{eq:ell}
			\dot{\ell}(t) = \max\left\{\frac{\left[{u}'_0(\ell(t)-t)-u_1(\ell(t)-t) - \int_{0}^{t} f(\tau,\tau - t + \ell(t)) \,d\tau \right]^2-2\kappa(\ell(t))}{\left[{u}'_0(\ell(t)-t)-u_1(\ell(t)-t) - \int_{0}^{t} f(\tau,\tau - t + \ell(t)) \,d\tau\right]^2+2\kappa(\ell(t))} ,\ 0\right\}.
		\end{equation} 
	\end{enumerate}
	We observe that the equation solved by $\ell$ is the analogue of \eqref{eq:omexplicitb} in the 1-d setting (see \eqref{eq:omega1d}), by means of d'Alembert's formula. Hence, conditions $(2)$ and $(3)$ of Definition~\ref{def:coupled} are satisfied by the function $u$ and the sets $(0,\ell(t))$. To conclude, we need to check also the validity of $(1)$.\par
	To this aim, we notice that \eqref{eq:compatibility in dim one-2} implies
	\begin{equation*}
		\left[u'_0(\ell_0)-u_1(\ell_0)\right]^2>2\kappa(\ell_0).
	\end{equation*}
	As a consequence, from \eqref{eq:ell} one obtains that $\dot{\ell}(0)> 0$ and so,
	by continuity, there exists $T^\ast \in (0,T]$ such that $\ell$ solves
	\begin{equation*}
		\dot{\ell}(t) = \frac{\left[{u}'_0(\ell(t)-t)-u_1(\ell(t)-t) - \int_{0}^{t} f(\tau,\tau - t + \ell(t)) \,d\tau \right]^2-2\kappa(\ell(t))}{\left[{u}'_0(\ell(t)-t)-u_1(\ell(t)-t) - \int_{0}^{t} f(\tau,\tau - t + \ell(t)) \,d\tau\right]^2+2\kappa(\ell(t))},\quad \text{in } [0,T^\ast].
	\end{equation*}
	
	In particular, as pointed out in \cite[Remarks 4.9 and 4.12]{RivNar}, by a classical bootstrap argument the regularity assumptions on $f, u_0$ and $u_1$ and the compatibility conditions \eqref{eq:compdim1} ensure that $\ell \in C^{2,1} ([0,T^*])$. Moreover, \eqref{eq:ell} directly yields $0\le \dot{\ell}(t)<1$ for all $t\in [0,T^*]$. Hence, the construction presented in Subsection \ref{ex:1d} provides the existence of the diffeomorphisms required in $(1)$ and we conclude.
\end{proof}
	The following result deals with the radial case in dimension 2; the extension to arbitrary dimension is straightforward. In order to state it we introduce the following notation. Given a ball or an annulus $A\subseteq\R^2$ we denote by $C^{k,\alpha}_{\rm rad}(\overline A)$ the space of functions $h\in C^{k,\alpha}(\overline A)$ which are radial, meaning that there exists a function $h^{\rm rad}\colon \R\to \R$ such that $h(x)=h^{\rm rad}(|x|)$ for all $x\in \overline{A}$.
\begin{teo}\label{Thm coupled prob radial}
	Let $R>\rho_0 >0$ and let $\kappa\in C^{1,1}_{\rm rad}(\overline{B_R(0)})$ satisfy $\kappa(x)>0$ for all $x\in \overline{B_R(0)}$. Setting $\Omega_0:= \{x \in \R^2: R-\rho_0 < |x| < R \}$, assume that $f\in C^{0,1}([0,T]; C^{0,1}_{\rm rad}(\overline{B_R(0)}))$, $u_0\in C^{2,1}_{\rm rad}(\overline{\Omega_0})$ and $u_1\in C^{1,1}_{\rm rad}(\overline{\Omega_0})$ satisfy
	\begin{subequations}
	\begin{equation*}\label{eq:compatibility in dim two-1}
		\begin{split}
		&u_0(x) = 0\quad \text{ if } |x| = R \text{ or } |x|=R-\rho_0,
		\\
		&u_1(x) = 0 \quad\text{ if } |x| = R,
		\end{split}
	\end{equation*}
	and 
	\begin{equation*}\label{eq:compatibility in dim two-2}
		u_1(x)\neq0,\quad \left(\frac{\partial u_0}{\partial\nu_{\Omega_0}}(x)\right)^2-u_1(x)^2=2\kappa(x)\quad \text{and}\quad\frac{\frac{\partial u_0}{\partial\nu_{\Omega_0}}(x)}{u_1(x)}<-1\quad \text{ if } |x|= R-\rho_0.
	\end{equation*}
\end{subequations}
	Then, there exist $T^\ast \in (0,T]$ and a unique weak solution $t\mapsto(u(t),\Omega_t)$ to the coupled problem \eqref{eq:u}\&\eqref{eq:mdp} in $[0,T^*]$ in the sense of Definition \ref{def:coupled}, where 
	\begin{equation*}
	\Omega_t:= \{x \in \R^2: R-\rho(t) < |x| < R \},
	\end{equation*}
	for a suitable $\rho\in C^{2,1}([0,T^*])$.
\end{teo}

\begin{proof}
	The proof is analogous to the one of Theorem \ref{Thm coupled prob dim one}, taking into account Subsection~\ref{ex:sublevels} and \eqref{eq:omegarad}. Here, the existence of the function $\rho$ is guaranteed by \cite[Theorem 3.6]{LazMolSol}. For the regularity of $\rho$ we instead refer to \cite[Remarks 3.7 and 3.8]{LazMolSol}. 
\end{proof}

	We finally stress once again that the well posedness of Definition~\ref{def:coupled} in the general case seems to be a difficult task, due to the strong coupling between the wave equation and the rule \eqref{eq:omexplicit} governing the evolution of the domains. We leave the problem open for future research.

	\bigskip
	
	\noindent\textbf{Acknowledgements.}
The authors would like to thank Maicol Caponi for fruitful discussions.
This work is part of the Project {\em Variational methods for stationary and evolution problems with singularities and interfaces} (PRIN 2017) funded by the Italian Ministry of Education, University, and Research.
The authors have been supported by the {\em Gruppo Nazionale per l'Analisi Matematica la Probabilità e le loro Applicazioni} (GNAMPA) of the {\em Istituto Nazionale di Alta Matematica} (INdAM). 
	
	\bigskip
	
	\bibliographystyle{siam}

\end{document}